\documentclass[amssymb,amsfonts,reqno]{amsart}

\usepackage{amssymb,amsmath,amsthm,verbatim,IEEEtrantools, mathtools}

\usepackage{comment}

\usepackage{latexsym}

\usepackage{graphicx}

\def\be#1{\begin{equation}\label{#1}}
\def\bas{\begin{align*}}
\def\eas{\end{align*}}
\def\bi{\begin{itemize}}
\def\ei{\end{itemize}}

\usepackage{pdfpages}

\theoremstyle{plain}
   \newtheorem{theorem}[subsection]{Theorem}

   \newtheorem{proposition}[subsection]{Proposition}
   
   \newtheorem{lemma}[subsection]{Lemma}
   \newtheorem{corollary}[subsection]{Corollary}

\begin{document}

\author{James Wright}
\address{Maxwell Institute of Mathematical Sciences and the School of Mathematics, University of
Edinburgh, JCMB, The King's Buildings, Peter Guthrie Tait Road, Edinburgh, EH9 3FD, Scotland}
\email{j.r.wright@ed.ac.uk}


%\thanks{The second author was supported in part by an EPSRC grant.}

%\date{14 August, 2010}

\title[A theory of complex oscillatory integrals]{A theory of complex oscillatory integrals: \\ A case study}

\maketitle

\begin{abstract}
In this paper we develop a theory for oscillatory integrals with complex phases. When
$f:{\mathbb C}^n \to {\mathbb C}$, we evaluate this phase function on the basic
character ${\rm e}(z) := e^{2\pi i x} e^{2\pi i y}$ of ${\mathbb C} \simeq {\mathbb R}^2$
(here $z = x+iy \in {\mathbb C}$ or $z = (x,y) \in {\mathbb R}^2$) and consider
oscillatory integrals of the form
$$
I \ = \ \int_{{\mathbb C}^n} {\rm e}(f({\underline{z}})) \, \phi({\underline{z}}) \, d{\underline{z}} 
$$
where $\phi \in C^{\infty}_c({\mathbb C}^n)$.
%These are {\it real} oscillatory integrals since the phase
%$\Phi = {\rm Re} f + {\rm Im} f$ (being the sum of the real and imaginary parts of $f$) is real. 

Unfortunately basic scale-invariant bounds for the oscillatory integrals $I$ do not hold in
the generality that they do in the real setting. Our main effort is to develop a perspective and arguments
to locate scale-invariant bounds in (necessarily) less generality than we are accustomed to in the real setting.
\end{abstract}

\section{Introduction}\label{introduction}

Oscillatory integral estimates known as van der Corput estimates are very useful
in a wide range of areas.
It states that for every integer $k\ge 1$, there is a constant $C_k$ such that whenever $f : [a,b] \to {\mathbb R}$
is a smooth function satisfying $|f^{(k)}|\ge 1$ on $[a,b]$, then
\begin{equation}\label{vc-real}
\bigl| \int_a^b e^{2\pi i \lambda f(x)} \, dx \bigr| \ \le \ C_k \ |\lambda|^{-1/k}.
\end{equation}
Here $\lambda \in {\mathbb R}$ is a real parameter. When $k=1$ a monotonicity condition is needed on $f'$.
The example $f(x) = x^k$ shows the sharpness of the exponent $1/k$. 

The usefulness
of these estimates lies in the uniformity of the constant $C_k$ which depends only on $k$. As a consequence
the quantitative hypothesis $|f^{(k)}(x)|\ge 1$ for $x \in [a,b]$ can be relaxed to $|f^{(k)}(x)|\ge \mu$
for $x \in [a,b]$ and the estimate {\it scales} accordingly. Applying \eqref{vc-real} to $g = \mu^{-1} f$
(so that $|g^{(k)}|\ge 1$ holds on $[a,b]$), we have
\begin{equation}\label{vc-real-scaled}
\bigl| \int_a^b e^{2\pi i \lambda f(x)} \, dx \bigr| \ \le \ C_k \ |\mu \lambda|^{-1/k},
\end{equation}
a conclusion we would not be able to deduce if the constant $C_k$ in \eqref{vc-real} depended on the phase $f$.
For lack of better terminology, we will call the bound in \eqref{vc-real} or \eqref{vc-real-scaled}
a {\it scale-invariant} bound. See \cite{Stein-beijing}
or \cite{Stein-big} for further details.

The situation changes when we move from the real field ${\mathbb R}$ to the complex field ${\mathbb C}$
and consider complex differentiable phases $f : D \to {\mathbb C}$ where $D$
is some domain in ${\mathbb C}^n$. Here we consider oscillatory integrals of the form
\begin{equation}\label{complex-osc-form}
I_{\phi}(f) \ = \ \int_{{\mathbb C}^n} {\rm e}(f({\underline{z}})) \, \phi({\underline{z}}) \, d{\underline{z}}
\end{equation}
where $\phi \in C^{\infty}_c({\mathbb C}^n)$ is a smooth cut-off function and ${\rm e}(z) = e^{2\pi i x}e^{2\pi i y}$
with $z = x + iy \in {\mathbb C}$ is the basic character on the locally compact abelian group ${\mathbb C} \simeq {\mathbb R}^2$.

The study of these oscillatory integrals is not
to be confused with the beautiful theory of complex oscillatory integrals defined by evaluating
the above complex phase $f$ on the basic character $t \to e^{2\pi i t}$ from ${\mathbb R}$ which
has a unique analytic extension from ${\mathbb R}$ to ${\mathbb C}$ and is very useful in developing
a geometric-invariant theory for real oscillatory integrals. See \cite{AGV}.

The oscillatory integral $I_{\phi}(f)$ is a {\it real} oscillatory integral with the real phase ${\rm Re}(f) + {\rm Im}(f)$.
It would be nice if we had complex versions of van der Corput estimates where scale-invariant bounds for $I_{\phi}(f)$
are derived from a condition that some complex derivative of $f$ is bounded below. We would then be able to apply
such estimates to the fourier transform ${\widehat{\sigma}}$ of measures $\sigma$ in ${\mathbb C}^n$
since we can write
$$
{\widehat{\sigma}}(T {\underline{w}}) \ = \ \int_{{\mathbb C}^n} {\rm e}(\langle {\underline{w}}, {\underline{z}} \rangle ) \,
d \sigma ({\underline{z}})
$$
for some sympletic (and hence measure-preserving) transformation $T$ on ${\mathbb C}^n$. Here
$\langle {\underline{w}}, {\underline{z}} \rangle = w_1 z_1 + \cdots + w_n z_n$. Therefore questions
regarding $L^p$ norms of ${\widehat{\sigma}}$ (which arise in the fourier restriction problem for example) can then
be investigated using complex van der Corput estimates. See Section \ref{Fourier-transform} for details.

Unfortunately the complex analogue of \eqref{vc-real} does not hold. When $n = 1$, suppose that $f$ is a complex differentiable
function satisfying $|f^{(k)}| \ge 1$ on the unit disc ${\mathbb D} = \{z \in {\mathbb C}: |z| \le 1\}$, say. Then a scale-invariant
bound
$$
\bigl| \int_{{\mathbb C}} {\rm e}(\lambda f(z)) \, \phi(z) \, dz \bigr| \ \le \ C_{k,\phi} \, |\lambda|^{-2/k} 
$$
with a constant $C_{k, \phi}$ only depending on $k$ and $\phi \in C^{\infty}_c({\mathbb D})$ does not hold.
Considering the example $f(z) = z^k$ shows the exponent $2/k$ is optimal. 

The examples illustrating
this lack of scale-invariance are non-polynomial. For polynomials phases, we are able to adapt and extend arguments
from elementary number theory to establish scale-invariant bounds for oscillatory integrals $I_{\phi}(f)$ in \eqref{complex-osc-form}.

Fix $\phi \in C^{\infty}_c({\mathbb C}^n)$ and define 
$$
H_{f,\phi} \ = \ \inf_{{\underline{z}} \in {\rm supp}(\phi)} H_f({\underline{z}}) \ \ {\rm where} \ \ 
H_f({\underline{z}}) \ = \ \max_{|\alpha|\ge 1} \bigl(|\partial^{\alpha} f({\underline{z}})/\alpha!|^{1/|\alpha|} \bigr) .
$$
for polynomials $f \in {\mathbb C}[X_1, \ldots, X_n]$.
Notation involving complex polynomials and partial derivatives will be defined in Section \ref{prelims}.

\begin{theorem}\label{H-main-osc}
If $f \in {\mathbb C}[X_1, \ldots, X_n]$ has degree $d$, then
$$
|I_{\phi}(f)| \ \le \ C_{d,n,\phi} \ H_{f,\phi}^{-2}.
$$
Here $C_{d,n,\phi}$ depends only on $d,n$ and $\phi$.
\end{theorem}

The real version for polynomials $f \in {\mathbb R}[X_1, \ldots, X_n]$ can be found in \cite{ACK}. It is possible
to use the real version (afterall our phase is ${\rm Re}(f) + {\rm Im}(f) \in {\mathbb R}[X_1, \ldots, X_{2n}]$)
but we would only obtain the bound $I_{\phi}(f)| \lesssim_{d,n,\phi} H_f^{-1}$ which has limited use in applications.

As an immediate consequence of Theorem \ref{H-main-osc}, we have the following scale-invariant version of \eqref{vc-real}
for complex polynomials.

\begin{corollary}\label{complex-vc-many}
Suppose $f \in {\mathbb C}[X_1, \ldots, X_n]$ has degree $d$ and satisfies
$|\partial^{\alpha}f| \ge \mu$ on the support of $\phi \in C^{\infty}_c({\mathbb C}^n)$ for some partial derivative 
$\partial^{\alpha}$. Then 
$$
|I_{\phi}(\lambda f)| \ \le \ C_{d,n,\phi} \, |\mu\lambda|^{-2/|\alpha|}.
$$
\end{corollary}

The bound in Theorem \ref{H-main-osc} has a number of applications. For example a complex version
of a robust oscillatory integral estimate with polynomial phases $f \in {\mathbb R}[X]$ due to Phong and Stein \cite{PS-I}
can be established. Let $f \in {\mathbb C}[X]$ be a complex polynomial of degree $d$ and consider the
derivative $f'(z) = a \prod_{j=1}^L (z - w_j)^{m_j}$ where $\{w_j\}$ are the distinct roots of $f'$. 

\begin{proposition}\label{PS-intro} We have
$$
|I_{\phi}(f)| \ \le \ C_{d,\phi} \, \max_{1\le j \le L} \ \min_{{\mathcal C}\ni w_j} \, 
\Bigl[\frac{1}{|a \prod_{w_k \notin {\mathcal C}} (w_j - w_k)^{m_k}|}\Bigr]^{2/(S({\mathcal C}) + 1)}
$$
where the minimum is taken over all {\it root clusters} ${\mathcal C} \subseteq \{w_k\}$ containing $w_j$ and
$S({\mathcal C}) = \sum_{w_k \in {\mathcal C}} m_k$.
\end{proposition}

This opens the door to establish complex versions of results based on the Phong-Stein bound.

Another application involves the fourier extension operator for the complex moment curve $z \to (z, z^2, \ldots, z^d)$
which we can view as a 2-surface in ${\mathbb R}^{2d}$. The oscillatory integral operator
$$
{\mathcal E} b ({\underline{w}}) \ = \ \int_{{\mathbb C}} {\rm e}(w_1 z + \cdots + w_d z^d) \, b(z) \, \phi(z) \, dz
$$
differs from the fourier extension operator of the 2-surface defined by the complex moment curve by a sympletic transformation.

\begin{proposition}\label{ACK-intro} Set $q_d = 0.5(d^2 + d) + 1$. Then ${\mathcal E}1 \in L^q ({\mathbb C}^d)$
if and only if $q>q_d$.

If ${\mathcal E}b ({\underline{w}})$ is defined with respect to the sparse polynomial $w_1 z^{k_1} + \cdots + w_d z^{k_d}$
where $K := k_1 + \cdots + k_d < (0.5) k_d (k_d +1)$, then ${\mathcal E} 1 \in L^q({\mathbb C}^d)$ if and only if
$q > K$.
\end{proposition}

The real analogue of Proposition \ref{ACK-intro} is due to Arkhipov, Chubarikov and Karatsuba; see \cite{ACK-1} and \cite{ACK-2}.

\subsection*{Structure of the paper} In the following section we illustrate the lack of scale-invariant bounds for general complex functions. In Sections 3 and 4, we motivate and develop the theory of sublevel sets for complex differentiable functions
and in Section \ref{H-functional} we introduce the $H$ functional in the complex setting and illustrate its usefulness
in the theory of sublevel sets. In Sections 6 and 7, we develop the theory for oscillatory integrals with complex 
polynomial phases. 
In Sections 8-11, we give the proofs of the main results for sublevel sets and oscillatory integrals (Theorem \ref{H-main-osc}), reducing matters
to a structural sublevel set statement which we establish in Sections 12 and 13. In Section 14, we give
the proof of Proposition \ref{PS-intro} and in Sections 15-19, we give the proof of Proposition \ref{ACK-intro}. 

\subsection*{Notation}
We use the notation $A \lesssim B$ between two positive quantities $A$ and $B$ to denote $A \le CB$ for some
constant $C$. We sometimes use the notation $A \lesssim_k B$ to emphasise that the
implicit constant depends on the parameter $k$. We sometimes use $A = O(B)$ to denote the
inequality $A \lesssim B$.
Furthermore, we use $A \ll B$ to denote $A \le \delta B$ for a sufficiently small constant $\delta>0$
whose smallness will depend on the context.

\subsection*{Acknowledgement}
We thank Rob Fraser, John Green and Jonathan Hickman for enlightening conversations on the topics of this paper.

\section{An illustration}\label{illustration}

Let us begin with a simple illustration. Suppose we have a function $f$ with a large derivative,
say $|f'|\ge 1$ everywhere in some region $I$. Therefore $f$ is not stationary and we expect
the sublevel set $S_{\epsilon} = \{ z \in I : |f(z)| \le \epsilon \}$ to be small when $\epsilon$ is small. Consider two points
$z_0, z_1 \in S$ so that $f$ is small at these two points but we know that $f'$ is large at all
points $z_t := z_0 + t(z_1 - z_0), \, t\in [0,1]$ on the line segment from $z_0$ to $z_1$. Of course
these two bits of information can be connected by the fundamental theorem of calculus:
$$
\int_0^1 f'(z_t) \, dt \ (z_1 - z_0) \ = \ f(z_1) \ - \ f(z_0).
$$
Taking absolute values, we have
\begin{equation}\label{triangle}
\Bigl| \int_0^1 f'(z_t) \, dt \Bigr| \, \bigl|z_1 - z_0\bigr| \ = \ |f(z_1) - f(z_0)| \ \le \ 2 \epsilon
\end{equation}
by a simple use of the triangle inequality and 
since $z_0, z_1 \in S_{\epsilon}$ . But we know $|f'(z_t)| \ge 1$ for
all $0\le t \le 1$ and so it seems we are one step away from deducing that the diameter of $S_{\epsilon}$ is
at most $2 \epsilon$.

{\it But} we have not specified what world we are living in. The above discussion makes sense for real functions
$f$ whose derivative is always larger than 1 but it also makes sense for complex functions whose complex
derivative $f'$ is everywhere large in absolute value. With a little imagination, the above discussion makes
sense for a wide range of functions defined over disparate fields with an absolute value $|\cdot|$.

The point of this
illustration is that the real world is a very nice world to live in because the underlying field ${\mathbb R}$ is not only a complete field, 
but it is also an ordered field which gives rise
to the {\it intermediate value theorem} from elementary calculus. This simple result can pack a powerful punch
at times. 

So for the time being, let us suppose that we are living in the real world looking at a real
function $f$ on the real line with a large derivative. Then by the intermediate value theorem, we conclude that either $f'\ge 0$ is always
nonnegative or $f'\le 0$ is always nonpositive and hence 
$$
\int_0^1 |f'(z_t)| \, dt \ = \ \bigl| \int_0^1 f'(z_t) \, dt \bigr|
$$
so that we can move the abolute value sign inside the integral for free! And so indeed we are just one step away
from concluding that ${\rm diam} (S_{\epsilon}) \le 2 \epsilon$ since $|f'(z_t)| \ge 1$ for all $t \in [0,1]$. The
inequality ${\rm diam} (S_{\epsilon}) \le 2 \epsilon$ is not only a nice structural statement for the sublevel set $S_{\epsilon}$
(it implies in particular the measure bound $|S_{\epsilon}| \le 2 \epsilon$), it is {\it scale-invariant} as described in the
Introduction;  namely, we can relax
the condition that $|f'| \ge 1$. For general $f$, set $\mu = \inf_I |f'|$ and scale $g = \mu^{-1} f$ (if $\mu > 0$),
noting $|g'| \ge 1$ on $I$ and applying the above diameter bound to $g$, we see that ${\rm diam} (S_{\epsilon}) \le 2 \epsilon/\mu$.
This inequality remains true if $\mu = 0$. Such a general statement is not possible if we only knew that
${\rm diam} (S_{\epsilon}) \le C \epsilon$ where $C$ depends on $f$. 

Scale-invariant inequalities are very powerful. For example the scale-invariant measure bound 
$|S_{\epsilon}| \le 2 \epsilon$ whenever $|f'| \ge 1$ on $I$ {\it almost} implies by itself that the scale-invariant
bound $|S_{\epsilon}| \le C_k \epsilon^{1/k}$ holds whenever $|f^{(k)}| \ge 1$ on $I$. The standard induction on $k$
argument needs one additional {\it a priori} structural statement for $S_{\epsilon} = \{z \in I: |f(z)| \le \epsilon\}$
when $f^{(k)}$ does not vanish on $I$; namely, that $S_{\epsilon}$ is the union of at most $k$ intervals. This is yet
another consequence of the intermediate value theorem or the order structure of ${\mathbb R}$. The
same story holds for oscillatory integrals with a real phase $f$. One proves a scale-invariant bound when
$|f'| \ge 1$ everywhere and then uses this bound to prove a bound whenever $|f^{(k)}|\ge 1$. These are
the van der Corput estimates we mentioned in the Introduction. All this
from the order structure of the real field ${\mathbb R}$. See \cite{Stein-big} for more details.

Now let us move from the real world to the complex world. The inequality \eqref{triangle} still holds for
complex functions whose complex derivative satisfies $|f'| \ge 1$ on $I$. How far are we from concluding
that ${\rm diam} (S_{\epsilon})  \le 2 \epsilon$ without the use of the intermediate value theorem coming from the order
structure of the reals? 

The problem here is that the real and imaginary parts of $f$ can conspire to produce many
zeros of $f$ in our region $I$ while $f$ still retains the property $|f'|\ge 1$ everywhere on $I$. Consider the function
$$
f(z) \ = \ \frac{e^N (e^{Nz} - 1)}{N}
$$
for some large $N\ge 1$. We have 
$$
|f'(z)| \ = \ e^{N({\rm Re}(z) + 1)} \ \ge \ 1 \ \ {\rm for \ all} \ \  
z \in I \ := \ \{ z \in {\mathbb C} : {\rm Re}(z) \ge -1 \}.
$$ 
On the other hand, $f(z) = 0$ for infinitely
many $z$ on the line ${\rm Re}(z) = 0$; precisely for $z = x + iy$ with $x=0$ and $y = 2\pi k /N$
for every $k \in {\mathbb Z}$. Since zeros of $f$ are clearly contained in any sublevel set $S_{\epsilon}$, we see
that ${\rm diam} (S_{\epsilon}) = \infty$ for every $\epsilon>0$. This is not an artifact that $I$ is an unbounded region since we could
restrict our attention to the unit square 
${\mathfrak U} = \{ z = x + iy : -1\le x, y \le 1 \}$ and conclude that ${\rm diam} (S_{\epsilon}\cap{\mathfrak U}) \sim 2$. 
In particular ${\rm diam}(S_{\epsilon}\cap{\mathfrak U}) \not\to 0$ as $\epsilon\to 0$ which is in sharp contrast
to the situation over the real field ${\mathbb R}$.
%Hence there are {\it no scale-invariant} bound of the form ${\rm diam} (S) \le C \epsilon^{\delta}$
%when $|f'|\ge 1$
%for some $\delta >0$ and some universal constant $C$.

A diameter bound is stronger than a measure bound
which is often what is required in applications. Let us modify the example above, moving the zeros
from the line ${\rm Re}(z) = 0$ where $|f'|$ is exponential in $N$ to the line ${\rm Re}(z) = -1$
where $|f'(z)| \equiv 1$. We shift the example to
\begin{equation}\label{example-complex}
f(z) \ = \ \frac{e^{N(z+1)} - 1}{N}
\end{equation}
so that the zeros of $f$ now lie on the line ${\rm Re} (z) = -1$ but $|f'(z)| \ge 1$ still holds on the
half-plane $I = \{ z \in {\mathbb C} : {\rm Re} (z) \ge -1 \}$. A simple calculation shows that
\begin{equation}\label{counterexample}
\bigl|\{ z \in {\mathfrak U} : |f(z)| \le \epsilon \}\bigr| \ \sim \ 
%\begin{cases} 
N \epsilon^2,  \ {\rm if} \ \epsilon < N^{-1} 
%N^{-1}           \ \ & \ {\rm if} \ N^{-1} \le \epsilon \le 1 \\
%\end{cases}
\end{equation}
for the sublevel set $S_{\epsilon}\cap{\mathfrak U}$ on the unit square
${\mathfrak U}$. 
Hence
there is {\bf no} scale-invariant bound for sublevel sets of the form
$|\{ |f| \le \epsilon\}| \le C \epsilon^2$ for general complex differentiable
functions with $|f'| \ge 1$ and where $C$ is a universal constant. A bound in terms
of $\epsilon^2$ is optimal and is the natural bound; when $f'\not= 0$, the function $f$
is an open map and locally 1-1 and so we expect, as the above example shows, the sublevel set to be a union of $\epsilon$ discs
centred at the zeros of $f$, at least when $\epsilon>0$ is small enough. 
%to aim for on a bounded region.
%Nevertheless the example above does not rule out a {\it scale-invariant} bound of
%the form $|S| \le C \epsilon$ which is less natural.

Considering $g(z) = [f(z)]^2$ where $f$ is the example above shows that there are no scale-invariant
bounds of the form $|\{ z \in {\mathfrak U}: |g(z)| \le \epsilon\}| \le C \epsilon$ which hold for
some universal constant $C$ and every complex differentiable function $g$ satisfying $|g''(z)| \ge 1$ on ${\mathfrak U}$.

\section{Complex sublevel set bounds}\label{vc-sublevel-complex}

The parameter $N$ in the counterexample $f = f_N$ above for a scale-invariant sublevel set bound
$|\{ |f| \le \epsilon\}| \le C \epsilon^2$ can be viewed as the number of zeros $f$ on the unit square
${\mathfrak U}$. But it can also be viewed as the logarithm of the
$L^{\infty}$ norm of $f$, $\log \|f\|_{L^{\infty}}$, on any larger disc, say ${\mathbb D}_2 = \{ z : |z| \le 2\}$. 
These two points of view are connected by Jensen's formula from basic complex analysis which has the
following consequence; we have
$$
\# \{z \in {\mathbb D}_r : f(z) = 0 \} \ \lesssim_r \ \log \|f\|_{L^{\infty}({\mathbb D}_2)}
$$
for any $r<2$ and
for general complex differentiable functions $f$ on ${\mathbb D}_2$ with $|f(0)|\ge 1$. 
Note that equality occurs for the
exponential example above. In some sense, the bound \eqref{counterexample} for our exponential example
is sharp.

\begin{proposition}\label{vc-sublevel-first-derivative} 
Suppose $f\in {\mathcal H}({\mathbb D}_2)$ with $|f(z)| \le M$ for $z\in {\mathbb D}_2$.
Let $N$ denote the number of zeros of $f$ in the disc ${\mathbb D}_{5/4}$. If
$64 \epsilon  \le M^{-1}$, then
\begin{equation}\label{sublevel-bound-first}
\bigl| \{z \in {\mathbb D} : |f(z)| \ \le \ \epsilon \} \bigr| \ \le \ 10 \, N \, \epsilon^2.
\end{equation}
%and in particular we have the bound $|\{z \in {\mathbb D}: |f(z)| \le \epsilon\}| \lesssim (\log M)  \epsilon^2$
%for small $\epsilon$ by Jensen's formula.
\end{proposition}

If we restrict our attention to complex
polynomials $f \in {\mathbb C}[X]$ of degree $d$, then the bound \eqref{sublevel-bound-first}
becomes $C_d \, \epsilon^2$ and we might think this gives us a scale-invariant bound for the class of
complex polynomials $f$ of degree at most $d$ (scaling $\mu f$ by a constant $\mu$ does not change
the degree of the polynomial $f$). But this is not the case due to the smallness condition $64 \epsilon \le M^{-1}$
which is necessary in general by the above example.

Proposition \ref{vc-sublevel-first-derivative} can be extended to higher derivatives.

\begin{proposition}\label{vc-complex-sublevel} Let $f \in {\mathcal H}({\mathbb D}_2)$ with $|f(z)|\le M$ for all $z\in {\mathbb D}_2$.
Suppose that $|f^{(k)}(z)| \ge 1$ for all $z \in {\mathbb D}$. Then if $\epsilon \ll_k M^{-(2k-1)}$,
\begin{equation}\label{sublevel-vc}
\bigl| \{ z \in {\mathbb D} : |f(z)| \le \epsilon \}\bigr| \ \lesssim_n \ M^{2(k-1)/k} N' \, \epsilon^{2/k} 
\end{equation}
where $N'$ now denotes the number of zeros of $f, f', \ldots and \ f^{(k-1)}$ 
in the disc ${\mathbb D}_{5/4}$.
\end{proposition}
%Here the notation $\lesssim_n$ means there exists some constant $C_n$, depending only on $n$, where we have 
%$\le C_n$.
%And  {\it if $\epsilon \ll_n M^{-(2n-1)}$, then ...} means that there is a sufficiently small constant $c_n < 1$, %depending only on $n$,
%such that {\it  if  $\epsilon \le c_n M^{-(2n-1)}$, then ...}

Propositions \ref{vc-sublevel-first-derivative} and \ref{vc-complex-sublevel} are proved using variants
of Hensel's lemma  from elementary number theory and adapting them to the archimedean setting 
of the real ${\mathbb R}$ or complex ${\mathbb C}$ field. In nonarchimedean settings, the
triangle inequality we used above to show $|f(z_1)-f(z_0)| \le 2 \epsilon$ for points $z_0, z_1 \in \{ |f|\le \epsilon\}$
improves to $|f(z_1) - f(z_0)| \le \epsilon$ and this is key to run a Hensel-type argument which
is an iterative scheme to find an actual zero of $f$, starting with an approximate zero $z \in \{|f|\le \epsilon\}$.
Losing a factor of 2 at each step in the iteration results in an unexceptable exponential loss $2^n$ at
the $n$th stage. However the example $(e^{N(z+1)} - 1)/N$ tells us that we can only expect
a good bound for small $\epsilon\ll_N 1$.  The argument can then be adjusted to hide/move
the factors of $2$ into the smallness of $\epsilon$.

Importantly, these Hensel-type arguments have the advantage of being local in nature;
the nondegeneracy condition $f' \not= 0$ (or $f^{(k)} \not= 0$) is only needed for points in the
sublevel set. The iteration scheme remains essentially in the sublevel set. Hence the bound \eqref{sublevel-vc} in Proposition \ref{vc-complex-sublevel} can be
improved to a measure bound of a {\it local sublevel set}:
\begin{equation}\label{sublevel-vc-local}
\bigl| \{ z \in {\mathbb D} : |f^{(k)}(z)| \ge 1, \  |f(z)| \le \epsilon \}\bigr| \ \lesssim_{k,M} \ \, \epsilon^{2/k} 
\ \ \ {\rm when} \ \ \epsilon \ll_{k,M} 1.
\end{equation}
In fact we can go further and deduce a structural statement about these local sublevel sets; they
are contained in a union of $\epsilon^{1/k}$ discs centred at the zeros of $f$ and its derivatives.   

This is the key to develop a theory and guide us to look for scale-invariant bounds. Not only are the arguments local
but we can {\it and should} be bounding local quantities, whether they be sublevel sets or oscillatory integrals.
Propositions \ref{vc-sublevel-first-derivative} and \ref{vc-complex-sublevel} are not central to the development
of our theory and so we provide proofs of these results in an appendix to this paper.

\section{Moving from analytic functions to polynomials}\label{analytic-versus-polynomial}

The argument at the outset for bounding the diameter of the sublevel set $\{ |f|\le \epsilon\}$
under a global condition $|f'|\ge 1$ no longer applies to the
the local sublevel set $\{ z : |f^{(k)}(z)|\ge 1, |f(z)|\le \epsilon\}$, even in the real setting. If we consider
the real function
\begin{equation}\label{example-real}
f(x) \ = \ \frac{2\sin(Nx)}{N} \ \ {\rm on} \ \ [0,1],
\end{equation}
we see that 
\begin{equation}\label{counterexample-real}
\bigl|\{ x \in [0,1] : |f'(x)| \ge 1, \ |f(x)| \le \epsilon \}\bigr| \ \sim \ 
\begin{cases} 
N \epsilon, & \ {\rm if} \ \epsilon < N^{-1} \\
1           \ \ & \ {\rm if} \ N^{-1} \le \epsilon \le 1 \\
\end{cases}.
\end{equation}
This is the real version of the example \eqref{counterexample}. Here the problem is that the
function $f$ and its derivative $f'$ are conspiring to produce many zeros of $f$ at places there
the derivative $f'$ is large. Hence, even in the real setting, there is {\bf no} scale-invariant bound
for local sublevel sets for general differentiable functions.

In both the real and the complex case, the examples \eqref{example-complex}
and \eqref{example-real}
are non-polynomial. If $f \in {\mathbb R}[X]$ has degree at most $d$, then the order structure
of ${\mathbb R}$ can once again be used to show that the local sublevel set
$$
S_{loc} \ = \ \bigl\{ x \in [0,1] :  |f^{(k)}(x)| \ \ge \ 1, \ |f(x)| \ \le \ \epsilon \bigr\}
$$
is a union $\cup I$ of at most $O_d(1)$ intervals and on each $I$, we can apply the argument
relying on the intermediate value theorem to obtain the scale-invariant bound $|S_{loc}| \le C_d \, \epsilon^{1/k}$.

There is an additional, less obvious, feature of the Hensel-like argument establishing bounds
such as \eqref{sublevel-vc-local}. The smallness condition $\epsilon \ll_{k,M} 1$ can be removed
if our function $f$ has the property that some derivative of bounded order has the global bound 
$|f^{(n)}(z)| \gtrsim M^{n/(n+1)}$ from below. 
Polynomials have this property and in fact satisfy a stronger property -- see Lemma \ref{triple-norm} below.

As a consequence, we can establish the following bound.

\begin{proposition}\label{poly-sublevel-derivative} Let $f \in {\mathbb C}[X]$ be a complex
polynomial of degree $d$. Then there is a constant $C_d$, depending only on $d$, such that 
\begin{equation}\label{sublevel-vc-local-poly}
\bigl| \{ z \in {\mathbb D} : |f^{(k)}(z)| \ge 1, \  |f(z)| \le \epsilon \}\bigr| \ \le \ C_d \, \epsilon^{2/k} 
\end{equation}
holds for any $k\ge 1$.
\end{proposition}
The bound \eqref{sublevel-vc-local-poly} in
Proposition \ref{poly-sublevel-derivative} is scale-invariant. As an immediate consequence, we see that
\begin{equation}\label{sublevel-vc-local-poly-again}
\bigl| \{ z \in {\mathbb D}_R : |f^{(k)}(z)| \ge \mu, \  |f(z)| \le \epsilon \}\bigr| \ \le \ C_d \, 
(\epsilon/\mu)^{2/k} 
\end{equation}
holds for any 
$f\in {\mathbb C}[X]$ of degree at most $d$, for any $k\ge 1$ and for any $R, \mu, \epsilon>0$. 

Here we use the notation ${\mathbb D}_R(z_0) = \{z \in {\mathbb C}: |z-z_0| \le R\}$ to denote
the disc of radius $R$ with centre $z_0$. Also we denote ${\mathbb D}_R(0)$ by ${\mathbb D}_R$ and ${\mathbb D}$
denotes the unit disc.
When we move to higher dimensions, we will denote 
${\mathbb B}^n_R({\underline{z}}_0) = \{{\underline{z}} \in {\mathbb C}^n: |{\underline{z}} - {\underline{z}}_0| \le R\}$
as the ball in ${\mathbb C}^n$ with centre ${\underline{z}}_0$ and radius $R$ with similar conventions for balls
centred at the origin and ${\mathbb B}^n$ denotes the unit ball in ${\mathbb C}^n$.

The local nature of the arguments allow us to extend the above result to polynomials of several variables.
Furthermore we are able to treat local sublevel sets defined by general linear partial differential operators
$Lf({\underline{z}}) = \sum_{1\le |\alpha| \le k} c_{\alpha}({\underline{z}}) \partial^{\alpha} f({\underline{z}})$
with general bounded measurable coefficients $\{c_{\alpha}({\underline{z}})\}$.
\begin{proposition}\label{poly-sublevel-several} Let $f \in {\mathbb C}[X_1,\ldots, X_n]$ be a complex
polynomial in $n$ variables of degree $d$ and let $L$ be as above. Then there is a constant $C_{d,n,L}$, depending only on $d, n$ and the
$L^{\infty}$ norms of the coefficients $\{c_{\alpha}({\underline{z}})\}$, 
such that 
\begin{equation}\label{sublevel-vc-local-poly-several}
\bigl| \{ {\underline{z}} \in {\mathbb B}^n : |Lf({\underline{z}})| \ge \mu, \  
|f({\underline{z}})| \le \epsilon \}\bigr| \ \le \ C_{d,n,L} \, (\epsilon/\mu)^{2/k} 
\end{equation}
holds for any $\mu, \epsilon>0$. If $L = \sum_{|\alpha|=k} c_{\alpha}({\underline{z}}) \partial^{\alpha}$
is  homogeneious of degree $k$, then
\begin{equation}\label{sublevel-vc-local-poly-several-R}
\bigl| \{ {\underline{z}} \in {\mathbb B}_R^n : |Lf({\underline{z}})| \ge \mu, \  
|f({\underline{z}})| \le \epsilon \}\bigr| \ \le \ C_{d,n,L} \, R^{2(n-1)} \, (\epsilon/\mu)^{2/k} 
\end{equation}
holds for any $R, \mu, \epsilon>0$. 
%Here ${\mathbb B}_R^n$ denotes the ball
%in ${\mathbb C}^n$ of radius $R$, centred at the origin.
\end{proposition}

{\bf Remarks}: The bounds \eqref{sublevel-vc-local-poly-several} and \eqref{sublevel-vc-local-poly-several-R} are scale-invariant.
In paricular to prove these bounds, by scaling $f$ by $\mu^{-1} f$, we can reduce to the case $\mu = 1$. Furthermore,
by a change of variables ${\underline{z}} = R {\underline{w}}$, the proof of \eqref{sublevel-vc-local-poly-several-R}
can be reduced to the case $R=1$ which is then subsumed by \eqref{sublevel-vc-local-poly-several}.

The polynomial $f({\underline{z}}) = g(z_1)$ could depend only on one variable
and so we see that the exponent $2/k$ is optimal (consider $g(z_1) = z_1^k$ and $L = \partial^k_1$).
And for the same reason, we see why the factor $R^{2(n-1)}$
is present.

When we consider general partial differential operators, even in the real ${\mathbb R}$ setting, there are
no scale-invariant bounds for general analytic functions. By Runge's theorem, one can show that for every $\epsilon>0$,
there is an analytic real-valued function $f$ with $\Delta f \ge 1$ on ${\mathbb D}$ such that
$|\{z \in {\mathbb D}: |f(z)| \le \epsilon\}| \ge 1$. Here $\Delta = \partial^2_x + \partial^2_y$ is the Laplace operator.

\section{The $H$ functional}\label{H-functional}

In the study of oscillatory integrals for real-valued phases, there is a very useful functional which we now consider
for complex functions $f$:
$$
H_{f,R} \ := \ \inf_{{\underline{z}}\in {\mathbb B}_R^n} H_f({\underline{z}}) \ \ {\rm where} \ \ 
H_f({\underline{z}}) \ = \ \max_{|\alpha|\ge 1} \bigl(\bigl| \partial^{\alpha} f({\underline{z}})/\alpha! |^{1/|\alpha|} \bigr).
$$
The estimate in Proposition \ref{poly-sublevel-several} can be extended in the following way.

\begin{proposition}\label{poly-sublevel-H} Let $n,d \in {\mathbb N}$ be given.
There exists a constant $C_{d,n}$, depending only on $d$ and $n$, such that for
any $f \in {\mathbb C}[X_1,\ldots, X_n]$ of degree $d$ and for any $R>0$,
we have
\begin{equation}\label{sublevel-poly-H}
\sup_{a \in {\mathbb C}} \
\bigl| \{ {\underline{z}} \in {\mathbb B}_R^n :  
|f({\underline{z}}) - a| \le 1 \}\bigr| \ \le \ C_{d,n} \, R^{2(n-1)} \, H_{f,R}^{-2}.
\end{equation}
There is a corresponding local version: 
\begin{equation}\label{sublevel-poly-H-local}
\sup_{a \in {\mathbb C}} \
\bigl| \{ {\underline{z}} \in {\mathbb B}_R^n :   H_f({\underline{z}}) \ge H, \ 
|f({\underline{z}}) - a| \le 1 \}\bigr| \ \le \ C_{d,n} \, R^{2(n-1)} \, H^{-2}
\end{equation}
holds for any $H>0$.
\end{proposition}
The bounds \eqref{sublevel-poly-H} and \eqref{sublevel-poly-H-local} are not scale-invariant due
the nonlinear nature of the functional $H_f$. However the bound \eqref{sublevel-poly-H-local} 
implies the scale-invariant bounds \eqref{sublevel-vc-local-poly-several} and \eqref{sublevel-vc-local-poly-several-R}. 
In fact, as remarked above, the bound \eqref{sublevel-vc-local-poly-several-R} reduces to \eqref{sublevel-vc-local-poly-several}
and furthermore, we may assume $\mu=1$. Finally to prove \eqref{sublevel-vc-local-poly-several} with $\mu =1$, we may
also assume $\epsilon \le 1$; otherwise, the trivial bound gives the desired result.
 
Setting $P = \epsilon^{-1} f \in {\mathbb C}[X_1, \ldots, X_n]$, we have
$$
 \bigl\{ {\underline{z}} \in {\mathbb B}^n : |Lf({\underline{z}})| \ge 1, \  
|f({\underline{z}})| \le \epsilon \bigr\} \subseteq  \bigcup_{1\le |\alpha| \le k} \bigl\{ {\underline{z}} \in {\mathbb B}^n :
|\partial^{\alpha}P({\underline{z}})| \ge (A\epsilon)^{-1}, \ |P({\underline{z}})| \le 1 \bigr\}
$$
where $A = C_{d,n} \max_{\alpha} \|c_{\alpha}\|_{\infty}$. For each $1\le |\alpha| \le k$, let $S_{\alpha}$ denote
the corresponding set on the right-hand side above. Note that
$$
S_{\alpha} \ \subseteq \ \bigl\{ {\underline{z}} \in {\mathbb B}^n :
H_P({\underline{z}}) \ge (A\epsilon)^{-1/|\alpha|}, \ |P({\underline{z}})| \le 1 \bigr\}
$$
and so \eqref{sublevel-poly-H-local} with $a=0$ implies $|S_{\alpha}| \lesssim_{d,n} \max(1, A^2) \, \epsilon^{2/|\alpha|}$.
But $\epsilon^{2/|\alpha|} \le \epsilon^{2/k}$ since $\epsilon \le 1$ and therefore
$$
\bigl| \{\{ {\underline{z}} \in {\mathbb B}^n : |Lf({\underline{z}})| \ge 1, \  
|f({\underline{z}})| \le \epsilon \}\bigr| \ \le \ C \, \epsilon^{2/k}
$$
where $C$ depends only on $d,n$ and the $L^{\infty}$ norms of the coefficients $\{c_{\alpha}\}$. Hence
\eqref{sublevel-vc-local-poly-several} holds.

\subsection*{The case $n=1$}
Consider the case of polynomials of a single variable $f \in {\mathbb C}[X]$, say of 
degree $d$. Let $z_{*} \in {\mathbb D}_R$
be a point where $H_{f,R} = H_f(z_{*})$. Then for $a = f(z_{*})$, we have $|f(z) - a| \le 1$
whenever $|z - z_{*}| \le (4 H_{f,R})^{-1}$. If fact $|f(z) - a| =$
$$
|f(z) - f(z_{*})| \ = \ \bigl| \sum_{k=1}^d f^{(k)}(z_{*})/k! (z - z_{*})^k \bigr| \ \le \ 
\sum_{k=1}^d (2 H_f)^k (4 H_f)^{-k} \ \le \ \sum_{k=1}^{\infty} 2^{-k} \ = \ 1.
$$
Hence if $\Lambda_R(f) := \min(R, H_{f,R}^{-1})$, then
$$
\Lambda_R(f)^2\\ \lesssim |{\mathbb D}_{(4H_{f,R})^{-1}}(z_{*}) \cap {\mathbb D}_R| \le  \sup_{a \in {\mathbb C}} 
\bigl| \{ z \in {\mathbb D}_R: |f(z) - a | \le 1 \}\bigr| \lesssim_d \, \Lambda_R(f)^2,
$$
illustrating the usefulness of the functional $H_f$ for sublevel set bounds. In particular,  we have
\begin{equation}\label{H-sublevel-equivalence}
\min(R, H_{f,R}^{-1})^2 \ \ \sim_d \ \
\sup_{a \in {\mathbb C}} \ \bigl| \{ z \in {\mathbb D}_R : \, |f(z) - a| \le 1 \}\bigr|
\end{equation}
and we can apply Proposition 3.3 from \cite{KW} (see also \cite{PS-II}) in the complex field setting which states
the following: if $f(z) = a \prod_{j=1}^m (z - w_j)^{e_j} \in {\mathbb C}[X]$, then
$$
\bigcup_{j=1}^m \bigl[B_{2^{-d} r_j}(w_j) \cap {\mathbb D}_R\bigr] \ \subseteq \
\bigl\{ z \in {\mathbb D}_R : |f(z)| \ \le \ 1 \bigr\} \ \subseteq \ 
\bigcup_{j=1}^m \bigl[B_{2^{d} r_j}(w_j) \cap {\mathbb D}_R\bigr] 
$$
where
$$
r_j \ = \ \min_{{\mathcal C}\ni w_j} \, \Bigl[ \frac{1}{|a \prod_{w_k\notin {\mathcal C}} (w_j - w_k)^{e_k}|} 
\Bigr]^{1/S({\mathcal C})}.
$$
Here the minimum is taken over all {\it root clusters} ${\mathcal C} \subset \{w_1, \ldots, w_m\}$
containing $w_j$
and $S({\mathcal C}) = \sum_{w_k\in {\mathcal C}} e_k$. Hence
$$
\max_{j\in J} \ \min\Bigl(R, \inf_{w_j \in {\mathcal C}}
\, \Bigl[ \frac{1}{|a \prod_{w_k\notin {\mathcal C}} (w_j - w_k)^{e_k}|} 
\Bigr]^{1/S({\mathcal C})}\Bigr)^2 \ \lesssim_d \ \min(R, H_{f,R}^{-1})^2
$$
holds where $J = \{1\le j \le m: B_{2^{-d} r_j}(w_j) \cap {\mathbb D}_R \not= \emptyset\}$.
%In particular
%^\begin{equation}\label{H-PS}
%\min_j \max_{w_j \in {\mathcal C}} \ |a \prod_{w_k\notin {\mathcal C}} (w_j - w_k)^{e_k}|^{1/S({\mathcal C})}
%\ \lesssim_d \ H_f.
%\end{equation} 

Momentarily we will see a similar but stronger relationship between $H_f$ and the roots of the derivative $f'$
of our polynomial.

\section{Oscillatory integrals with complex-valued phases}\label{oscillatory-int-complex}

In the Introduction we introduced oscillatory integrals
$$
I_{\phi}(f) \ = \ \int_{{\mathbb C}^n} {\rm e} (f({\underline{z}})) \, \phi({\underline{z}}) \, d{\underline{z}}
$$
with complex phases $f : {\mathbb C}^n \to {\mathbb C}$. Here
$\phi \in C^{\infty}_c ({\mathbb C}^n)$ is a smooth cut-off function.

The integrals $I_{\phi}(f)$ are connected to complex sublevel sets in the same way
that real oscillatory integrals are
connected to real sublevel sets. Consider the sublevel set
$$
S_{a,\epsilon}^R \ = \ \bigl\{ {\underline{z}} \in {\mathbb B}_R^n : \, |f({\underline{z}})-a| \ \le \ \epsilon \bigr\}
$$
and nonnegative functions $\phi_R \in C^{\infty}_c({\mathbb C}^n)$ such that 
$\phi_R \equiv 1$ on ${\mathbb B}^n_R$ and $\psi \in C^{\infty}_c({\mathbb C})$ such that $\psi \equiv 1$
on ${\mathbb D}$.
Then
$$
|S_{a,\epsilon}^R| \le \int_{{\mathbb C}^n} \psi((f({\underline{z}})-a)/\epsilon) \, \phi_R({\underline{z}}) \, d 
{\underline{z}} \ = \ \int_{{\mathbb C}} {\tilde{\psi}}(w) e(-a w/\epsilon) \Bigl[ \int_{{\mathbb C}^n} 
e(w f({\underline{z}})/\epsilon) \, \phi_R({\underline{z}}) \, d {\underline{z}} \Bigr] \, d w
$$
where ${\tilde{\psi}}(w_1 + i w_2) = 2 {\widehat{\psi}}(u,v)$ and $w_1 = (u+v)/2$ and $w_2 = (u-v)/2$.
In fact if we set ${\mathfrak u} = [{\rm Re} f] /\epsilon$ and ${\mathfrak v}= [{\rm Im} f]/\epsilon$, then by the fourier inversion formula,
$$
\psi({\mathfrak u}, {\mathfrak v}) \ = \ \int_{{\mathbb R}^2} {\widehat{\psi}}(u,v) 
e^{2\pi i [u {\mathfrak u} + v {\mathfrak v}]} \, du dv \ = \ 2
\int_{{\mathbb C}} {\tilde{\psi}}(w_1+ i w_2) {\rm e}((w_1+ i w_2) f({\underline{z}})) \, d w .
$$
Therefore we have
\begin{equation}\label{sublevel-osc}
|S_{a,\epsilon}^R| \ \le \ \int_{{\mathbb C}} |{\tilde \psi}(w)| \, |I_{\phi_R}(w f/\epsilon)| \, dw
\end{equation}
where ${\tilde{\psi}}$ is a Schwartz function on ${\mathbb C}$. Hence bounds for $I_{\phi_R}(w f)$
give bounds for $S_{a,\epsilon}^R$. 

When $n=1$, suppose we have a general complex differentiable function $f$ on the disc ${\mathbb D}_2$
which satisfies $|f''(z)| \ge 1$ on ${\mathbb D}_2$. Then a scale-invariant bound of the form
$$
|I_{\phi}(\lambda f)| \ \le \ C_{\phi} \, |\lambda|^{-1}
$$
where $\phi \in C^{\infty}_c({\mathbb D}_2)$ cannot hold with a constant $C_{\phi}$ only depending on $\phi$.
If such a scale-invariant bound were true, then \eqref{sublevel-osc} would imply the scale-invariant
bound
$$
|S_{a,\epsilon}^1| \ \le \ C_{\phi} \, \int_{{\mathbb C}} \min(1, \epsilon/|w|) \, |{\tilde{\psi}}(w)| \, dw
\ \lesssim_{\phi} \ \epsilon
$$
for sublevel sets which we have observed is impossible. 

We can relate derivatives of the real phase $\Phi(z) = {\rm Re}f(z) + {\rm Im} f(z)$ 
defining $I_{\phi}(f)$
to complex derivatives of $f$ when $n=1$. By the Cauchy-Riemann equations, we have
$\|\nabla \Phi(z)\| = \sqrt{2} |f'(z)|$ where $\|\cdot\|$ is the euclidean norm on ${\mathbb R}^2$.
Furthermore we have $|{\rm det} ({\rm Hess} \Phi(z))| = 4 |f''(z)|^2$.

For the convenience of the reader, we recall the statement of Theorem \ref{H-main-osc}, our main
oscillatory integral bound for $I_{\phi}(f)$ when the phase $f \in {\mathbb C}[X_1,\ldots, X_n]$
is a polynomial. Recall 
the $H$ functional adapted to the support of $\phi$:
$$
H_{f,\phi} \ := \ \inf_{{\underline{z}}\in {\rm supp}(\phi)} H_f({\underline{z}}) \ \ {\rm where} \ \ 
H_f({\underline{z}}) \ = \ \max_{|\alpha|\ge 1} \bigl(\bigl| \partial^{\alpha} f({\underline{z}})/\alpha! |^{1/|\alpha|} \bigr).
$$
%We will drop the subscript $\phi$ in $H_{f,\phi}$ but the context should be clear.

{\bf Theorem 1.1.} \ {\it For any $f \in {\mathbb C}[X_1, \ldots, X_n]$ of degree $d$ and for
any $\phi \in C^{\infty}_c({\mathbb C}^n)$, we have
\begin{equation}\label{I-H-bound}
|I_{\phi}(f)| \ \le \ C \ H_{f,\phi}^{-2}
\end{equation}
where $C = C_{d,n,\phi}$ only depends on $d,n$ and $\phi$.}

{\bf Remark}: \, If our smooth cut-off function is of the form $\phi_R({\underline{z}}) = \varphi(R^{-1}{\underline{z}})$
for some normalised bump function $\varphi$ (say ${\rm supp}(\varphi) \subseteq {\mathbb B}^n$ or
$\varphi \equiv 1$ on ${\mathbb B}^n$ and ${\rm supp}(\varphi) \subseteq {\mathbb B}^n_2$), then a change of
variables shows
$$
I_{\phi_R}(f) \ = \ R^{2n} I_{\varphi}(Q_R) \ \ {\rm where} \ \  Q({\underline{w}}) \ = \ f(R {\underline{w}})
$$
and since $H_{Q,1} = R H_{f,R}$, we see that \eqref{I-H-bound} implies
\begin{equation}\label{min-R}
|I_{\phi_R}(f)| \ \le \ C_{d,n,\varphi} \, R^{2(n-1)} \, \min(R, H_{f,R}^{-1})^2.
\end{equation}

As a consequence of Theorem \ref{H-main-osc} we have the following scale-invariant bound.

{\bf Corollary 1.2.} \ {\it Let $P \in {\mathbb C}[X_1,\ldots, X_n]$ have degree $d$
and suppose $|\partial^{\alpha}P({\underline{z}})/\alpha!| \ge 1$ 
for ${\underline{z}} \in {\rm supp}(\phi)$. Then for $\lambda \in {\mathbb C}$,
\begin{equation}\label{I-derivative-bound}
|I_{\phi}(\lambda P)| \ \le \ C_{d,n,\phi} \ |\lambda|^{-2/|\alpha|}.
\end{equation}}

\begin{proof} We apply Theorem \ref{H-main-osc} to the polynomial $f({\underline{z}}) = \lambda P({\underline{z}})$.
Note that for all ${\underline{z}} \in {\rm supp}(\phi)$, we have
$|\partial^{\alpha} f({\underline{z}})/\alpha!| = |\lambda| |\partial^{\alpha}P({\underline{z}})/\alpha!| 
\ge |\lambda|$
and so
$$
H_{f}({\underline{z}})^{|\alpha|} \ \ge \ |\partial^{\alpha} f({\underline{z}})/\alpha!| \ \ge \ |\lambda|,
$$
implying $H_{f,\phi} \ge |\lambda|^{1/|\alpha|}$. Hence the bound \eqref{I-H-bound} implies \eqref{I-derivative-bound}.
\end{proof}

Now consider a polynomial $f \in {\mathbb C}[X]$ in one variable and fix $R>0$. 
%$1 \lesssim_d \max_j R^j |b_j|$ (when $R=1$, this just says the norm of the coefficients is
%larger than 1). Consider the polynomial $Q(w) = P(Rw)$ on ${\mathbb D}$. Lemma \ref{triple-norm} applied
%to $Q$ gives us a $0\le k \le d$ such that $\max_j |b_j| R^j  \lesssim_d |Q^{(k)}(w)|$ for all $w \in {\mathbb D}$.
Putting \eqref{min-R} together with \eqref{H-sublevel-equivalence} and \eqref{sublevel-osc}, we have the
following observations: let $\phi_R(z) = \varphi(R^{-1}z)$ 
where $\varphi \in C^{\infty}_c({\mathbb D})$ such
that $\varphi \equiv 1$ on ${\mathbb D}_{1/2}$. Then
$$
\boxed{|I_{\phi_R}(f)| \ \lesssim_{d,\varphi}  \ \min(R, H_{f,R}^{-1})^2 \ \lesssim \ 
\sup_{a \in {\mathbb C}} \ \bigl| \{ z \in {\mathbb D}_R : \, |f(z) - a| \le 1 \}\bigr| }
$$
and so
$$
|I_{\phi_R}(f)| \ \lesssim_{d,\varphi} \ 
\int_{{\mathbb C}} |\psi(w)| 
|I_{\phi_{2R}}(w f)| \, dw 
$$
where $\psi$ is some fixed Schwartz function on ${\mathbb C}$ (whose fourier transform is nonnegative
and larger than 1 on ${\mathbb D}$). 

So here we have succeeded in controlling oscillatory integrals
with a general polynomial phase (in one variable) by the measure of sublevel sets.
This reverses the usual relationship and one can deduce oscillatory integral bounds from sublevel set bounds.
In particular we can bound
individual oscillatory integrals by an average of oscillatory integrals.

\section{Sharpness of \eqref{I-H-bound} in Theorem \ref{H-main-osc} when $n=1$}\label{sharpness}

The real analogue of the bound \eqref{I-H-bound} can be found in the book \cite{ACK}
where it is shown that 
\begin{equation}\label{I-H-real}
\Bigl| \int_{[0,1]^n} e^{2\pi i f({\underline{x}})} \, d{\underline{x}} \Bigr| \ \le \ C_{d,n} \, H_f^{-1}
\end{equation}
holds for any $f \in {\mathbb R}[X_1,\ldots, X_n]$ of degree at most $d$. Here $H_f = H_{f,1}$.
The bound \eqref{I-H-real}
is proved by applying the classical van der Corput estimates, together with many applications of
the intermediate value theorem. 

When $n=1$, it is shown in \cite{ACK}
that given any $f \in {\mathbb R}[X]$, there is a $c = c(f) \in [0,1]$ such that
\begin{equation}\label{H-c}
\Bigl| \int_0^c e^{2\pi i f(x)} \, dx \Bigr| \ \sim_d \ H_f^{-1}
\end{equation}
whenever $H_f \ge 1$. The proof of \eqref{H-c} relies heavily on the order structure
of the reals ${\mathbb R}$ and we do not know how to prove an analogous statement with truncations
for oscillatory integrals $I_{\phi}(f)$ with complex-valued phases. However in this section,
an alternative to \eqref{H-c} will be proposed which unforunately will have limited use. 

The asymptotic bound \eqref{H-c} for some $c \in [0,1]$ shows the sharpness of the bound \eqref{I-H-real} and it
can be used to compare the bound \eqref{I-H-real} to other known bounds which
are robust under truncations of oscillatory integrals. 

For instance, if the derivative $f'(x) = a \prod_{j=1}^m (x - z_j)^{e_j}$ has distinct roots $\{z_j\}$
with $\sum_j e_j = d-1$ where $d = {\rm deg}(f)$,
then a sharp bound due to Phong and Stein \cite{PS-I} is the following: for any $a < b$,
\begin{equation}\label{PS-integral-real}
\Bigl| \int_a^b e^{2\pi i f(x)} \, dx \Bigr| \ \le \ C_d \ \max_{1\le k \le m} \min_{{\mathcal C}\ni z_k}
\Bigl[\frac{1}{|a \prod_{z_j\not= {\mathcal C}} (z_k - z_j)^{e_j}}\Bigr]^{1/(S({\mathcal C}) + 1)}
\end{equation}
where the minimum is taken over all {\it root clusters} ${\mathcal C} \subseteq \{z_j\}$ containing
$z_k$ and $S({\mathcal C}) := \sum_{z_j \in {\mathcal C}} e_j$.
The constant $C_d$ only depends on the degree of $f$ and can be taken to be
independent of $a$ and $b$. Therefore by \eqref{H-c}, the bound \eqref{PS-integral-real} implies
\begin{equation}\label{PS-H-real}
\min_k \, \max_{{\mathcal C}\ni z_k} \
\bigl[ |a \prod_{z_j\notin {\mathcal C}} (z_k - z_j)^{e_j}\bigr]^{1/(S({\mathcal C}) + 1)} \ \le \
C_d \, H_f .
\end{equation}
In Section \ref{Phong-Stein}, we will give a direct proof of \eqref{PS-H-real} which will also hold for
complex polynomials $f \in {\mathbb C}[X]$. As a consequence, Theorem \ref{H-main-osc} will
imply a complex version of the Phong-Stein bound as stated in Proposition \ref{PS-intro}.

%The bound \eqref{I-H-bound} does not quite imply the bound 
%$\sup_{a \in {\mathbb C}} |S_{a,1}^1| \lesssim_{d} H_{f,1}^{-2}$
%in Proposition \ref{poly-sublevel-H} via \eqref{sublevel-osc} since the $H$ functional bound \eqref{I-H-bound}
%scales like $H_{w f}^{-2} \le |w|^{-2} H_f^{-2}$ for $w \in {\mathbb D}$ and $|w|^{-2}$
%is not integrable over the unit disc ${\mathbb D}$.

When $n=1$, we saw that $\sup_{a\in {\mathbb C}} |S_{a,1}^R| \sim_d \min(R, H_{f.R}^{-1})^2$ and so
the functional $H_{f,R}$ is precisely the right quantity which controls the measure of sublevel sets $S_{a,1}^R$.
Here we will see 
%that there is an asymptotic bound for oscillatory integrals where the
%functional $H_{f,\phi}$ is replaced by
the usefulness of a variant of $H_{f,\phi}$:
$$
J_{f,\phi} \ = \ \inf_{{\underline{z}}\in {\rm supp}(\phi)} J_f({\underline{z}}) \ \ {\rm where} \ \ 
J_f({\underline{z}}) \ = \ \max_{|\alpha|\ge 2} 
\bigl( |\partial^{\alpha} f({\underline{z}})/\alpha!|^{1/|\alpha|}\bigr).
$$
For sharp cut-offs, we write $J_{f,R} = \inf_{{\underline{z}} \in {\mathbb B}^n_R} J_f({\underline{z}})$.
We have $J_{f,\phi} \le H_{f,\phi}$ and so Theorem \ref{H-main-osc} implies
\begin{equation}\label{I-J-bound}
|I_{\phi}(f)| \ \le \ C_{d,n,\phi} \, J_{f,\phi}^{-2}
\end{equation}
which is a bound that does not depend on the linear coefficients of $f$ and
so gives bounds which are robust under linear perturbations of the phase. This is 
useful in many problems. 
%it is useful in fourier restriction theory. 
%giving sharp $L^2$ bounds where \eqref{I-J-bound}
%is the appropriate Fourier decay bound. 

Furthermore if our cut-off is of the form $\phi_R({\underline{z}}) = \varphi(R^{-1}{\underline{z}})$
for some normalised bump function $\varphi$, then as in \eqref{min-R}, we have
\begin{equation}\label{min-R-again}
|I_{\phi_R}(f)| \ \le \ C_{d,n,\varphi} \, R^{2(n-1)} \, \min(R, J_{f,R}^{-1})^2.
\end{equation}

For $f \in {\mathbb C}[X_1,\ldots, X_n]$, we write 
$f_{a,{\underline{b}}}({\underline{z}}) = f({\underline{z}}) - a - {\underline{b}}\cdot {\underline{z}}$
where $a\in {\mathbb C}$ and ${\underline{b}} \in {\mathbb C}^n$. We note that
$J_{f_{a,{\underline{b}}}, R} = J_{f,R}$ and since $J_{f,R} \le H_{f,R}$, we see that 
\eqref{sublevel-poly-H} in Proposition \ref{poly-sublevel-H} implies
\begin{equation}\label{sublevel-a-b}
\sup_{a, {\underline{b}}} \, \bigl| \{ {\underline{z}} \in {\mathbb B}^n_R: |f({\underline{z}}) - a - 
{\underline{b}} \cdot {\underline{z}}| \le 1 \}\bigr| 
\ \le \  C_{d,n} \, R^{2(n-1)} \ \min(R, J_{f,R}^{-1})^2.
\end{equation}

\begin{comment}
Importantly for us, the $J$ functional scales like $J_{w f, R}^{-2} \le |w|^{-1} J_{f,R}^{-2}$
for $w \in {\mathbb D}$ and $|w|^{-1}$ is integrable on ${\mathbb D}$. We modify
\eqref{sublevel-osc} in the following way: arguing as above, we have
\begin{equation}\label{sublevel-osc-again}
\sup_{a, {\underline{b}}} \ \bigl| \{ {\underline{z}} \in {\mathbb B}^n_R: 
|f({\underline{z}}) - a - {\underline{b}} \cdot {\underline{z}}| \le 1 \}\bigr| \ \le \  
\sup_{{\underline{b}} \in {\mathbb C}^n} \,  
\int_{{\mathbb C}} |{\tilde \psi}(w)| \, |I_{\phi_R}(w f_{{\underline{b}}})| \, dw
\end{equation}
where $f_{{\underline{b}}}({\underline{z}}) = f({\underline{z}}) - {\underline{b}} \cdot {\underline{z}}$
and $\phi_R \in C^{\infty}_c({\mathbb C}^n)$ with $\phi_R \equiv 1$ on ${\mathbb B}_R^n$. 
We note that $J_{w f_{{\underline{b}}}} = J_{wf} \ge \min(|w|^{1/2}, |w|^{1/d}) J_f$ and so the bound 
\eqref{I-J-bound} implies 
$$
\sup_{a, {\underline{b}}} \, \bigl| \{ {\underline{z}} \in {\mathbb B}^n_R: |f({\underline{z}}) - a - 
{\underline{b}} \cdot {\underline{z}}| \le 1 \}\bigr| 
\ \le \  C_{d,n,\phi_R} \,
J_{f,\phi_R}^{-2} \  
\int_{{\mathbb C}} |{\tilde \psi}(w)| \, \frac{1}{\min(|w|, |w|^{2/d})} \, dw 
$$
or
\begin{equation}\label{sublevel-a-b}
\sup_{a, {\underline{b}}} \, \bigl| \{ {\underline{z}} \in {\mathbb B}^n_R: |f({\underline{z}}) - a - 
{\underline{b}} \cdot {\underline{z}}| \le 1 \}\bigr| 
\ \le \  C_{d,n,\phi_R} \ J_{f,\phi_R}^{-2}.
\end{equation}
\end{comment}

Consider the case $n=1$ and $f \in {\mathbb C}[X]$. Let $z_{*} \in {\mathbb D}_R$
be a point where $J_{f,R} = J_{f,R}(z_{*})$. Then for $a = f(z_{*}) - f'(z_{*}) z_{*}$ and $b = f'(z_{*})$, 
we have $|f(z) - a - b z| \le 1$
whenever $|z - z_{*}| \le (4 J_{f,R})^{-1}$. If fact
$$
|f(z) - f(z_{*}) - f'(z_{*})(z-z_{*})| \ = \ \bigl| \sum_{k=2}^d f^{(k)}(z_{*})/k! (z - z_{*})^k \bigr| \ \le \ 
\sum_{k=2}^d (2 J_{f,R})^k (4 J_{f,R})^{-k} 
$$
which is at most $1/2$. Hence by \eqref{sublevel-a-b}, if $\Omega_{R}(f) := \min(R, J_{f,R}^{-1})$,
$$
\Omega_{R}(f)^2 \ \lesssim \ |{\mathbb D}_{(4J_{f,R})^{-1}}(z_{*}) \cap {\mathbb D}_R| \ \le \ \sup_{a, b \in {\mathbb C}} \
\bigl| \{ z \in {\mathbb D}_R: |f(z) - a - b z| \le 1 \}\bigr| \ \lesssim_d  \ \Omega_{R}(f)^2 
$$
and so
\begin{equation}\label{sublevel-a-b-again}
\Omega_{R}(f)^2 \ \sim_d \ \sup_{a,b \in {\mathbb C}} \, \bigl| \{ z \in {\mathbb D}_R: |f(z) - a - b z| \le 1 \}\bigr|.
\end{equation}
This illustrates the usefulness of the $J$ functional for sublevel set bounds. 

Fix $\varphi \in C^{\infty}_c({\mathbb D})$ such that $\varphi \equiv 1$ on ${\mathbb D}_{1/2}$ and
set $\phi_R(z) = \varphi(R^{-1} z)$. 
Arguing as in \eqref{sublevel-osc}, we have
\begin{equation}\label{sublevel-osc-again}
\sup_{a, b} \ \bigl| \{ z \in {\mathbb D}_R: 
|f(z) - a - b z| \le 1 \}\bigr| \ \le \  
\sup_{b \in {\mathbb C}} \,  
\int_{{\mathbb C}} |{\tilde \psi}(w)| \, |I_{\phi_{2R}}(w f_{b})| \, dw
\end{equation}
where $f_{b}(z) = f(z) - b z$.

Now set 
$$
\alpha_{f,R,\varphi} \ = \ \sup_{w, b \in {\mathbb C}} \ \bigl[\Omega_{2R}(w f)^{-2} |I_{\phi_{2R}}(wf_b)|\bigr].
$$ 
From
\eqref{min-R-again}, we have $\alpha_{f,R,\varphi} \le C_{d,\varphi}$ and we seek a lower bound.
Importantly for us, the $J$ functional scales like $J_{w f, R}^{-2} \le |w|^{-1} J_{f,R}^{-2}$
for $w \in {\mathbb D}$ and $|w|^{-1}$ is integrable on ${\mathbb D}$. 
We note that $J_{w f_{b}, 2R} = J_{wf, 2R} \ge \min(|w|^{1/2}, |w|^{1/d}) \, J_{f,2R}$ and so 
$$
\sup_{b \in {\mathbb C}} \,  
\int_{{\mathbb C}} |{\tilde \psi}(w)| \, |I_{\phi_{2R}}(w f_{b})| \, dw
\ \le \  \alpha_{f,R,\varphi} \
\Omega_{2R}(f)^{2} \  
\int_{{\mathbb C}} |{\tilde \psi}(w)| \, \frac{1}{\min(|w|, |w|^{2/d})} \, dw ,
$$
which implies, by \eqref{sublevel-a-b-again} and
\eqref{sublevel-osc-again}, 
$$
\Omega_{R}(f)^2 \ \lesssim_d \ \int_{{\mathbb C}} |{\tilde \psi}(w)| \, |I_{\phi_{2R}}(w f_{b})| \, dw
\ \lesssim_{d} \ \alpha_{f, R,\varphi} \ \Omega_{2R}(f)^2
$$
and so
$$
\Bigl[\frac{\min(1, J_{Q,1}^{-1})}{\min(1, J_{Q,2}^{-1})}\Bigr]^2 \ = \ 
\frac{\min(R, J_{f,R}^{-1})^2}{\min(R, J_{f,2R})^2} \ \lesssim_{d} \ \alpha_{f,R,\phi}
$$
where $Q(w) = f(Rw)$. Here we used the observation $J_{Q,1} = R J_{f,R}$.

\begin{comment}
\begin{lemma}\label{min-inequality}
We have
\begin{equation}\label{min(1,Q)}
\min(1, J_{Q,2}^{-1}) \ \lesssim_d \ \ \min(1, J_{Q,1}^{-1}).
\end{equation}
\end{lemma}

\begin{proof}
First suppose that $J_{Q,2} \lesssim_d 1$ and let $z_{*} \in {\mathbb D}_2$ be such that
$J_Q(z_{*}) = J_{Q,2}$. For any $z \in {\mathbb D}$ and $k\ge 2$, we expand
$$
Q^{(k)}(z) \ = \ Q^{(k)}(z_{*}) + Q^{(k+1)}(z_{*}) (z - z_{*}) + \cdots.
$$
Since $|z - z_{*}| \le 3$ and each $|Q^{(k+\ell)}(z_{*})| \lesssim_d J_{Q,2}^{k+\ell} \lesssim_d J_{Q,2}^k$,
we see that $|Q^{(k)}(z)|^{1/k} \lesssim_d J_{Q,2}$ which implies $J_{Q,1} \lesssim_d J_{Q,2}$.
This shows $\min(1, J_{Q,2}^{-1}) \sim_d J_{Q,2}^{-1} \lesssim_d J_{Q,1}^{-1}$ and
since $\min(1, J_{Q,2}^{-1})$ is clearly less than 1, we see that \eqref{min(1,Q)} holds in this case.

\end{proof}

\end{comment}

As a consequence, we see that $c_d \le \alpha_{f,R,\phi} \le C_{d,\phi}$ where
$$
c_d \ := \ \inf_{Q \in {\mathcal P}_d} \ \Bigl[\frac{\min(1, J_{Q,1}^{-1})}{\min(1, J_{Q,2}^{-1})}\Bigr]^2
$$
and ${\mathcal P}_d$ is the space of all polynomials $Q \in {\mathbb C}[X]$ of degree at most $d$.
Hence
for any $f\in {\mathbb C}[X]$ of degree at most $d$, there exist $w, b \in {\mathbb C}$ such that 
%\begin{equation}\label{I-sharpness}
$$
\boxed{c_{d} \ \min(R, J_{w f, R}^{-1})^2 \ \le \ \Bigl| \int_{{\mathbb C}} e(w [f(z) - b z]) \, \phi(R^{-1}z) \, dz \Bigr| 
\ \le \ C_{d,\phi} \ \min(R, J_{w f, R}^{-1})^2 .}
$$
%\end{equation}
This is our alternative to \eqref{H-c}.

Unforunately this nice general lower bound for oscillatory integrals is not very useful since
the constant $c_d$ turns out to be zero whenever $d\ge 3$!
For $f(z) = a + b z + c z^2$, we have $J_{f,R} = |c|^{1/2}$ for any $R$ and
so $c_2 = 1$. To see $c_d = 0$ when $d\ge 3$,
consider $f(z) = a(z - 3/2)^d$ and note that $J_{Q,2} \le |a|^{1/d}$ but
$\max_{k\ge 2} |a|^{1/k} \lesssim J_{Q,1}$. Hence
$\min(1, J_{Q,1}^{-1})/\min(1,J_{Q,2}^{-1}) \lesssim |a|^{-(d-2)/2d}$ as $|a| \to \infty$.

However there are certain subspaces of polynomials of degree at most $d$ where
a uniform lower bound is possible. For example when $d=4$, the subspace ${\mathcal P} := 
\{a + bz + c z^3 + d z^4\}$ of quartics with no quadratic term has the property
that 
$$
J_{Q,1} \ \sim \ J_{Q,2} \ \sim \ \max( |c|^{1/3}, |d|^{1/4} )
$$
for all $Q \in {\mathcal P}$.
Hence 
$$
\inf_{Q \in {\mathcal P}} \ \frac{\min(1, J_{Q,1}^{-1})}{\min(1, J_{Q,2}^{-1})}
\ \sim \ 1
$$
and so the above {\it boxed} statement can be applied to those $f$ in the subspace ${\mathcal P}$. 

\section*{{\bf \large Proofs of Proposition \ref{poly-sublevel-H} and Theorem \ref{H-main-osc}}}

We now give the details of the proofs of the main sublevel set bound, Proposition \ref{poly-sublevel-H},
and the main oscillatory integral bound, Theorem \ref{H-main-osc}.

\section{Preliminaries for the proof of Proposition \ref{poly-sublevel-H} and Theorem \ref{H-main-osc}}\label{prelims}

Let ${\mathbb C}[X_1, \ldots, X_n]$ denote the space of complex polynomials in $n$ variables.
For a multi-index $\alpha = (\alpha_1, \ldots, \alpha_n) \in {\mathbb N}^n$, we set 
$$
{\underline{z}}^{\alpha} \ = \ z_1^{\alpha_1} \cdots z_n^{\alpha_n}
$$
so that a polynomial $P \in {\mathbb  C}[X_1,\ldots, X_n]$ of degree at most $d$ can be written as
$$
P({\underline{z}}) \ = \ \sum_{|\alpha|\le d} c_{\alpha} {\underline{z}}^{\alpha} .
$$
Here $|\alpha| := \alpha_1 + \cdots + \alpha_n$.

For a multi-index $\alpha \in {\mathbb N}^n$
and $P \in {\mathbb C}[X_1, \ldots, X_n]$, we set
$$
\partial^{\alpha} P  \ := \ 
\frac{\partial^{|\alpha|} P}{\partial^{\alpha_1}z_1 \cdots \partial^{\alpha_n} z_n} \ \ \ {\rm and} \ \ \ 
\alpha! = \alpha_1! \alpha_2! \cdots \alpha_n!.
$$
Let $V_{n,k}$ denote the complex vector space of homogeneous polynomials in ${\mathbb C}[X_1, \ldots, X_n]$
of degree $k$. So $Q \in V_{n,k}$ means 
$Q({\underline{z}}) = \sum_{|\alpha| = k} b_{\alpha} \, {\underline{z}}^{\alpha}$.
The hermitian inner product
$$
\langle P, Q \rangle \ := \ \sum_{|\alpha|=k} \alpha! \,  a_{\alpha} {\overline{b_{\alpha}}}
$$
gives $V_{n,k}$ a Hilbert space structure. If $Q({\partial}) = \sum_{|\alpha| = k} b_{\alpha} \partial^{\alpha}$
denotes the corresponding differential operator, then note that
$$
{\overline{Q}}(\partial) P ({\underline{z}}) \ \equiv \ 
\langle P, Q \rangle \ \ \ {\rm where} \ \ \ {\overline{Q}}({\underline{z}}) = \sum_{|\alpha|=k} {\overline{b_{\alpha}}}
\ {\underline{z}}^{\alpha}.
$$
We need the following well-known fact (see for example, \cite{Stein-big}).

\begin{lemma}\label{unit}
Let $d(n,k)$ denote the dimenson of the Hilbert space $V_{n,k}$. There exists a sequence
of unit vectors ${\underline{u}}_j, j=1,\ldots, d(n,k)$ in ${\mathbb C}^n$ such that 
$$
Q_j({\underline{z}})\  := \  ({\underline{u}}_j \cdot {\underline{z}} )^k, \ \ j=1, \ldots, d(n,k)
$$
forms a basis for $V_{n,k}$.
\end{lemma}

\begin{proof} It suffices to show that ${\rm Span}\bigl\{ ({\underline{u}} \cdot {\underline{z}})^k : {\underline{u}}
\in {\mathbb C}^n \bigr\} = V_{n,k}$. Suppose not. Then since
$$
V_{n,k} \ = \ M \ \oplus \ M^{\perp}
$$
where $M = {\rm Span}\bigl\{ ({\underline{u}} \cdot {\underline{z}})^k : {\underline{u}}
\in {\mathbb C}^n \bigr\}$, we can find a non-zero $P \in M^{\perp}$. In particular, we have
$\langle P, Q_{\underline{u}} \rangle = 0$ for all ${\underline{u}} \in {\mathbb C}^n$ where
$Q_{{\underline{u}}}({\underline{z}}) = ({\underline{u}} \cdot {\underline{z}})^k$. Hence
$$
Q_{\underline{u}}(\partial) P ({\underline{z}}) \ = \ ({\underline{u}} \cdot \nabla )^k P({\underline{z}}) \
\equiv \ 0 \ \ {\rm for \ all} \ \  {\underline{u}} \in {\mathbb C}^n.
$$
Note that 
$$
 f^{(k)}(0) \ = \ ({\underline{u}} \cdot \nabla )^k P(0) \ = \ 0  \ \ \ {\rm where} \ \ \ 
f(z) = P(z {\underline{u}}) \ = \ \bigl[ \sum_{|\alpha| = k} a_{\alpha} {\underline{u}}^{\alpha} \bigr] \, z^k,
$$
implying $P({\underline{u}}) = 0$ for all ${\underline{u}} \in {\mathbb C}^n$. Hence we arrive at the contradiction 
$P = 0$.
\end{proof}

As a consequence of Lemma \ref{unit}, we see that for every $\alpha \in {\mathbb N}^n$
with $|\alpha| = k$, 
$$
{\underline{z}}^{\alpha} \ = \ \sum_{j=1}^{d(n,k)} c_j ({\underline{u}}_j \cdot {\underline{z}})^k
$$
for some choice of coefficients $c_j = c_j(\alpha) \in {\mathbb C}$. Hence we can write
\begin{equation}\label{partial-u}
\partial^{\alpha} \ = \ \sum_{j=1}^{d(n,k)} c_j \, ({\underline{u}}_j \cdot \nabla)^k
\end{equation}
for every $\alpha$ with $|\alpha| = k$.

\section{A structural sublevel set statement}\label{structural-statement}

Here we state a key sublevel set bound central to the proofs of Proposition \ref{poly-sublevel-H} and Theorem \ref{H-main-osc}.

\begin{proposition}\label{sublevel-structure-basic}
Let $Q \in {\mathbb C}[X]$ have degree $d$
and $z_0 \in {\mathbb D}$. Suppose $|Q^{(k)}(z_0)| \ge 1$ for some $k\ge 1$ and
$|Q(z_0)| \le \epsilon$ for some $0< \epsilon \le \epsilon_d$ where $\epsilon_d$ is a sufficiently small
positive constant, depending only on the degree of $Q$. Then there exists a zero $z_{*}$
of $Q^{(j)}$ for some $0\le j \le d$ such that $|z_0 - z_{*}| \lesssim_d  \, \epsilon^{1/k}$.
\end{proposition}

Proposition \ref{sublevel-structure-basic} has the following consequence.
%which is an extension of Proposition 3.1 in {\it From oscillatory integrals and sublevel  ... }, J. Geom Anal.
%2011.

\begin{corollary}\label{sublevel-structure-balls}
Let $P \in {\mathbb C}[X]$ be a polynomials of degree $d$ and
set 
$$
{\mathcal Z} \ = \ \bigl\{z \in {\mathbb C}: P^{(j)}(z) = 0 \ {\rm for \ some} \ 0 \le j \le d \bigr\}.
$$ 
Then there is a $C_d >0$, depending only on $d$, such that 
\begin{equation}\label{ball-containment}
\bigl\{z \in {\mathbb D}_R :\, |P(z)| \le \epsilon, \, |P^{(k)}(z)| \ge \mu \bigr\} \ \subseteq \ 
\bigcup_{z_{*} \in {\mathcal Z}} {\mathbb D}_{C_d (\epsilon/\mu)^{1/k}}(z_{*})
\end{equation}
holds for any $R, \epsilon, \mu > 0$.

Hence there is a constant $C_d$ such that
\begin{equation}\label{sublevel-estimate}
\bigl|\{z \in {\mathbb D}_R: |P(z)| \le \epsilon, \, |P^{(k)}(z)| \ge \mu \}\bigr| \ \le \ C_d \, (\epsilon/\mu)^{2/k} 
\end{equation}
holds for all $R, \epsilon, \mu>0$. This 
is the bound \eqref{sublevel-vc-local-poly-again} stated after Proposition \ref{poly-sublevel-derivative}.
\end{corollary}

\begin{proof} We may suppose the local sublevel set 
$$
S_R \ :=  \bigl\{z \in {\mathbb D}_R : |P(z)| \le \epsilon,
|P^{(k)}(z)|\ge \mu\bigr\}
$$
is nonempty. We will make a reduction to the case $\epsilon/(\mu R^k) \le \epsilon_d$
where $\epsilon_d$ is the sufficiently small constant appearing in Proposition \ref{sublevel-structure-basic}.
This will allow us to invoke Proposition \ref{sublevel-structure-basic}. We note that to prove \eqref{sublevel-estimate},
this reduction is immediate since the trivial bound $|S_R| \le \pi R^2$ implies \eqref{sublevel-estimate}
when $\epsilon/(\mu R^k) \ge \epsilon_d$.

First consider the case $\epsilon/(\mu R^k) \ge \epsilon_d$. In this case, we claim that
there is a zero $z_{*} \in {\mathcal Z}$
such that $|z_{*}| \le B_d (\epsilon/\mu)^{1/k}$ for some sufficiently large $B_d$. If this is true,
then for any $z \in S_R$,
$$
|z - z_{*}| \ \le \ R + B_d (\epsilon/\mu)^{1/k}\ \le \ [1/\epsilon_d + B_d] \, (\epsilon/\mu)^{1/k}
$$
and so $z \in {\mathbb D}_{C_d (\epsilon/\mu)^{1/k}}(z_{*})$ for $C_d = B_d + \epsilon_d^{-1}$,
establishing \eqref{ball-containment}.

To prove the claim, we argue by contradiction and suppose that 
$|z_{*}| \ge B_d (\epsilon/\mu)^{1/k}$ for all $z_{*} \in {\mathcal Z}$.
Factor 
$$
P(z) \ = \ a_d z^d + \cdots + a_0 \ = \ a_d (z-z_1) \cdots (z- z_d)
$$ 
into linear factors and order the
roots $|z_d| \le |z_{d-1}| \le \cdots \le |z_1|$. If $B_d$ is large enough, we have
$|z - z_{*}| \sim |z_{*}|$ for all $z \in {\mathbb D}_R$ and $z_{*} \in {\mathcal Z}$.
Hence for any $z \in {\mathbb D}_R$, 
 $ |P(z)| \sim_d |a_d z_1 z_2 \cdots z_d|$ and in particular 
\begin{equation}\label{<eps}
 |a_d \, z_1 \cdots z_d| \ \lesssim_d \ \epsilon
\end{equation}
since we are assuming $S_R$ is nonempty. Next note that 
$$
P^{(k)}(z)/k! = a_k + (k+1) a_{k+1} z + \cdot + {d\choose k} a_d z^{d-k} = c_{k} a_d (z - \eta_1) \cdots
(z- \eta_{d-k})
$$
and so for all $z \in {\mathbb D}_R$, $|P^{(k)}(z)| \sim_d |a_d \eta_1 \cdots \eta_{d-k}| = |a_k|$.
But $a_k = \pm a_d s_{d-k}(z_1, \ldots, z_d)$ where $s_j$ is the $j$th elementary symmetric polynomial
in $d$ variables. Hence 
$$
|s_{d-k}(z_1, \ldots, z_d)| \ \lesssim_d \ |z_{d-k} \cdots z_1|
$$ 
by our ordering of the roots. Hence 
\begin{equation}\label{mu<}
\mu \ \le \ |P^{(k)}(z)| \ \lesssim_d \ |a_d \, z_{d-k} \cdots z_1|
\end{equation}
since we are assuming $S_R$ is nonempty. Putting \eqref{<eps} and \eqref{mu<} together, we have
$$
\mu \, |a_d \, z_1 \cdots z_d| \ \lesssim_d \ \mu \epsilon \ \lesssim_d \ \epsilon \, |a_d \, z_{d-k}\cdots z_1|
$$
and so $|z_d \cdots z_{d-k+1}| \lesssim_d \epsilon/\mu$. But since $|z_j| \ge B_d (\epsilon/\mu)^{1/k}$
for every $j$, we see $B_d (\epsilon/\mu) \lesssim_d \epsilon/\mu$ which is impossible if $B_d$ is chosen
large enough.

It remains to treat the case $\epsilon/(\mu R^k) \le \epsilon_d$. Write any $z \in {\mathbb D}_R$ as $z = R w$
where $w\in {\mathbb D}$ and consider the polynomial $Q(w) = (\mu R^k)^{-1} P(R w)$. Fix
any $z \in S_R$ so that $|Q(w)| \le \epsilon/(\mu R^k)$ and $|Q^{(k)}(w)|\ge 1$. Since 
$\epsilon' = \epsilon/(\mu R^k) \le \epsilon_d$, we can apply Proposition \ref{sublevel-structure-basic}
to conclude there is a zero $w_{*}$ of $Q^{(j)}$ for some $0\le j\le d$ 
(so that $z_{*} = R w_{*} \in {\mathcal Z}$) such that $|w - w_{*}| \lesssim_d {\epsilon'}^{1/k} =
R^{-1} (\epsilon/\mu)^{1/k}$. Hence $|z - z_{*}| = R|w - w_{*}| \lesssim_d (\epsilon/\mu)^{1/k}$,
establishing \eqref{ball-containment}.
\end{proof}

Proposition \ref{sublevel-structure-basic} is a consequence of the following higher order
Hensel lemma which in turn is a extension of Proposition 2.1 in \cite{W-JGA}
from the nonarchimedean setting.

\begin{lemma}\label{hensel-L}
Let $\phi \in {\mathcal H}({\mathbb D}_2)$ with $M = M_{\phi} := \max_{z\in {\mathbb D}_{7/4}} |\phi(z)|$. Fix $L\ge 1$.
Suppose $z_0 \in {\mathbb D}$ is a point where $\phi^{(k)}(z_0) \not= 0$  for each $1\le k \le L$. 
Set \\
$\delta := |\phi(z_0)\phi'(z_0)^{-1}\phi^{(L)}(z_0)^{-1}|$ and for $1\le k \le L-1$, set
$$
\delta_k \ := \ |\phi^{(k+1)}(z_0) \phi^{(k)}(z_0)^{-1} \phi(z_0) \phi'(z_0)^{-1}|.
$$
Suppose $\delta_k \lesssim_L 1,  \, 2\le k \le L-1$ and suppose $|\phi(z_0)\phi'(z_0)^{-1}| \le 1/8$.
Then if $\delta M, \, \delta_1 \ll_L 1$,  there is a $z\in {\mathbb D}_{5/4}$ such that
$$
(a) \ \phi(z) \ = \ 0 \ \ \ {\rm and} \ \ \ (b) \ |z-z_0| \ \le \  2 |\phi(z_0)\phi'(z_0)^{-1}|.
$$
\end{lemma}

{\bf Remark}: When $L=1$, there are no $\delta_k$ for $1\le k \le L-1$ and the conditions reduce to the single
condition $M \delta = M |\phi(z_0) \phi'(z_0)^{-2}| \le 1/64$.

We postpone the proofs of Proposition \ref{sublevel-structure-basic} and Lemma \ref{hensel-L} 
until Section \ref{Prop-8} and Section \ref{hensel}, respectively.

\section{Proof of Proposition \ref{poly-sublevel-H}}\label{Prop 4.1}

Here we show how Proposition \ref{poly-sublevel-H} follows from Corollary \ref{sublevel-structure-balls}. 

Let $P \in {\mathbb C}[X_1, \ldots, X_n]$ be a polynomial  
of degree at most $d$. Our aim is to establish the bound \eqref{sublevel-poly-H-local} from Proposition \ref{poly-sublevel-H}. 
A simple scaling
argument (just make the change of variables ${\underline{z}} = R {\underline{w}}$) shows
that we may assume $R=1$. We reproduce the statement of \eqref{sublevel-poly-H-local} when $R=1$
for the convenience of the reader: for any $a \in {\mathbb C}$,
$$
\bigl| \{ {\underline{z}} \in {\mathbb B}^n : |H_P({\underline{z}})| \ge H, \
|P({\underline{z}}) - a | \le 1 \} \bigr| \ \le \ C_{d,n} \, H^{-2}.
$$
The bound \eqref{sublevel-poly-H-local} will follow
from the bound \eqref{sublevel-estimate} in Corollary \ref{sublevel-structure-balls}.

Note that
$$
S(H) \ := \ \bigl\{ {\underline{z}} \in {\mathbb B}^n : H_P({\underline{z}}) \ge H, \,
|P({\underline{z}}) - a| \le 1 \bigr\} \ \subseteq \ \bigcup_{\alpha} S_{\alpha}(H)
$$
where 
$$
S_{\alpha}(H) \ = \ \bigl\{ {\underline{z}} \in S(H) : H_P({\underline{z}}) \ = \ 
|\partial^{\alpha} P({\underline{z}})/\alpha!|^{1/|\alpha|} \bigr\}.
$$
 
For $\alpha$ with $|\alpha| = k$, Lemma \ref{unit} implies
there exists $(a_1, \ldots, a_{d(n,k)}) \in {\mathbb C}^{d(n,k)}$ such that 
$$
{\underline{z}}^{\alpha} \ = \ a_1 ({\underline{u}}_1 \cdot {\underline{z}})^k \ + \ \cdots \ + \
a_{d(n,k)} ({\underline{u}}_{d(n,k)} \cdot {\underline{z}} )^k
$$
or
$$
\partial^{\alpha} \ = \  a_1 ({\underline{u}}_1 \cdot \nabla )^k \ + \ \cdots \ + \
a_{d(n,k)} ({\underline{u}}_{d(n,k)} \cdot \nabla )^k .
$$
Hence there exists $c_{d,n} > 0$ such that whenever 
${\underline{z}} \in {\mathbb C}^n$ satisfies $|\partial^{\alpha} P ({\underline{z}})\alpha! | \ge H^k$,
there exists a $j = j({\underline{z}})$ such that $|({\underline{u}}_j \cdot \nabla )^k P ({\underline{z}})| \ge c_{d,n} H^k$.
Therefore for each $\alpha$ with $|\alpha| = k$,
$$
S_{\alpha}(H) \ \subseteq \ \bigcup_{j=1}^{d(n,k)} \bigl\{ {\underline{z}} \in S(H) : 
|({\underline{u}}_j \cdot \nabla)^k  P ({\underline{z}})| \ge c_{d,n} H^k \bigr\}
\ =: \ \bigcup_{j=1}^{d(n,k)} S_{\alpha, j}(H).
$$

It suffices to bound each $S_{\alpha,j}(H)$. Let $R_j : {\mathbb C}^n \to {\mathbb C}^n$ 
be an orthogonal transformation such
that $R_j {\underline{u}}_j = (1,0,\ldots, 0)$. We make the change of variables ${\underline{w}} = R_j {\underline{z}}$
so that
$$
|S_{\alpha, j}(H)| \ \le \ | \bigl\{ {\underline{w}} \in {\mathbb B}^n :  
| \partial^k_1 Q_j ({\underline{w}})| \ge c_{d,n} H^k, \
|Q_j({\underline{w}})| \le 1 \bigr\} |
$$
where $Q_j = P \circ R_j^{-1} - a$ is still a polynomial of degree at most $d$. Now fix $
{\underline{w}}' \in {\mathbb B}^{n-1}$ and define $Q(w) = Q_j(w, {\underline{w}}')$, a complex polynomial
in one variable of degree at most $d$. Consider the slice
$$
S_{\alpha, j}^{{\underline{w}}'}(H) \ := \ 
\Bigl\{ w \in {\mathbb D}_{\sqrt{1 - \|{\underline{w}}'\|^2}} : \,
|Q^{(k)}(w)| \ge c_{d,n} H^k, \ |Q(w)| \le 1 \Bigr\}
$$ 
so that
$$
|S_{j,\alpha}(H)| \ \le \ \int_{{\mathbb B}^{n-1}} |S_{\alpha, j}^{{\underline{w}}'}(H)| \, d{\underline{w}}' \ \le \ 
\int_{{\mathbb B}^{n-1}}  \bigl| \bigl\{ w \in {\mathbb D} : |f^{(k)}(w)| \ge 1, \ |f(w)| \le \delta \bigr\}\bigr| \, d{\underline{w}}'
$$
where $f(w) := Q(w)/(c_{d,n}H^k) \in {\mathcal P}_d$ and $\delta = (c_{d,n} H^k)^{-1}$. For each fixed 
${\underline{w}}' \in {\mathbb B}^{n-1}$, we apply \eqref{sublevel-estimate} to conclude
$$
|S_{\alpha,j}(H)| \ \le  \ C_d \ \int_{{\mathbb B}^{n-1}} \delta^{2/k} \, d{\underline{w}}' \ \le \ C_{d,n} \ 
H^{-2},
$$
finishing the proof of Proposition \ref{poly-sublevel-H}.

\section{The proof of Theorem \ref{H-main-osc}}\label{Thm 5.1}

For $\phi \in C^{\infty}_c({\mathbb C}^n)$ 
and $P \in {\mathbb C}[X_1, \ldots, X_n]$, we set
$$
I_{\phi}(P) \ = \ \int_{{\mathbb C}^n} e(P(z)) \,  \phi(z) \, dz
$$
and redefine (slightly)
$$
H_{P,\phi} \ := \ \inf_{z \in {\rm supp}(\phi)} \, 
\max_{|\alpha|\ge 1} |\partial^{\alpha} P(z)|^{1/|\alpha|},
$$
dropping the factorials for notational convenience. Since the cut-off $\phi$ is fixed, we will
also drop the subscipt $\phi$ and write $H_P$ instead of $H_{P,\phi}$, again for notational convenience.
%$\langla {\underline{z}}, {\underline{w}} \rangle 

Our aim is to establish the bound
\begin{equation}\label{poly-several-int}  
|I_{\phi}(P)| \ \le \ C_{d,n,\phi} \, H_P^{-2}.
\end{equation}

%\section*{Proof of Theorem \ref{poly-several-int}}

{\bf Remark}: In the application establishing Proposition \ref{ACK-intro}, it will be important to track
the dependence of the constant $C_{d,n,\phi}$ in \eqref{poly-several-int} on $\phi$
more precisely. The proof will show (see \eqref{phi-dependence-again}) that
\begin{equation}\label{phi-dependence}
C_{d,n,\phi} \ \lesssim_N \ [1 + J_P^{-N} \|\phi\|_{C^N}]
\end{equation}
for any $N\ge 4$. Here
$$
J_{P} \ = \ \inf_{z \in {\rm supp}(\phi)} J(z) \ := \ \inf_{z\in {\rm supp}(\phi)} 
\max_{2\le |\alpha| \le d} |\partial^{\alpha} P(z)|^{1/|\alpha|}
$$
so that
$H_{P} = \inf_{z\in {\rm supp}(\phi)} \max(|\nabla P(z)|, J(z))$. 

Set $r(z) := 1/J(z)$. 

\begin{lemma}\label{J-constant} \ 

1. There exists $\epsilon_{d,n}>0$ such that if $\epsilon \le \epsilon_{d,n}$ and $z = w + \epsilon r(w) u$
with $u\in {\mathbb B}^n$, then
\begin{equation}\label{J-2}
(1/2) \, J(w) \ \le \ J(z) \ \le \ 2 \, J(w).
\end{equation}

2. For all $A > 0$, there exists a constant $C = C_{A,d,n}$ such that if $z = w + A r(w) u$ with $u\in {\mathbb B}^n$,
then
\begin{equation}\label{J-A}
J(z) \ \le \ C \, J(w).
\end{equation}
\end{lemma}

\begin{proof} Let $a = \epsilon$ or $A$ and suppose $J(w) = |\partial^{\alpha_0}P(w)|^{1/|\alpha_0|}$ for some
$2\le |\alpha_0| =: k_0 \le d$. For any $\alpha$ with $|\alpha| = k$ and $2\le k \le d$,  set $Q = \partial^{\alpha}P$.
Hence 
$$
Q(z) \ = \ Q(w) \ + \ \sum_{1\le |\beta| \le d-k} \frac{1}{\beta !} \ \partial^{\beta} Q(w) \ (a r(w))^{|\beta|} \, u^{\beta}.
$$
For $1 \le |\beta| \le d - k$,
$$
| \partial^{\beta} Q(w)|  (a r(w))^{|\beta|} \ = \ |\partial^{(\alpha + \beta)} P(w)| a^{|\beta|} 
\frac{1}{|\partial^{\alpha_0} P(w)|^{|\beta|/k_0}} \ \le \ a^{|\beta|} \, |\partial^{\alpha_0} P(w)|^{k/k_0}
$$
and therefore
$$
\Bigl| \sum_{1\le |\beta| \le d-k} \frac{1}{\beta !} \ \partial^{\beta} Q(w) \ (a r(w))^{|\beta|} \Bigr|
\ \le \ |\partial^{\alpha_0} P(w)|^{k/k_0} \sum_{\ell=1}^{d-k} a^{\ell} \sum_{|\beta| = \ell} \frac{1}{\beta!}
$$
$$
= \  |\partial^{\alpha_0} P(w)|^{k/k_0} \sum_{\ell=1}^{d-k} \frac{(an)^{\ell}}{\ell!} \ \le  \
 (e^{a n} - 1) \, |\partial^{\alpha_0} P(w)|^{k/k_0} .
$$
Now suppoer $a = A$. Then
$$
|\partial^{\alpha} P(z) | \ \le \ |\partial^{\alpha_0} P(w)|^{k/k_0} \ +  \  (e^{an} - 1) \, 
|\partial^{\alpha_0}P(w)|^{k/k_0} \ \le \ C \, |\partial^{\alpha_0}P(w)|^{k/k_0}
$$
and so for all $\alpha$ with $2 \le |\alpha| \le d$,
$$
|\partial^{\alpha} P(z)|^{1/|\alpha|} \ \le \ C J(w), \ \ {\rm implying} \ \ J(z) \ \le \ C \, J(w).
$$
For $a = \epsilon$, we have $e^{\epsilon n} - 1 \le 1/2$ for $\epsilon_{d,n}$ small enough and so
for all $\alpha$ with $2\le |\alpha| \le d$,
$$
|\partial^{\alpha} P(z)|^{1/|\alpha|} \ \le \ 2 J(w), \ \ {\rm implying} \ \ J(z) \ \le \ 2 \, J(w).
$$
Furthermore when $\alpha = \alpha_0$,
$$
|\partial^{\alpha_0} P (z) | \ \ge \ |\partial^{\alpha_0} P(w)| \ - \ (e^{\epsilon n} - 1) |\partial^{\alpha_0} P(w)|
\ \ge \ (1/2)|\partial^{\alpha_0}P(w)|
$$
and so 
$$
(1/2) J(w) \ \le \ |\partial^{\alpha_0} P(z)|^{1/|\alpha_0|} \ \le \ J(z).
$$
\end{proof}

For any $\epsilon < \epsilon_{d,n}$, we have by the standard Vitali covering lemma, 
$$
{\rm supp}(\phi) \ \subseteq \ \bigcup_{j=1}^M {\mathbb B}^n_{3 \epsilon r(z_j)}(z_j) \ \ {\rm where \ each} \ \ z_j \in {\rm supp}(\phi)
$$
and $\{{\mathbb B}^n_{\epsilon r(z_j)}(z_j)\}_{j=1}^M$ are pairwise disjoint.

\begin{lemma}\label{bounded-overlap} For any $C>0$, there exists $N = N_{C,d,n,\epsilon} \in {\mathbb N}$
such that every
$$
z_{*} \ \in \ \bigcup_{j=1}^M {\mathbb B}^n_{C r(z_j)}(z_j)
$$ 
lies in at most $N$ of these balls. 
\end{lemma}

\begin{proof} First we note that for $z \in {\mathbb B}^n_{C r(z_j)}(z_j)$, we have
$$
J(z) \ \le \ C' J(z_j) \ \ {\rm or} \ \ r(z_j) \ \le \ C' r(z) \eqno{(*)}
$$
by Lemma \ref{J-constant}, part 2. The constant $C'$ depends on $C, d$ and $n$. We may assume
that $C > \epsilon$. Now fix $z_{*} \in B_{C r(z_j)}(z_j)$ and choose 
$w_{\ell} = z_j + t_{\ell}(z_{*} - z_j) \in {\mathbb B}^n_{C r(z_j)}(z_j)$ with $1\le \ell \le L$ 
and $t_{\ell} \in [0,1]$ such that $|z_j - w_1| = \epsilon r(z_j)$,
$$
 |w_1 - w_2| = \epsilon r(w_1), \ldots, |w_{L-1} - w_L| = \epsilon r(w_{L-1})
\ \ \ {\rm and} \ \ \  |w_L - z_{*}| \le \epsilon r(w_L).
$$
Set $w_0 = z_j$ and $w_{L+1} = z_{*}$. Note that $(*)$ implies $r(z_j) \le C' r(w_{\ell})$ for every
$0\le \ell \le L+1$ and so
$$
\epsilon r(z_j) L \ \le \ C' \epsilon \sum_{\ell=0}^{L-1} r(w_{\ell}) \ = \ C' \sum_{\ell=0}^{L-1}
|w_{\ell} - w_{\ell+1}| \ \le \ C' |z_j - z_{*}| \ \le \ C C' r(z_j),
$$
implying $L \lesssim_{C,C',\epsilon} 1$. 

Next by Lemma \ref{J-constant} part 1., we have
$$
(1/2) J(w_{\ell-1}) \ \le \ J(w_{\ell}) \ \ \ {\rm or} \ \ \ r(w_{\ell}) \le 2 r(w_{\ell-1})
$$
for $1\le \ell \le L$. Hence
$$
r(z_{*}) \ = \ r(w_{L+1}) \ \le \ 2 r(w_{L}) \ \le \ 2^2 r(w_{L-1}) \ \le \ \cdots \ \le \ 2^{L+1} r(w_0) \ = \
2^{L+1} r(z_j)
$$
and so
$$
r(z_{*}) \ \lesssim_{C,C',\epsilon} \ r(z_j). \eqno{(\dag)}
$$
Set
$$
{\mathcal A} \ := \ \bigl\{ 1\le j \le M : \ z_{*} \in {\mathbb B}^n_{C r(z_j)}(z_j) \bigr\}
$$
so that by $(*)$,
$$
\bigcup_{j\in {\mathcal A}} {\mathbb B}^n_{\epsilon r(z_j)}(z_j) \ \subseteq \ 
{\mathbb B}^n_{2 C' r(z_{*})}(z_{*}).
$$
Hence by the disjointness of $\{{\mathbb B}^n_{\epsilon r(z_j)}(z_j) \}_{j=1}^M$ and $(\dag)$,
$$
\# {\mathcal A}\  r(z_{*})^{2n} \ \lesssim_{C,C',\epsilon} \ \sum_{j\in {\mathcal A}} |{\mathbb B}^n_{\epsilon r(z_j)}(z_j) | \ 
\le \ [2 C' r(z_{*})]^{2n},
$$
implying $\# {\mathcal A} \lesssim_{C,d,n,\epsilon} 1$.
\end{proof}

Now fix $\psi \in C^{\infty}_c({\mathbb B}^n_2)$ such that $\psi \equiv 1$ on ${\mathbb B}^n$ and set
$$
\Psi(z) \ := \ \sum_{j=1}^M \psi((3\epsilon)^{-1} r(z_j)^{-1} (z - z_j)),
$$
the sum having bounded overlap by Lemma \ref{bounded-overlap} and so $\Psi$ is bounded.
Note that for every $z \in {\rm supp}(\phi)$, 
we have $z \in {\mathbb B}^n_{3\epsilon r(z_j)}(z_j)$ for some $1\le j \le M$,
implying $\psi((3\epsilon)^{-1} r(z_j)^{-1} (z - z_j)) = 1$ and hence $\Psi \ge 1$ on ${\rm supp}(\phi)$. Also
$$
{\rm supp}(\Psi) \ \subseteq \ \bigcup_{j=1}^M {\mathbb B}^n_{6 \epsilon r(z_j)}(z_j).
$$
We decompose the oscillatory integral 
$$
I \ = \ \sum_{j=1}^M \, \int_{{\mathbb C}^n} e(P(z)) \, \varphi_j(z) \, dz \ =: \ \sum_{j=1}^M I_j
$$
where 
$$
\varphi_j(z) \ := \ \psi_j((3\epsilon)^{-1} r(z_j)^{-1} (z - z_j)) \ \phi(z) \, \Psi^{-1}(z).
$$
Let 
$$
A \ := \ \{ 1 \le j \le M: \, |\nabla P(z_j)| \ \le \ J(z_j) \}.
$$
Fix $j\in A$ and let $z \in {\rm supp}(\varphi_j)$. Then $z = z_j + 6\epsilon r(z_j) u$ for some
$u \in {\mathbb B}^n$. Suppose $J(z_j) =  |\partial^{\alpha}P(z_j)|^{1/|\alpha|}$
for some $2 \le |\alpha| \le d$. Then since
$\epsilon < \epsilon_{d,n}$ and $\epsilon_{d,n}$ is small,
we see by the proof of Lemma \ref{J-constant}, part 1,
$$
(1/2) |\partial^{\alpha} P(z)|^{1/|\alpha|} \ \le \ |\partial^{\alpha} P(z_j)|^{1/|\alpha|} \ \le \ 
2 |\partial^{\alpha}P(z)|^{1/|\alpha|}
$$
which implies $J(z_j) \le 2 J(z)$. Furthermore the proof of Lemma \ref{J-constant} also shows that
$$
|\nabla P(z)| \ \le \ 2 |\partial^{\alpha}P(z_j)|^{1/|\alpha|} \ \le \ 4 |\partial^{\alpha}P(z)|^{1/|\alpha|}
$$
and therefore
$$
 |\nabla P(z)| \ \le \ 4  |\partial^{\alpha}P(z)|^{1/|\alpha|}
\ \le \ 4 J(z).
$$
Hence
$$
\bigl| \sum_{j\in A} I_j \bigr| \ \le \ \int_{{\mathfrak S}}
\bigl[ \sum_{j\in A} \varphi_j(z)\bigr] \, dz
$$
where ${\mathfrak S} := \{z \in {\rm supp}(\phi) :  |\nabla P(z)| \le 4 J(z)\}$. The sum
$ \sum_{j\in A} \varphi_j(z)$ is uniformly bounded and so 
\begin{equation}\label{S-frak}
\bigl| \sum_{j\in A} I_j \bigr| \ \lesssim \ |{\mathfrak S}|.
\end{equation}
For $z \in {\mathfrak S}$, we note that
$$
 J(z) \ \sim \ H(z) \ := \ \max_{1\le k \le d} \ \max_{|\alpha| = k} 
|\partial^{\alpha}P(z)|^{1/|\alpha|}.
$$
Since $|\nabla P(z)| \sim \max_{1\le j \le n} |\partial_j P(z)|$, we see that
$$
{\mathfrak S} \ \subseteq \ \bigcup_{j=1}^n \bigl\{ z \in {\mathfrak S} : |\nabla P(z)| \sim |\partial_j P(z)| \bigr\}
\ =: \ \bigcup_{j=1}^n {\mathfrak S}_j.
$$
Note that ${\mathfrak S}_j \ = \ \{z \in S :  |\partial_j P(z)| \le 4 J(z) \sim H(z)\}$. 
Since 
$$
J(z) \ = \ \max_{2\le |\alpha| \le d} |\partial^{\alpha}P(z)|^{1/|\alpha|},
$$
we can decompose each
${\mathfrak S}_j$ further; we have
$$
{\mathfrak S}_j \ \subseteq \ \bigcup_{2\le |\alpha| \le d} \bigl\{ z \in {\mathfrak S}_j : J(z) = |\partial^{\alpha}P(z)|^{1/|\alpha|}
\bigr\} \ =: \ \bigcup_{2\le |\alpha| \le d} {\mathfrak S}_{j, \alpha}.
$$
For each $1\le j \le n$ and $2\le |\alpha| \le d$, we decompose
$$
{\mathfrak S}_{j,\alpha} \ \subseteq \ \bigcup_{r,\ell \ge 0} 
\bigl\{ z \in {\mathfrak S}_j :2^{\ell} |\partial_j P(z)| \sim |\partial^{\alpha}P(z)|^{1/|\alpha|} \sim 2^r H_P \bigr\}
$$
Hence $|{\mathfrak S}_{j,\alpha}| \le \sum_{r,\ell\ge 0} |{\mathfrak S}_{j,\alpha}^{r,\ell}|$ where
$$
{\mathfrak S}_{j,\alpha}^{r,\ell} \ := \ 
\bigl\{ z \in {\mathfrak S}_j :2^{\ell} |\partial_j P(z)| \sim |\partial^{\alpha}P(z)|^{1/|\alpha|} \sim 2^r H_P \bigr\}.
$$
We apply Proposition \ref{poly-sublevel-several} to
$Q(z) = \partial_j P(z)$ to conclude
$$
|{\mathfrak S}_{j,\alpha}^{r,\ell}| \ \le \ C_{d,n} \Bigl[ \frac{2^{r-\ell} H_P}{(2^r H_P)^{|\alpha|}}\Bigr]^{2/(|\alpha| - 1)}
\ = \ C_{d,n} \, 2^{-2\ell(|\alpha| - 1)^{-1}} 2^{-2r} \, H_P^{-2} 
$$
which sums in $r\ge 0$ and $\ell \ge 0$ to produce the bound $|{\mathfrak S}_{j,\alpha}| \le C_{d,n} H_P^{-2}$
for every $1\le j \le n$ and $2 \le |\alpha| \le d$. Hence $|{\mathfrak S}| \lesssim_{d,n} H_P^{-2}$ and
so \eqref{S-frak} implies
\begin{equation}\label{A}
\bigl| \sum_{j\in A} I_j \bigr| \ \lesssim_{d,n} \ H_P^{-2}.
\end{equation}

We now turn to the sum $\sum_{j\in B} I_j$ where 
$$
B \ := \ \{ 1 \le j \le M: \, J(z_j) \ \le \ |\nabla P(z_j)| \, \}.
$$
For each $j \in B$, we make a change of variables in the integral $I_j$ to write it as
$$
I_j \ = \ r(z_j)^{2n} \int_{{\mathbb C}^n} e(P(z_j + r(z_j) w) \psi((3\epsilon)^{-1} w) {\tilde{\phi}}(w) \, dw
$$
where ${\tilde{\phi}}(w) = [\phi/\Psi](z_j + r(z_j) w)$. We expand
$$
P(z_j + r(z_j) w) = P(z_j) + r(z_j) \nabla P(z_j) \cdot w + \sum_{2\le |\alpha| \le d} \frac{1}{\alpha!} \partial^{\alpha}P(z_j)
\, r(z_j)^{|\alpha|} w^{\alpha}
$$
so that
$$
I_j \ = \ r(z_j)^{2n} e(P(z_j)) \int_{{\mathbb C}^n} e(\Lambda_j Q(w)) \psi((3\epsilon)^{-1} w) {\tilde{\phi}}(w) \, dw
$$
where $\Lambda_j = r(z_j) |\nabla P(z_j)|$ and 
$$
Q(w) \ := \ A_j \cdot w +  \sum_{2\le |\alpha| \le d} B_{\alpha} w^{\alpha}.
$$
Here $A_j = \nabla P(z_j)/|\nabla P(z_j)|$ and 
$$
B_{\alpha} \ := \  \Lambda_j^{-1} \, \frac{1}{\alpha!} \partial^{\alpha}P(z_j) \, r(z_j)^{|\alpha|} 
$$
Since $j \in B$, we see that $\Lambda_j \ge 1$ and each $|B_{\alpha}| \lesssim 1$. Hence, since 
$|w| \le 6 \epsilon$ and $\epsilon>0$ is small, we have $|\nabla Q(w)| \gtrsim 1$ and $\|Q\|_{C^{d}} \lesssim 1$
on the support of $\psi((3\epsilon)^{-1} \cdot)$, we have
$$
|I_j| \ \lesssim_N \ r(z_j)^{2n} \, \Lambda_j^{-N}
$$
for all $N \ge 1$. 

{\bf Remark}: In fact integration by parts shows
\begin{equation}\label{phi-dependence-again}
|I_j| \lesssim_N r(z_j)^2 [1 + r(z_j)^N \|{\tilde{\phi}}\|_{C^N}] \Lambda_j^{-N} \lesssim_{N,\psi}
r(z_j)^2 [1+J_P^{-N}\|\phi\|_{C^N}] \Lambda_j^{-N}
\end{equation}
where $J_P = \inf_{z \in {\rm supp}(\phi)} J(z)$. This gives the dependence on the cut-off $\phi$
made in the Remark \eqref{phi-dependence} at the beginning of this section.

From Lemma \ref{J-constant}, we know that $r(z)$ is roughly constant on the support of
$\psi((3\epsilon)^{-1} \cdot)$. Furthermore, using $J(z_j) \le |\nabla P(z_j)|$ for $j \in B$, we see
that $|\nabla P(z)|$ is roughly constant on $\psi((3\epsilon)^{-1} \cdot)$ (the proof is precisely the same
as in Lemma \ref{J-constant}). Therefore
$$
|I_j| \ \lesssim_{N,\epsilon} \ \int_{{\mathbb C}^n} \Bigl[\frac{J(z)}{|\nabla P(z)|}\Bigr]^N \, \psi((3\epsilon)^{-1}r(z_j)^{-1}
(z-z_j)) \, dz
$$ 
and so 
$$
\bigl| \sum_{j\in B} I_j \bigr| \ \lesssim_{N,\phi,\epsilon} \ \int_{{\mathcal B}} \Bigl[\frac{J(z)}{|\nabla P(z)|}\Bigr]^N \, dz
$$
where ${\mathcal B} = \{ z \in {\mathbb B}^n : J(z) \lesssim |\nabla P(z)| \}$. Exactly as in the sum over $A$,
we may decompose
$$
{\mathcal B} \subseteq \bigcup_{j,\alpha} \bigcup_{r,\ell\ge0} {\mathcal B}_{j,\alpha}^{r,\ell} \ \ {\rm where} \ \
{\mathcal B}_{j,\alpha}^{r,\ell} := \bigl\{ z \in {\mathcal B}: 2^{\ell} |\partial^{\alpha}P(z)|^{1/|\alpha|} \sim 
|\partial_j P(z)| \sim 2^r H_P \, \bigr\},
$$
implying
\begin{equation}\label{B-sum}
\bigl| \sum_{j\in B} I_j \bigr| \ \lesssim_{N,\phi,\epsilon} \ \sum_{j=1}^n \sum_{2\le |\alpha| \le d} \ \sum_{r,\ell\ge 0} 2^{-N\ell}
|{\mathcal B}_{j,\alpha}^{r,\ell}|.
\end{equation}
For fixed $1\le j \le n$ and $2\le |\alpha| \le d$, 
we apply Proposition \ref{poly-sublevel-several} to
$Q(z) = \partial_j P(z)$ to conclude
$$
|{\mathcal B}_{j,\alpha}^{r,\ell}| \ \lesssim_{d,n,\phi} \  \Bigl[ \frac{2^{r} H_P}{(2^{r-\ell} H_P)^{|\alpha|}}\Bigr]^{2/(|\alpha| - 1)}
\ = \ C_{d,n,\phi} \, 2^{2\ell |\alpha| (|\alpha| - 1)^{-1}} 2^{-2r} \, H_P^{-2} .
$$
Inserting this bound into \eqref{B-sum} with $N > 2 d/(d-1)$ implies that
$$
\bigl| \sum_{j\in B} I_j \bigr| \ \lesssim_{d,n,\phi} \ H_P^{-2}
$$
which, together with \eqref{A}, completes the proof of \eqref{poly-several-int} and hence Theorem \ref{H-main-osc}.

\section*{{\bf \large Proofs of Proposition \ref{sublevel-structure-basic} and Lemma \ref{hensel-L}}}

Here we give the proofs of Proposition \ref{sublevel-structure-basic} and Lemma \ref{hensel-L}.

%We finish these notes with a proof of Proposition \ref{sublevel-structure-basic} and Lemma \ref{hensel-L}.

\section{Proof of Proposition \ref{sublevel-structure-basic}}\label{Prop-8}

Let us recall the statement of Proposition \ref{sublevel-structure-basic}. Let $Q \in {\mathbb C}[X]$
have degree $d$.
Let $z_0 \in {\mathbb D}$ and $0 < \epsilon \le \epsilon_d$ for a sufficiently small $\epsilon_d>0$.
Suppose 
$|Q^{(k)}(z_0)|\ge 1$ and $|Q(z_0)| \le \epsilon$. Our aim is to find a zero $z_{*}$
of $Q^{(j)}$ for some $0\le j \le d$ such $|z_0 - z_{*}| \lesssim_d  \epsilon^{1/k}$.

We begin with the following lemma.

\begin{lemma}\label{triple-norm}
Let ${\mathcal P}_d \subseteq {\mathbb C}[X]$ denote the space of complex polynomials of degree at most $d$.
For $P(z) = c_d z^d + \cdots + c_1 z + c_0 \in {\mathcal P}_d$, define $\|P\| = \max_j |c_j|$.
Then there exists a constant $c_d >0$ depending only on $d$ such that 
for every $P \in {\mathcal P}_d$, 
\begin{equation}\label{poly-below}
|P^{(n)}(z)| \ \ge \ c_d \|P\| \ \ {\rm for \ all} \ z \in {\mathbb D}
\end{equation}
holds for some $0\le n \le d$.
\end{lemma}

\begin{proof}
Consider the ``norm'' 
$$
||| P ||| \ := \ \max_{j\ge 0} \min_{z \in {\mathbb D}} |P^{(j)}(z)|
$$
which satisfies (1) if $|||P||| = 0$, then $P=0$ and (2) $||| \lambda P ||| = |\lambda| |||P|||$
but it does not satisfy the triangle inequality. Nevertheless one can run the usual {\it equivalence of norms}
argument with this ``norm''. 

Let $S = \{P \in {\mathcal P}_d : \|P\| =1 \}$ denote the unit sphere in ${\mathcal P}_d$ with
respect to the norm $\|P\| = \max_j |c_j|$. We will show that
$$
c_d \ := \ \inf_{P \in S} ||| P ||| \ > \ 0
$$
which implies $||| P ||| \ \ge \ c_d \, \|P\|$ for all $P\in {\mathcal P}_d$ and this establishes \eqref{poly-below} by the
definition of this triple ``norm''  $||| \cdot |||$. 

If $c_d = 0$, then there exists a sequence $P_j \in S$ such that $|||P_j||| \to 0$
and $\|P_j - P\| \to 0$ for some $P \in S$. 
Although the triangle inequality does not hold for $||| \cdot |||$ we do have the inequality
$$
|||P||| \ \le \ |||P_j||| + \max_n  ||P^{(n)}_j - P^{(n)}||_{L^{\infty}({\mathbb D})} \ \le \
|||P_j ||| + C_d \| P_j - P\| \ \to \ 0,
$$
implying $|||P||| = 0$ and so $P = 0$ but this is impossible since $P \in S$.
\end{proof}

We turn our attention to our polynomial $Q(z) = c_0 + c_1 z + \cdots + c_d z^d$ in
Proposition \ref{sublevel-structure-basic} satisfying $|Q^{(k)}(z_0)| \ge 1$ and $|Q(z_0)| \le \epsilon$
for some $z_0 \in {\mathbb D}$. Set $\lambda = \max_j |c_j|$ and note that
$1 \lesssim_d |Q^{(k)}(z_0)| \lesssim_d \lambda$ implies $1 \lesssim_d \lambda$. We will
use Lemma \ref{hensel-L} to prove Proposition \ref{sublevel-structure-basic}. Since 
$\|Q\|_{L^{\infty}({\mathbb D}_2)} \le C_d \lambda$, it suffices to use $\lambda$ in place of $M$
when verifying the conditions of Lemma \ref{hensel-L}.

By Lemma \ref{triple-norm}, there exists an $0\le n \le d$ such that
$|Q^{(n)}(z)/n!| \ge c_d \lambda$. Since $\lambda \gtrsim_d 1$ and $\epsilon \ll_d 1$ is small,
then in fact it must be the case that $1\le n \le d$. Furthermore
$|Q^{(j)}(z)| \lesssim_d |Q^{(n)}(z)|$ for every $0\le j \le d$
and every $z \in {\mathbb D}$.
In particular when $k>n$, we have
\begin{equation}\label{k>n}
1 \ \le \ \bigl[\epsilon^{-1} |Q^{(k)}(z)| \bigr]^{1/k} \ \lesssim_d \ \bigl[ \epsilon^{-1} Q^{(n)}(z) \bigr]^{1/n},
\end{equation}
a bound we will use in the proof of Proposition \ref{sublevel-structure-basic}.

We fix small constants $c_1, c_2, \ldots, c_{n-1}$ (we define $c_n =1$), depending only on $d$, satsifying the relationships
\begin{equation}\label{c's}
c_j^2 \ \ll \ c_{j-1} c_{j+1}
\end{equation}
for every $2\le j \le n-1$. For each $2\le j \le n-1$, we will introduce small parameters $\rho_{\ell} = \rho_{\ell}(j), 2\le \ell \le j$ 
below satisfying \eqref{rho's}. We first choose the small parameter $c_j$'s satisfying \eqref{c's}, then we choose
the parmeters $\rho_{\ell}$ satisfying \eqref{rho's} 
and finally we choose $\epsilon_d$, depending on all these other parameters,
and assume $\epsilon < \epsilon_d$. 

We set
$$
K_j(z) \ := \  \bigl[c_{n-j} \epsilon^{-1} |Q^{(n-j)}(z)| \bigr]^{1/(n-j)} \ \ {\rm and} \ \ K(z) \ := \ \max_{0\le j \le n-1} K_j(z)
$$
%and decompose 
%$$
%S \ := \ \big\{ z \in {\mathbb D}: |P^{(k)}(z)| \ge 1, \, |P(z)| \le \epsilon \bigr\} \ = \ \bigcup_{j=0}^{m-1} S_j
%$$
%into $m$ sets where
%$$
%S_j \ := \ \bigl\{ z \in S : H_j(z) = H(z) \bigr\}.
%$$
and note that
\begin{equation}\label{K-k}
\epsilon^{-1/k} \ \le \ \bigl[ \epsilon^{-1} |Q^{(k)}(z_0)| \bigr]^{1/k} \ \lesssim_d \ K(z_0).
\end{equation}
This is clearly true if $k\le n$ and when $k>n$, this follows from \eqref{k>n}.

We consider several cases.

{\bf Case 0}:  $K(z_0) = K_0(z_0) = \bigl[ \epsilon^{-1} |Q^{(n)}(z_0)| \bigr]^{1/n}$. 

In this case, we apply Lemma \ref{hensel-L} with $L=1$ to $\phi = Q^{(n-1)}$ and so \\
$\delta = |Q^{(n-1)}(z_0) Q^{(n)}(z_0)^{-2}|$.
Hence
$$
\lambda \delta \ \lesssim_d  \ |Q^{(n-1)}(z_0)| |Q^{(n)}(z_0)|^{-1} \le c_{n-1}^{-1} 
\bigl[\epsilon^{-1}|Q^{(n)}(z_0)|\bigr]^{-1/n} 
$$
and so by \eqref{K-k}, we have 
$$
\lambda \delta \ \lesssim_d \  \bigl[
 \epsilon^{-1} |Q^{(n)}(z_0)|^{-1/n} \ \lesssim_d \ \epsilon^{1/k} |Q^{(k)}(z_0)|^{-1/k} \ \ll \ 1
$$
since $|Q^{(k)}(z_0)| \ge 1$ and $\epsilon < \epsilon_d$ is small. Hence Lemma \ref{hensel-L} shows there exists a zero $z_{*}$ of $Q^{(n-1)}$ such that
$$
|z_0 - z_{*}| \ \le \ 2 |Q^{(n-1)}{(z_0)} Q^{(n)}(z_0)^{-1}| \ \le \ 2 c_{n-1}^{-1} 
\bigl[\epsilon^{-1} |Q^{(n)}(z_0)|\bigr]^{-1/n}
$$
and so by \eqref{K-k}, we have
$$
|z_0 - z_{*}| \ \lesssim_d \ \bigl[
 \epsilon^{-1} |Q^{(n)}(z_0)|^{-1/n} \ \lesssim_d \ \epsilon^{1/k} |Q^{(k)}(z_0)|^{-1/k} \ \le \ \epsilon^{1/k}
$$
since $|Q^{(k)}(z_0)| \ge 1$. This completes the proof in Case 0.

{\bf Case 1}: $K(z_0) = K_1(z_0) = \bigl[c_{n-1} \epsilon^{-1} |Q^{(n-1)}(z_0)|\bigr]^{1/(n-1)}$.

We apply Lemma \ref{hensel-L} with $L=2$ to $\phi = Q^{(n-2)}$ so that 
$$
\delta \ = \ |Q^{(n-2)}(z_0) Q^{(n-1)}(z_0)^{-1} Q^{(n)}(z_0)^{-1}|.
$$
Hence 
$$
\lambda \delta \ \lesssim_d \ |\epsilon^{-1}Q^{(n-1)}(z_0)|^{-1/(n-1)} \ \lesssim_d  \ 
\epsilon^{1/k} |Q^{(k)}(z_0)|^{-1/k} 
$$
by \eqref{K-k}. Hence $\lambda \delta \ll 1$ since $|Q^{(k)}(z_0)|\ge 1$ and
$\epsilon < \epsilon_d$ is small. For Lemma \ref{hensel-L}, we also need to verify that $\delta_1 \ll 1$ where
$\delta_1 = |Q^{(n-2)}(z_0) Q^{(n-1)}(z_0)^{-2} Q^{(n)}(z_0)|$. We have
$$
\delta_1 \ \le \ \frac{c_{n-1}^{2}}{c_{n-2}} \, \frac{\epsilon}{\epsilon^{(n-2)/(n-1)}} 
\frac{\epsilon}{\epsilon^{n/(n-1)}} |Q^{(n-1)}(z_0)|^{\frac{n-2}{n-1} + \frac{n}{n-1} - 2} \ = \ 
c_{n-1}^{2} c_{n-2}^{-1} \ \ll \ 1
$$
by \eqref{c's} applied to $j = n-1$ (recall $c_n = 1$).
Therefore Lemma \eqref{hensel-L} implies there is a zero $z_{*}$ of $Q^{(n-2)}$
such that
$$
|z_0 - z_{*}| \ \le \ 2 |Q^{(n-2)}(z_0) Q^{(n-1)}(z_0)^{-1}| \ \lesssim_d \ \bigl[\epsilon^{-1} |Q^{(n-1)}(z_0)|\bigr]^{-1/(n-1)}
$$
and so by \eqref{K-k},
$$
|z_0 - z_{*}| \ \lesssim_d \ \bigl[\epsilon^{-1} |Q^{(k)}(z_0)|\bigr]^{-1/k} \ \lesssim_d \ \epsilon^{1/k}
$$
since $|Q^{(k)}(z_0)| \ge 1$. The completes the proof in Case 1.

{\bf General case j}: \ $K(z_0) = K_j(z_0)$.  Here $2\le j \le n$.

We split the general case $j$ into $j$ subcases. To define these subcases,
%For each $2\le j \le m-1$, we decompose $S_j = S_j^1 \cup \cdots \cup S_j^j$ further into $j$ sets. To define
we introduce the following conditions:
$$
|Q^{(n-j +\ell)}(z_0)| \,  |Q^{(n-j+\ell -1)}(z_0)|^{-1} \ \le \ \rho_{\ell}^{-1} \, K_j(z_0) 
\eqno{{\mathcal S}_{\ell}}
$$
for $2\le \ell \le j$ where $\rho_{\ell}$ are small constants satisfying the relationships
\begin{equation}\label{rho's}
\rho_{\ell} \ \ll \ \rho_{\ell+1}
\end{equation}
for every $2\le \ell \le j-1$.

$$
{\rm \bf subcase} \ 1: \ \ \  \ \  (I_1): \ \ K_j(z_0) \, |Q^{(n-1)}(z_0)| \ \le \ \rho_j \, |Q^{(n)}(z_0)|. 
$$

If $z_0$ does not satisfy $(I_1)$, then the property ${\mathcal S}_j$ holds. Inductively the other subcases
for $2\le \ell \le j-1$ are defined by

\hskip 35pt {\rm \bf subcase \ $\ell$}: \ \ $(I_{\ell-1})$ \ does  not hold but
$$
\ \ \ \ \ (I_{\ell}): \ \ 
K_j(z_0) \, |Q^{(n-\ell)}(z_0)| \ \le \ \rho_{j-\ell+1} \, |Q^{(n-\ell+1)}(z_0)|
$$
holds.

If $(I_{\ell})$ does not hold, then the properties ${\mathcal S}_{j-r}$ hold
for every $0\le r \le \ell -1 $. The final subcase is

\hskip 35pt {\rm \bf subcase \ $j$}: \ \ $(I_{j-1})$ \ does  not hold.

In this last subcase  $j$, all the  
properties ${\mathcal S}_{\ell}$, for $2\le \ell \le j$ hold.

Let us consider this final subcase. We apply Lemma \ref{hensel-L} with $L = (j+1)$ to 
$\phi = Q^{(n-j-1)}$ so that
$\delta = |Q^{(n-j-1)}(z_0) Q^{(n-j)}(z_0)^{-1} Q^{(n)}(z_0)^{-1}|$. Also for $1\le \ell \le j$, we have
$$
\delta_{\ell} \ := \ |Q^{(n-j-1)}(z_0) Q^{(n-j)}(z_0)^{-1} Q^{(n-j+\ell -1)}(z_0)^{-1} Q^{(n-j+\ell)}(z_0)|.
$$
We need to verify $\lambda \delta, \, \delta_1 \ll 1$ and for $2\le \ell \le j$, $\delta_{\ell} \lesssim_d 1$.

First, we have
$$
\lambda \delta \ \lesssim_d \  \bigl[\epsilon^{-1} |Q^{(n-j)}(z_0)|\bigr]^{-1/(n-j)} \,
  \ \lesssim_d  \ \epsilon^{1/k} |Q^{(k)}(z_0)|^{-1/k} 
$$
by \eqref{K-k} and so
$$
\lambda \delta \ \lesssim_d \ \epsilon^{1/k} \ \ll \ 1
$$
since $|Q^{(k)}(z_0)| \ge 1$ and $\epsilon < \epsilon_d$ is small. Also 
$$
\delta_1 \le 
c_{n-j-1}^{-1} c_{n-j}^{\frac{n-j-1}{n-j}} c_{n-j+1}^{-1} c_{n-j}^{\frac{n-j+1}{n-j}} \epsilon^{1/(n-j)} \epsilon^{-1/(n-j)}
|Q^{(n-j)}(z_0)|^{\frac{n-j-1}{n-j} + \frac{n-j+1}{n-j} - 2}
$$
$$
= \ \ \  c_{n-j-1}^{-1} c_{n-j+1}^{-1} c_{n-j}^{2} \ \ll \ 1
$$
by \eqref{c's}. Finally for every $2\le \ell \le j$, since the properties ${\mathcal S}_{\ell}$ for for every $2\le \ell \le j$ hold,
we have
$$
\delta_{\ell} \ \lesssim_d \  |Q^{(n-j-1)}(z_0) Q^{(n-j)}(z_0)^{-1}| \, K_j(z_0) \ \lesssim_d \ 1.
$$
Hence there exists a zero $z_{*}$ of $Q^{(n-j-1)}$ such that
$$
|z_0 - z_{*}| \ \le \ 2 |Q^{(n-j-1)}(z_0) Q^{(n-j)}(z_0)^{-1}| \ \lesssim_d 
\bigl[\epsilon^{-1} 
|Q^{(n-j)}(z_0)|\bigr]^{-1/(n-j)} 
$$
$$ 
\lesssim_d \
\epsilon^{1/k} |Q^{(k)}(z_0)|^{-1/k} \ \lesssim_d \ \epsilon^{1/k}
$$
by \eqref{K-k} and the fact that $|Q^{(k)}(z_0)| \ge 1$.

Finally we consider the subcases $\ell$ with $2\le \ell \le j-1$. In these subcases,
we see that the properties ${\mathcal S}_{j-r}$ hold for
$0\le r \le \ell -2$. Here we apply Lemma \ref{hensel-L} with $L = \ell$ to $\phi = Q^{(n-\ell)}$
so that 
$$
\delta \ = \ |Q^{(n-\ell)}(z_0) Q^{(n-\ell +1)}(z_0)^{-1} Q^{(n)}(z_0)^{-1}|.
$$
Also for
$1\le t \le \ell -1$,
$$
\delta_t \ = |Q^{(n-\ell)}(z_0) Q^{(n-\ell +1)}(z_0)^{-1} Q^{(n-\ell + t)}(z_0)^{-1} Q^{(n-\ell + t + 1)}(z_0)|.
$$
We need to verify $\lambda \delta, \, \delta_1 \ll 1$ and for $2\le t \le \ell -1$, $\delta_{t} \lesssim_d 1$.
We have
$$
\lambda \delta \ \lesssim_d \  \bigl[\epsilon^{-1} |Q^{(n-j)}(z_0)|\bigr]^{-1/(n-j)} 
 \ \lesssim_d  \ \epsilon^{1/k} |Q^{(k)}(z_0)|^{-1/k} \  \ll \ 1
$$
by \eqref{K-k}, our hypothesis $|Q^{(k)}(z_0)|\ge 1$ and 
since $\epsilon < \epsilon_d$ is small. Also since ${\mathcal S}_{j-\ell +2}$ holds, we have
$$
\delta_1 \le 
\rho_{j-\ell +1} K_j(z_0)^{-1} |Q^{(n-\ell+2)}(z_0) Q^{(n-\ell +1)}(z_0)^{-1}|  \le  \frac{\rho_{j-\ell+1}}{\rho_{j-\ell +2}} \, K_j(z_0)^{-1}
K_j(z_0) \ll 1
$$
by \eqref{rho's}. Also since ${\mathcal S}_{j-\ell +t+1}$ holds for $2\le t \le \ell -1$,
$$
\delta_t \ \lesssim_d \ 
 K_j(z_0)^{-1} |Q^{(n-\ell+t+1)}(z_0) Q^{(n-\ell +t)}(z_0)^{-1}|  \ \lesssim_d \  K_j(z_0)^{-1}
K_j(z_0) \ = \  1.
$$
Hence there exists a zero $z_{*}$ of $Q^{(n-\ell)}$ such that
$$
|z_0 - z_{*}| \ \le \ 2 |Q^{(n-\ell)}(z_0) Q^{(n-\ell +1)}(z_0)^{-1}| \ \lesssim_d \ K_j(z_0)^{-1} \ \le \ \epsilon^{1/k} |Q^{(k)}(z_0)|^{-1/k}
\ \le \ \epsilon^{1/k}
$$
by \eqref{K-k} and from our hypothsis $|Q^{(k)}(z_0)| \ge 1$. 

This completes the proof of Proposition \ref{sublevel-structure-basic}.

\section{Proof of Lemma \ref{hensel-L}}\label{hensel}

It remains to prove Lemma \ref{hensel-L} which we restate for convenience.

{\bf Lemma 9.3} \ 
{\it Let $\phi \in {\mathcal H}({\mathbb D}_2)$ with $M = M_{\phi} := \max_{z\in {\mathbb D}_{7/4}} |\phi(z)|$. Fix $L\ge 1$.
Suppose $z_0 \in {\mathbb D}$ is a point where $\phi^{(k)}(z_0) \not= 0$  for each $1\le k \le L$. 
Set \\
$\delta := |\phi(z_0)\phi'(z_0)^{-1}\phi^{(L)}(z_0)^{-1}|$ and for $1\le k \le L-1$, set
$$
\delta_k \ := \ |\phi^{(k+1)}(z_0) \phi^{(k)}(z_0)^{-1} \phi(z_0) \phi'(z_0)^{-1}|.
$$
Suppose $\delta_k \lesssim_L 1,  \, 2\le k \le L-1$ and suppose $|\phi(z_0)\phi'(z_0)^{-1}| \le 1/8$.
Then if $\delta M, \, \delta_1 \ll_L 1$,  there is a $z\in {\mathbb D}_{5/4}$ such that
$$
(a) \ \phi(z) \ = \ 0 \ \ \ {\rm and} \ \ \ (b) \ |z-z_0| \ \le \  2 |\phi(z_0)\phi'(z_0)^{-1}|.
$$}

We will only give the proof in the case $L\ge 2$. The case $L=1$ is easier.

We define a sequence $\{z_n\}$ recursively by
\begin{equation}\label{recursive}
z_n \ = \ z_{n-1} - \phi(z_{n-1}) \phi'(z_{n-1})^{-1}
\end{equation}
and set $\Lambda := c_L \delta_1$ for some large constant $c_L$ chosen later. Then our condition
$\delta_1 \ll_L 1$ will imply in particular $\Lambda \le 1/2$. We make the
following claim.

{\bf Claim:} \ For every $n\ge 1$,

$(1)_n \ \ |z_n - z_{n-1}| \ \le \ |\phi(z_{0}) \phi'(z_{0})^{-1}| \, \Lambda^{2^{n-1} - 1}, $

$(2)_n \ \ |\phi(z_{n-1})| \ \le \ |\phi(z_0)| \Lambda^{2^{n-1}-1}, $

$(3)_n \ |\phi'(z_{n-1})| \ \ge \ (1 - \epsilon_{n-1}) \, |\phi'(z_0)|,$ 

$(4)_n$ \ for $2\le k \le L$, $|\phi^{(k)}(z_{n-1})| \le (1 + \varepsilon_{n-1}) |\phi^{(k)}(z_0)|$

where for $n\ge 2$, $\epsilon_{n-1} = \sum_{j=2}^{n} 2^{-j}$ and $\varepsilon_{n-1} = \sum_{j=1}^{n-1} 2^{-j}$.
When $n=1$, we set $\epsilon_0 = \varepsilon_0 = 0$.
Note the claim implies that for every $n\ge 1$,
\begin{equation}\label{sequence-II}
|z_n - z_0| \ \le \ \sum_{j=1}^n |z_j - z_{j-1}| \ \le \ |\phi(z_0)\phi'(z_0)^{-1}| \sum_{j\ge 0} 2^{-j} \ \le \ 2/8
\end{equation}
and hence $z_n \in {\mathbb D}_{5/4}$.

Before we start the proof of the claim, we review Taylor series with remainder for a $\psi \in {\mathcal H}({\mathbb D}_2)$: 
for any $m\ge 0$,
\begin{equation}\label{TS}
\psi(z_n) \ = \ \psi(z_{n-1}) \ + \ \sum_{j=1}^m \frac{\psi^{(j)}(z_{n-1})}{j!} (z_n - z_{n-1})^j \ + \ R_{m,\psi}
\end{equation}
where $z_t := z_{n-1} + t(z_n - z_{n-1})$ and
$$
R_{m, \psi} \ := \ 
\frac{1}{m!} \int_0^1 \psi^{(m+1)}(z_t) \, dt \ (z_n - z_{n-1})^{m+1}.
$$
The sum in \eqref{TS} does not appear when $m=0$.
We will have need to apply \eqref{TS} for $\psi(z) = \phi^{(k)}(z)$ where $|\phi(z)| \le M$ for $z\in {\mathbb D}_{7/4}$.
Since $z_n \in {\mathbb D}_{5/4}$ for all $n$, then $z_t \in {\mathbb D}_{5/4}$ and so 
by Cauchy's integral formula,
$$
\psi^{(m)}(z_t) \ = \ \frac{(m+k)!}{2\pi i} \int_{C_{1/2}(z_t)} \frac{\phi(w)}{(w - z_t)^{m+k +1}} \, dw,
$$
implying $|\psi^{(m)}(z_t)| = |\phi^{(m+k)}(z_t)| \lesssim_{m,k} M$ and hence
\begin{equation}\label{remainder}
|R_{m,k} | \ \lesssim_{m,k} M \, |z_n - z_{n-1}|^{m+1}
\end{equation}
where $R_{m,k} = R_{m,\psi}$ with $\psi =  \phi^{(k)}$.

We now proceed with the proof of the claim.
The claim is true for $n=1$ and so suppose $(1)_j, (2)_j$ and $(3)_j$ holds for all $1\le j \le n$.
Note that $(3)_n$ implies $\phi'(z_{n-1}) \not= 0$ and so $z_n$ is well-defined. 
We begin with proving $(4)_{n+1}$. Applying \eqref{TS} to $\psi = \phi^{(k)}$ for $2\le k \le L$ and $m=L-k$, we see
\begin{equation}\label{TS-II}
\phi^{(k)}(z_n) \ = \ \phi^{(k)}(z_{n-1}) + \sum_{j=1}^{L-k} \frac{\phi^{(k+j)}(z_{n-1})}{j!} (z_n - z_{n-1})^j \ + \ R_{L-k,k}
\end{equation}
where
by the remainder estimate in \eqref{remainder}, 
$$
|R_{L-k,k}| \ \lesssim_L \ M |z_n - z_{n-1}|^{L-k+1} \ \lesssim_L M |\phi(z_0)\phi'(z_0)^{-1}| \Lambda^{2^{n-1}-1} |z_n - z_{n-1}|^{L-k}
$$
\begin{equation}\label{remainder-II}
\ \lesssim_L \ M \delta \Lambda^{2^{n-1}-1} |\phi^{(L)}(z_0)|  |z_n - z_{n-1}|^{L-k} .
\end{equation}
Furthermore by $(1)_n$ and $(4)_n$, for each $1\le j \le L-k$, 
$$
|\phi^{(k+j)}(z_{n-1}) (z_n - z_{n-1})^j | \ \lesssim_L \ \Lambda^{2^{n-1}-1} |\phi(z_0)\phi'(z_0)^{-1}|^j |\phi^{(k+j)}(z_0)| .
$$
By the definition of the $\delta_{\ell}, \, 2 \le \ell \le L-1$ in the statement of Lemma \ref{hensel-L}, we have
$$
|(\phi(z_0)\phi'(z_0)^{-1})^j \phi^{(k+j)}(z_0)| \ = \ |(\delta_k \cdots \delta_{k+j-1}) \phi^{(k)}(z_0)| \ \lesssim_L |\phi^{(k)}(z_0)|
$$
for $1\le j \le L-k$ and hence
\begin{equation}\label{k+j}
|\phi^{(k+j)}(z_{n-1}) (z_n - z_{n-1})^j | \ \lesssim_L \ \Lambda^{2^{n-1}-1}  |\phi^{(k)}(z_0)| .
\end{equation}

Plugging \eqref{remainder-II} and \eqref{k+j} into \eqref{TS-II}, we see that 
$$
|\phi^{(k)}(z_n)| \ \le \ |\phi^{(k)}(z_0)| \bigl[1 + \varepsilon_{n-1} + b_L \Lambda^{2^{n-1}-1} \bigr] 
$$
by $(4)_n$. But $b_L \Lambda^{2^{n-1}-1} \ll 2^{-2(n-1)} \le 2^{-n}$ and so
$$
|\phi^{(k)}(z_n)| \ \le \ |\phi^{(k)}(z_0)| \bigl[1 + \varepsilon_{n-1} + 2^{-n} \bigr]  \ = \ (1 + \varepsilon_n) |\phi^{(k)}(z_0)|,
$$
establishing $(4)_{n+1}$.

We now turn to $(3)_{n+1}$. Applying \eqref{TS} to $\psi = \phi'$ and $m=L-1$, we see
\begin{equation}\label{TS-III}
\phi'(z_n) \ = \ \phi'(z_{n-1}) + \sum_{j=1}^{L-1} \frac{\phi^{(j+1)}(z_{n-1})}{j!} (z_n - z_{n-1})^j \ + \ R_{L,k}
\end{equation}
where
by the remainder estimate in \eqref{remainder}, 
$$
|R_{L,k}| \ \lesssim_L \ M |z_n - z_{n-1}|^{L} \ \lesssim_L M |\phi(z_0)\phi'(z_0)^{-1}| \Lambda^{2^{n-1}-1} |z_n - z_{n-1}|^{L-1}
$$
\begin{equation}\label{remainder-III}
\ \lesssim_L \ M \delta \Lambda^{2^{n-1}-1} |\phi^{(L)}(z_0)|  |z_n - z_{n-1}|^{L-1} .
\end{equation}
Furthermore by $(1)_n$ and $(4)_n$, for each $1\le j \le L-1$, 
$$
|\phi^{(j+1)}(z_{n-1}) (z_n - z_{n-1})^j | \ \lesssim_L \ \Lambda^{2^{n-1}-1} |\phi(z_0)\phi'(z_0)^{-1}|^j |\phi^{(j+1)}(z_0)| .
$$
Again using the definition of the $\delta_{\ell}, \, 2\le \ell \le L-1$,
$$
|(\phi(z_0)\phi'(z_0)^{-1})^j \phi^{(j+1)}(z_0)| \ = \ |(\delta_2 \cdots \delta_{j}) \phi(z_0)\phi'(z_0)^{-1}\phi''z_0)| \ \lesssim_L \
\delta_1 |\phi'(z_0)|.
$$
for $1\le j \le L-1$ and hence
\begin{equation}\label{1+j}
|\phi^{(j+1)}(z_{n-1}) (z_n - z_{n-1})^j | \ \lesssim_L \ \Lambda^{2^{n-1}-1}  \delta_1 |\phi'(z_0)| .
\end{equation}
There is a similar bound for the right hand side of \eqref{remainder-III} and using this
and \eqref{1+j} in \eqref{TS-III}, we see that 
$$
|\phi'(z_n)| \ \ge \ |\phi'(z_0)| \bigl[1 - \epsilon_{n-1} - b_L  \delta_1 \Lambda^{2^{n-1}-1}\bigr] 
$$
by $(3)_n$. But $b_L \delta_1 \Lambda^{2^{n-1}-1} \le b_L \delta_1 (1/2)^{n-1} \le 2^{-n-1}$ since we'll take
$\delta_1$ small so that $b_L \delta_1 \le 1/4$. Therefore
$$
|\phi(z_n)| \ \ge \ |\phi'(z_0)| \bigl[1 - \epsilon_{n-1} - 2^{-n-1}\bigr]  \ = \ |\phi'(z_0)| \bigl[ 1 - \epsilon_n \bigr],
$$
completing the proof of $(3)_{n+1}$.

For $(2)_{n+1}$, we apply \eqref{TS} with $\psi = \phi$ and $m=L$ to conclude
$$
\phi(z_n)  \ = \ \phi(z_{n-1}) + \phi'(z_{n-1}) (z_n - z_{n-1}) + \sum_{j=2}^L \frac{\phi^{(j)}(z_{n-1})}{j!} \ (z_n - z_{n-1})^j \ + \
R_L
$$
and since $\phi(z_{n-1}) + \phi'(z_{n-1}) (z_n - z_{n-1}) = 0$ by definition of $z_n$, we have
$$
\phi(z_n)  \ = \ \sum_{j=2}^L \frac{\phi^{(j)}(z_{n-1})}{j!} \ (z_n - z_{n-1})^j \ + \ R_L
$$
where 
\begin{equation}\label{remainder-IV}
|R_L| \ \lesssim_L \ M |z_n - z_{n-1}|^{L+1} \ \lesssim_L \ M \delta \Lambda^{2^{n-1}-1} |\phi^{(L)}(z_0)|  |z_n - z_{n-1}|^L 
\end{equation}
as before. For $2\le j \le L$, we have by $(4)_n$,
$$
|\phi^{(j)}(z_{n-1})| |z_n - z_{n-1}|^j \le 2 |\phi^{(j)}(z_0)| |\phi(z_0) \phi'(z_0)^{-1}|^{j-1} \Lambda^{2^n - 2} |\phi(z_0)\phi'(z_0)^{-1}| .
$$
Proceeding as above, using the definition of the $\delta_{\ell}$'s, we have
$$
|\phi^{(j)}(z_0)| |\phi(z_0) \phi'(z_0)^{-1}|^{j-1}|  \ = \ (\delta_2\cdots \delta_{j-1}) |\phi''(z_0) \phi(z_0) \phi'(z_0)^{-1}|,
$$
implying
$$
|\phi^{(j)}(z_{n-1})| |z_n - z_{n-1}|^j \ \lesssim_L  \delta_1 \Lambda^{2^n - 2} |\phi(z_0)|.
$$
We have a similar estimate for the right hand side of \eqref{remainder-IV} and so, althogether, we have
$$
|\phi(z_n)| \ \stackrel{(*)}{\le} b_L \delta_1 \Lambda^{2^n -2} |\phi(z_0)| \ \le \ \Lambda^{2^n -1} |\phi(z_0)|,
$$
completing the proof of $(2)_{n+1}$ by ensuring $b_L \le c_L$ (recall $\Lambda = c_L \delta_1$).

Finally for $(1)_{n+1}$, we use $(2)_{n+1}$,$(3)_ {n+1}$ and the inequality $(*)$ above to see
$$
|z_{n+1} - z_n| \ = \ |\phi(z_n) \phi'(z_n)^{-1}| \ \le \ 2 |\phi'(z_0)^{-1} \phi(z_n)| \ \le \ 2 b_L \delta_1 \Lambda^{2^n -2} |\phi(z_0)|.
$$
By taking $c_L \ge 2 b_L$, we see that $|z_{n+1} - z_n| \le |\phi(z_0)| \Lambda^{2^n -1}$ which establishes $(1)_{n+1}$,
completing the proof of the claim.

Statement $(1)_n$ of the claim implies that for $m\le n$,
$$
|z_n - z_m| \ \le \ \sum_{j=m+1}^n |z_j - z_{j-1}| \ \le \delta |\phi'(z_0)| \sum_{j\ge m} (1/2)^{2^{j-1} -1} \ \to \ 0 
$$
as $m\le n \to \infty$ and hence $\{z_n\}$ forms a Cauchy sequence of complex numbers and hence $z_n \to z$ for some
$z \in {\mathbb D}_{5/4}$. The statement $(2)_n$ then implies $\phi(z) = 0$ and \eqref{sequence-II} shows
that $z \in {\mathbb D}_{5/4}$ and $|z - z_0| \le 2 |\phi(z_0) \phi'(z_0)^{-1}|$. 
This completes the proof of Lemma \ref{hensel-L}.

\section*{{\bf \large Some applications}}

We give two applications of the main estimate in Theorem \ref{H-main-osc} which are contained
in Propositions \ref{PS-intro} and \ref{ACK-intro}. The first is a complex version
of an oscillatory integral bound for polynomial phases due to Phong and Stein in \cite{PS-I}. Second, we prove
a complex version of a result of Arkhipov, Chubarikov and Karatsuba \cite{ACK-1} on the convergence exponent for
the singular integral in Tarry's problem which is equivalent to the $L^q$ integrability of the fourier extension
operator $E b$ on the function $b = 1$ with respect to complex curves. 
%$z \to (z_1, \ldots, z_d)$ which we can view as a 2-surface in ${\mathbb R}^{2d}$.

\section{Proof of Proposition \ref{PS-intro}}\label{Phong-Stein}

Let $f \in {\mathbb C}[X]$ have degree at most $d$. Consider the derivative
$f'(z) = a \prod_{j=0}^m (z - \xi_j)^{e_j}$ where $\{\xi_j\}$ are the distinct roots of
$f'$ with multiplicities $\{e_j\}$. Here we prove the complex version of a stable bound
for oscillatory integrals with polynomial phases due to Phong and Stein.

\begin{proposition}\label{PS-bound-complex} For any $f \in {\mathbb C}[X]$
and $\phi \in C^{\infty}_c({\mathbb C})$, we have
\begin{equation}\label{PS-bound-integral}
\Bigl| \int_{{\mathbb C}} {\rm e} (f(z)) \, \phi(z) \, dz \Bigr| \ \le \ C_{d,\phi} \
\max_{\xi} \ \min_{{\mathcal C}\ni \xi}
\Bigl[\frac{1}{| a \prod_{\xi_j \notin {\mathcal C}} (\xi - \xi_j)^{e_j}|} \Bigr]^{1/(S({\mathcal C}) + 1)}.
\end{equation}
\end{proposition}

\begin{proof}
By Theorem \ref{H-main-osc}, it suffices to prove 
$$
\min_{\xi} \, \max_{{\mathcal C}\ni \xi} \
\bigl[ |a \prod_{\xi_j\notin {\mathcal C}} (\xi - \xi_j)^{e_j}|\bigr]^{1/(S({\mathcal C}) + 1)} \ \le \
C_d \, H_f 
$$
where $H_f := \inf_{z \in {\mathbb C}} H_f(z)$ (note that $H_{f,\phi} \le H_f$). The above bound
in turn is implied by the following: for every $z_{*} \in {\mathbb C}$, there is a root $\xi$
of $f'$ such that 
\begin{equation}\label{PS-H-complex}
\bigl[ |a \prod_{\xi_j\notin {\mathcal C}} (\xi - \xi_j)^{e_j}|\bigr]^{1/(S({\mathcal C}) + 1)} \ \le \
C_d \, H_f(z_{*})
\end{equation} 
holds for every root cluster ${\mathcal C}$ containing $\xi$.

We fix $z_{*} \in {\mathbb C}$ and choose any root $\xi$ of $f'$ such that
$|z_{*} - \xi| = \min_j |z_{*} - \xi_j|$. Let ${\mathcal C}$ be any root cluster
containing $\xi$. Without loss of generality, suppose that $\xi = \xi_0$ and that the rest of roots 
are ordered so that
$$
|z_{*} - \xi| \ \le \ |z_{*} - \xi_1| \ \le \ |z_{*} - \xi_2| \ \le \ \cdots \ \le \ |z_{*} - \xi_m|.
$$
We fix a large constant $A>0$ to be determined later and set $j_0 = 0$. Let
$j_1 \ge 1$ be the smallest integer such that
$A |z_{*} - \xi_{j_1 -1}| \le |z_{*} - \xi_{j_1}|$. Next let $j_2 \ge j_1 + 1$ be the smallest integer
such that $A |z_{*} - \xi_{j_2 -1}| \le |z_{*} - \xi_{j_2}|$. And so on..., producing a sequence
$0 = j_0 < j_1 < \ldots < j_t$ of integers with
$$
A \, |z_{*} - \xi_{j_k -1}| \ \le \ |z_{*} - \xi_{j_k}| \ \ {\rm for \ every} \ \ 1\le k \le t
$$
and $|z_{*} - \xi_j| \le A |z_{*} - \xi_{j_t}|$ for every $t+1 \le j \le m$. Here $0\le t \le m$
where the $t=0$ case means $|z_{*} - \xi_j| \le A |z_{*} - \xi|$ for every $0\le j \le m$.

We split our root cluster ${\mathcal C}$ as
$$
{\mathcal C} \ = \ C_0 \cup C_1 \cup \cdots \cup C_t \ \ {\rm where} \ \ 
C_k \ = \ \{\xi_{j_k}, \ldots, \xi_{j_{k+1}-1}\} \cap {\mathcal C}
$$
for $0\le k \le t-1$ and $C_t = \{\xi_{j_t}, \ldots, \xi_m\} \cap {\mathcal C}$.
We set $f_k = \sum_{\xi_j \in C_k} e_{j}$ so that $S({\mathcal C}) = f_0 + \cdots + f_t$. 
Note that if $z_{*} = \xi$, we have $C_0 = \{\xi\}$ and so $f_0 = e_0$.

For each $0\le k \le t$, set
$$
F_k(z) \ = \ \prod_{j=j_k}^{j_{k+1}-1} (z - \xi_j)^{e_j} \ =: \ F_k^1(z) \, F_k^2(z)
$$
where
$$
F_k^1(z) \ = \ \prod_{\xi_j \in C_k} (z - \xi_j)^{e_{j}} \ \ {\rm and} \ \
F_k^2(z) \ = \ \prod_{\xi_j \notin C_k} (z - \xi_j)^{e_{j}} .
$$
Therefore $f'(z) = a \prod_{k=0}^t F_k(z)$ and if 
$Q_{{\mathcal C}} := | a \prod_{k=0}^t F_k^2(\xi)|$, then our goal is to prove
\begin{equation}\label{Q-aim}
Q_{{\mathcal C}}^{1/(S({\mathcal C}) +1)} \ \le \ C_d \, H_f(z_{*})
\end{equation}
which will establish \eqref{PS-H-complex}.

By the formula above for $f'$, we have
\begin{equation}\label{H-1}
H_f(z_{*}) \ \ge \
|f'(z_{*})| \ = \ |a \prod_{k=0}^t F_k^1(z_{*}) F_k^2(z_{*})| \ \gtrsim_d \ 
|z_{*} - \xi|^{S({\mathcal C})} Q_{\mathcal C}
\end{equation}
since for any root $\eta$ of $f'$, $|\xi - \eta| \le |z_{*} - \xi| + |z_{*}-\eta| \le 2 |z_{*}-\eta|$.

To derive other lower bounds for $H_f(z_{*})$ in terms of $Q_{\mathcal C}$,
we consider the derivatives $f^{(1+\rho_k)}$ of $f'$ where $\rho_k := \sigma_0 + \cdots + \sigma_k$
and $\sigma_k = \sum_{j=j_k}^{j_{k+1}-1} e_j$.
To do this, set 
$$
{\mathcal F}_k(z) \ = \ \prod_{\ell = k+1}^t F_{\ell}(z) 
$$ 
for each $0\le k \le t-1$ and note that
$$
f^{(1+ \rho_k)}(z)/\rho_k! \ = \ a {\mathcal F}_k(z) \ + \ a {\mathcal H}_k(z)
$$
where both ${\mathcal F}_k(z)$ and ${\mathcal H}_k(z)$ are homogeneous functions of degree $d-1 - \rho_k$
($d = {\rm deg}(f)$) in the variables $z-\eta$ as $\eta$ runs over the distinct roots of $f'$.
When $z_{*} = \xi$, we have $\rho_0 = \sigma_0 = e_{0}$ and so 
$f^{(1+e_{0})}(z)/e_{0}! = a {\mathcal F}_0(z)$ and ${\mathcal H}_0(z) = 0$.

When $z_{*} \not= \xi$,
each term in ${\mathcal H}_k(z_{*})$ has a factor $z_{*}-\xi_j$ for some $1\le j\le j_{k+1}-1$ and 
so $A |z_{*} - \xi_j| \le |z_{*} - \xi_{j_{k+1}}|$ which implies 
$|a {\mathcal H}_k(z_{*})| \le (1/2) |a {\mathcal F}_k(z_{*})|$ if $A$ is chosen large enough.
Therefore
$|f^{(1 + \rho_k)}(z_{*})/\rho_k!| \ge (1/2) |a {\mathcal F}_k(z_{*})|$.

Hence 
\begin{equation}\label{H-F}
H_f(z_{*})^{1+\rho_k} \ \gtrsim_d \ |f^{(1+\rho_k)}(z_{*})/\rho_k! | \ \gtrsim_d \ |a {\mathcal F}_k(z_{*})| 
\end{equation}
and if $f^k := \sigma_k - f_k$, then for all $0\le k \le t-1$,
$$
Q_{\mathcal C} \ \lesssim_d \ |z_{*} - \xi_{j_k}|^{f^0 + \cdots f^k} 
\frac{1}{|\prod_{\ell=k+1}^t F_{\ell}^1(z_{*})|} 
\ |a {\mathcal F}_k(z_{*})|
$$
$$
\lesssim_d \ |z_{*} - \xi_{j_k}|^{f^0 +\cdots + f^k} 
\frac{1}{|z_{*}- \xi_{j_{k+1}}|^{f_{k+1}+\cdots + f_t}} \
H_f(z_{*})^{1+\rho_k}.
$$
The first inequality follows from the fact that $|\xi - \xi_j| \le 2 |z_{*} - \xi_j|$ 
for every $j\ge 0$ and $|z_{*} - \xi_j| \le A |z_{*} - \xi_{j_k}|$ for every $j_k \le j \le j_{k+1}-1$.
The second inequality follows from \eqref{H-F}.
If $z_{*} = \xi$, then $f^0 = 0$ and we interpret $|z_{*}-\xi|^{f^0} = 1$
in the $k=0$ case.

Therefore since $|z_{*} - \xi_{j_k}| \le |z_{*} - \xi_{j_{k+1}}|$, we apply the above inequality for $Q_{\mathcal C}$
for $k$ and $k+1$ to conclude that
\begin{equation}\label{Q-k}
Q_{\mathcal C} \ \lesssim_d \ B_{k+1}^{\rho_k - S({\mathcal C})} H_f(z_{*})^{1+\rho_k} \ \ {\rm and} \ \ 
Q_{\mathcal C} \ \lesssim_d \ B_{k+1}^{\rho_{k+1} - S({\mathcal C})} H_f(z_{*})^{1+\rho_{k+1}}
\end{equation}
where $B_{k+1} = |z_{*} - \xi_{j_{k+1}}|$.
The first inequality with $k=-1$ incorporates \eqref{H-1} if we interpret $\rho_{-1} = 0$.

We now divide the analysis into cases depending on the size $S({\mathcal C})$ of ${\mathcal C}$.
Suppose $\rho_k < S({\mathcal C}) \le \rho_{k+1}$ for some
$-1\le k \le t-1$. Again with the interpretation that $\rho_{-1} = 0$, we see that any cluster
of roots ${\mathcal C}$ must have a size lying in one of these intervals. With $\rho_k < S({\mathcal C}) \le 
\rho_{k+1}$, we see that the first inequality in \eqref{Q-k} implies
$$
B_{k+1}^{S({\mathcal C})- \rho_k} \, Q_{\mathcal C} \ \lesssim_d \ H_f(z_{*})^{1+\rho_k} 
$$
and this implies \eqref{Q-aim} when $Q_{\mathcal C}^{-1/(S({\mathcal C}) +1)} \le B_{k+1}$
and therefore we may assume
\begin{equation}\label{assumption-1}
B_{k+1} \ \le \ Q_{\mathcal C}^{-1/(S({\mathcal C}) + 1)}.
\end{equation}
When $z_{*} = \xi$, the reduction to \eqref{assumption-1} when $k=-1$ is automatic.

The second inequality in \eqref{Q-k}, together with \eqref{assumption-1}, implies
$$
Q_{\mathcal C} \ \le \  B_{k+1}^{\rho_{k+1} - S({\mathcal C})} H_f(z_{*})^{1+\rho_{k+1}} \ \le \ 
Q_{\mathcal C}^{-(\rho_{k+1} - S({\mathcal C}))/(S({\mathcal C}) + 1)} \, H_f(z_{*})^{1+\rho_{k+1}}
$$
and this unravels to \eqref{Q-aim}, completing the proof of \eqref{PS-H-complex}.

\end{proof}

\section{The fourier transform of measures in ${\mathbb C}^d$}\label{Fourier-transform}

Recall that $t\to e^{2\pi i t}$ gives the basic character on ${\mathbb R}$. 
All other characters on ${\mathbb R}$ arise
from elements $s\in {\mathbb R}$,  $\chi_s(t) = e^{2\pi i s t}$. In the same way, 
starting with the basic character ${\rm e}(z) = e^{2\pi i x} e^{2\pi i y}$ where $z = x + i y$, the other
characters on
${\mathbb C} \simeq {\mathbb R}^2$
arise from elements $w = (u,v) \in {\mathbb R}^2$,
${\rm e}_w (z) = e^{2\pi i (x u + y v)}$. 
%We write $\langle z, w \rangle := xu + y v = (x,y)\cdot (u,v)$
%so that $e_w(z) = e^{2\pi i \langle z, w \rangle}$.

%to distinguish it from complex multiplication  

The fourier transform ${\mathcal F}(\sigma) = {\widehat{\sigma}}$ of a Borel measure $\sigma$ on ${\mathbb C}$ is defined as
$$
{\widehat{\sigma}}(\xi,\eta) \ = \ \int_{{\mathbb R}^2} e^{2\pi i [(\xi,\eta)\cdot (x,y)]} \, d\sigma(x,y) 
$$
where $(x,y)\cdot (\xi, \eta) = x \xi + y \eta$. 
Now we write this in complex notation using the transformation $T: {\mathbb C} \to {\mathbb R}^2$ defined
by $T(w) = (u+v, u - v)$ where $w = u + i v$. Note that if $z = x + iy$, then $z w = (xu - yv) + i (xv + yu)$ and so
${\rm Re}(zw) + {\rm Im}(zw) = x(u+v) + y(u - v) = (x,y)\cdot T w$. Hence, using the basic nonprincipal character $e$ on ${\mathbb C}$,
$$
{\widehat{\sigma}}(T w) \ = \ \int_{{\mathbb R}^2} e^{2\pi i ({\rm Re}(wz) + {\rm Im}(wz))} \, d\sigma(x,y)  \ = \ 
\int_{\mathbb C} e(wz) \, d\sigma(z) .
$$
It will be convenient for us to think of ${\mathcal F}\circ T$ as the more appropriate notion of the fourier transform.
%and so we set ${\widetilde{\sigma}}(w) := {\widehat{\sigma}}(T w)$. 

This discussion readily extends to ${\mathbb C}^d \simeq {\mathbb R}^{2d}$. A complex vector 
${\underline{\xi}} = (\xi_1, \ldots, \xi_d) \in {\mathbb C}^d$ can be viewed a real vector in ${\mathbb R}^{2d}$
by writing out the real and imaginary parts of each $\xi_j = \gamma_j + i \eta_j$ so that
${\underline{\xi}}_r = (\gamma_1, \eta_1, \ldots, \gamma_d, \eta_d) \in {\mathbb R}^{2d}$. If
$\sigma$ now denotes a Borel measure on ${\mathbb C}^d$, then for ${\underline{z}} = (z_1, \ldots, z_d) \in {\mathbb C}^d$
(or ${\underline{z}}_r = (x_1, y_1, \ldots, x_d, y_d) \in {\mathbb R}^{2d}$),
$$
{\widehat{\sigma}}({\underline{\xi}}) \ = \ \int_{{\mathbb R}^{2d}} e^{2\pi i [{\underline{\xi}}_r \cdot {\underline{z}}_r]} \, 
d\sigma({\underline{x}},{\underline{y}})
$$
and this again can be written in complex notation (using complex multiplication). For 
${\underline{z}}, {\underline{w}} \in {\mathbb C}^d$,
we write $\langle {\underline{z}}, {\underline{w}} \rangle = \sum_{j=1}^d z_j w_j$ which is a slight variant of the usual
hermitian inner product on ${\mathbb C}^d$ (the form being symmetric instead of being skew-symmetric). 
Extend the transformation $T$ above to ${\underline{T}} : {\mathbb C}^d \to {\mathbb R}^{2d}$
by defining
$$
{\underline{T}} ({\underline{w}}) \  = \ (T w_1, \ldots, Tw_d) \in {\mathbb R}^2 \times \cdots \times {\mathbb R}^2 
\ \simeq \ {\mathbb R}^{2d}.
$$
We have
$$
{\widehat{\sigma}}({\underline{T}} \, {\underline{w}}) \ = \ \int_{{\mathbb R}^{2d}} 
e^{2\pi i (\sum_{j=1}^d [{\rm Re}(w_j z_j) + {\rm Im}(w_j z_j)])} \, d\sigma({\underline{x}},{\underline{y}})  \ = \ 
\int_{{\mathbb C}^d} e(\langle w, z\rangle ) \, d\sigma({\underline{z}}) 
$$
and again, it will be more convenient to think of ${\widehat{\sigma}} \circ {\underline{T}}$ as the Fourier transform of 
$\sigma$.

An important class of Borel measures $\sigma$ which arise in euclidean harmonic analysis
is surface measure on some polynomially parameterised $n$-dimensional surface 
${\underline{x}} \in {\mathbb R}^n \to \Phi({\underline{x}}) \in {\mathbb R}^N$ where 
$\Phi({\underline{x}}) = (Q_1({\underline{x}}), \ldots, Q_N({\underline{x}}))$
is an $N$-tuple of polynomials $Q_j \in {\mathbb R}[X_1, \ldots, X_n]$.

One could consider problems over ${\mathbb R}$ and examine analogues over ${\mathbb C}$. For instance,
the fourier restriction problem with respect to the moment curve $\gamma_r: {\mathbb R} \to {\mathbb R}^d$
defined by $\gamma_r(t) = (t, t^2, \ldots, t^d)$ was 
solved (with a remarkable argument) by Drury in the 1980s \cite{Drury}. One could try to see if arguments for the fourier restriction
problem with respect to $\gamma_r$ extend to establishing fourier restriction estimates for the complex moment
curve $\gamma_c : {\mathbb C} \to {\mathbb C}^d$ defined by $\gamma_c(z) = (z, z^2, \ldots, z^d)$, considered as a $2$-surface 
in ${\mathbb R}^{2d} \simeq {\mathbb C}^d$. If $z = x + iy$, this $2$-surface is parameterised by 
$\Gamma(x,y) = (\phi_1(x,y), \psi_1(x,y), \ldots, \phi_d(x,y), \psi_d(x,y))$ where
$$
\phi_1(x,y) = x, \psi_1(x,y) = y, \phi_2(x,y) = x^2 - y^2, \psi_2(x,y) = 2 xy, \phi_3(x,y) = x^3 - 3xy^2, ...
$$
So this is a problem in {\it real} harmonic analysis.
%but one
%might think that certain arguments for the moment curve in ${\mathbb R}$ may still hold for the moment (complex) curve in $%{\mathbb C}$.

The problem of fourier restriction to complex curves has been studied by a number of authors; see for example,
\cite{Bak-Ham}, \cite{Chung} and more recently \cite{Dios} and \cite{Conor} where positive results have
been obtained. The question arises whether the results are sharp.

In the real case, it is easy to determine by a scaling
argument what is the necessary relationship between the Lebesgue exponents $p$ and $q$ when the restriction operator 
${\mathcal R} : L^p \to L^q$
to the curve $\gamma_r$ is bounded. To determine the necessary range in $p$ where there is some restriction estimate, one usually
looks at the extension operator (the dual operator to ${\mathcal R}$)
$$
E b({\underline{x}}) \ := \ \int_0^1 e^{2\pi i [{\underline{x}} \cdot \gamma_r(t)]} b(t) \, dt 
$$
(so that a bound ${\mathcal R}: L^p \to L^q$ is equivalent to a bound $E : L^{q'} \to L^{p'}$) and applies it to
$b \equiv 1$ which is in every $L^{q'}$. Hence we want to determine the range of $p'$ such that
$E 1 ({\underline{x}})$ is in $L^{p'}$. 

Note that 
$$
E1({\underline{x}}) \ = \ {\widehat{\sigma}}({\underline{x}}) \ = \ \int_0^1 e^{2\pi i P_{\underline{x}}(t)} \, dt
$$
where $\sigma$ is (more or less) arclength measure on a compact piece of $\gamma_r$ and the phase 
$P(t) = P_{\underline{x}}(t) = x_1 t + x_2 t^2 + \cdots + x_d t^d$  is a real polynomial whose coefficients
give us a function in ${\underline{x}} \in {\mathbb R}^d$. The bound 
$| {\widehat{\sigma}}({\underline{x}})| \le C_d \|{\underline{x}}\|^{-1/d}$ follows from the classical van der
Corput estimates and the exponent $1/d$ is 
best possible {\it if} we measure the decay in terms of the isotropic norm $\|\cdot\|$ on ${\mathbb R}^d$.
This estimate alone does not give us the correct $L^{p'}$ range of integrability. In fact the precise range in $p'$
where $ {\widehat{\sigma}} \in L^{p'}$ is not a straightforward problem to resolve.
%and harmonic analysts 
%did not know whether Drury's result was sharp until recently. 

It turns out that the Lebesgue norm $\| {\widehat{\sigma}}\|_{L^{p'}({\mathbb R}^d)}$ arises as a constant 
(the singular integral) in
an asymptotic formula for the number of solutions to a system of diophantine equations. Determining a precise
count for various systems of diophantine equations became known as Tarry's problem which have been studied
since the 1920s. Ever since then, number theorists have been interested in finding the exact range of $L^{p'}$ integrability
for ${\widehat{\sigma}}$ with some partial progress given by Vinogradov and Hua in the 1930s. In the late 1970s,
Arkhipov, Chubarikov and Karatsuba \cite{ACK-1} (see also \cite{ACK} and \cite{ACK-2}) solved this problem with an elegant argument using a very clever application
of van der Corput estimates for oscillatory integrals. As a consequence we now know
that Drury's result is sharp.

So one might ask whether the argument can extend to the complex case. For the complex moment curve, the
extension operator applied to $b \equiv 1$ can be written as
$E 1 ({\underline{\xi}}_r)  =$
$$ 
 \int_{{\mathbb R}^2} e^{2\pi i [{\underline{\xi}}_r \cdot \Gamma(x,y)]} 
\phi(x,y) \, dx dy = 
\int_{{\mathbb R}^{2}} e^{2\pi i [ \sum_{j=1}^d \xi_j \phi_j (x,y) + \sum_{j=1}^d \eta_j \psi_j(x,y)]}
\phi(x,y) \, dx dy
$$
where (as above) the complex vector
${\underline{\xi}} = (\xi_1, \ldots, \xi_n) \in {\mathbb C}^d$ with $\xi_j = \gamma_j + i \eta_j$ can be viewed a real vector 
${\underline{\xi}}_r = (\gamma_1, \eta_1, \ldots, \gamma_d, \eta_d)$ in ${\mathbb R}^{2d}$. Here 
$\phi \in C^{\infty}_c({B})$ where ${B}$ is
the unit ball in ${\mathbb R}^2$.

Using the notation developed above, we have
$$
E1( {\underline{T}} \, {\underline{w}}) \ = \  \int_{{\mathbb C}} e(\langle {\underline{w}}, \gamma_c(z) \rangle) 
\phi(z) \, dz
$$
where $\phi \in C^{\infty}_c({\mathbb D})$. The
problem of determining the $L^{p'}$ integrability of $E1$ is the same as the one for $E1 \circ {\underline{T}}$. Note that
the phase $\langle {\underline{w}}, \gamma_c(z) \rangle = \sum_{j=1}^d w_j z^j$ is a general complex polynomial of degree $d$.
Hence we arrive at studying the two dimensional oscillatory integral
$$
I({\underline{w}}) \ := \ \int_{{\mathbb C}} e( P_{\underline{w}} (z)) \phi(z) \, dz
$$
where $P_{\underline{w}} (z) = w_1 z + w_2 z^2 + \cdots + w_d z^d$ is a complex polynomial.

Also relevant for Tarry's problem as well as for the fourier restriction problem is the case of
sparse polynomials $P^{\rm sparse}(z) = w_1 z^{k_1} + \cdots w_d z^{k_d}$
where $k_1 < k_2 < \cdots < k_d$ when $K := k_1 + \cdots + k_d < (0.5) k_d(k_d +1)$. The interest
is in determining when 
$$
I_{\rm sparse}({\underline{w}}) \ := \ \int_{{\mathbb C}} e( P^{\rm sparse}_{\underline{w}} (z)) \phi(z) \, dz
$$
belongs to the space  $L^q({\mathbb C}^d)$.

Proposition \ref{ACK-intro} determines the $L^q$ integrability of $I({\underline{w}})$ and $I_{\rm sparse}({\underline{w}})$.  

\begin{comment}
I am certain (although I have not written out the details) that $\| I \|_{L^q({\mathbb C}^d)}$ arises
in the asymptotic formula for counting the number of solutions of the classical system of diophantine
equations arising in the Vinogradov Mean Value Theorem but where now we count {\it gaussian} integer solutions
instead of {\it rational} integer solutions as is done in VMT. 
So again, knowing precisely which Lebesgue spaces $L^q({\mathbb C}^d)$
the oscillatory integral $I({\underline{w}})$ belongs to is not only relevant for the Fourier restriction problem but
it may be relevant in the theory of diophantine equations.
\end{comment}

\section{Proposition \ref{ACK-intro} -- some preliminaries}\label{complex-ack}

For ${\underline{w}} = (w_1, \ldots, w_d) \in {\mathbb C}^d$, let
$$
P(z) \ = \ P_{{\underline{w}}}(z) \ = \ w_d z^d + w_{d-1} z^{d-1} + \cdots + w_1 z
$$ 
be a complex polynomial of degree $d$ with coefficients ${\underline{w}}$.
Consider the oscillatory integral
$$
I({\underline{w}}) \ = \ \int_{\mathbb C} e(P_{{\underline{w}}}(z)) \, \phi(z) \, dz
$$
where $\phi \in C^{\infty}_c ({\mathbb D})$ satisfies $\phi(z) \equiv 1$ in ${\mathbb D}_{1/2}$.
For this particular oscillatory integral, we have changed notation from $I_{\phi}(P_{{\underline{w}}})$ to $I({\underline{w}})$ to highlight that we want to view
this oscillatory integral as a function of the coefficients of $P_{{\underline{w}}}(z)$. Similarly
we define $I_{\rm sparse}({\underline{w}})$ with respect to a polynomials with sparse powers as in the previous section.

Recall the statement of Proposition \ref{ACK-intro}: we have
\begin{equation}\label{finite}
{\mathcal I} \ := \ \int_{{\mathbb C}^d} |I({\underline{w}})|^q \, d{\underline{w}} \ < \ \infty
\end{equation}
if and only if $q > 0.05(d^2 + d) + 1$. Furthermore, $\|I_{\rm sparse}\|_{L^q} < \infty$ if and only if $q > K$.

We follow
the argument in \cite{ACK}, complexifying two key results. As in the real case, \eqref{finite}
relies on van der Corput bounds for oscillatory integrals as in Theorem \ref{H-main-osc} when $n=1$, together with a sublevel set bound for
$$
H({\underline{w}}) \ := \ \inf_{z\in {\rm supp}(\phi)} \max\bigl(|P'(z)|, |P''(z)/2|^{1/2}, \ldots, |P^{(d)}(z)/d!|^{1/d}).
$$
This is the case whether we are examining the $L^q$ norm of $I$ or $I_{\rm sparse}$. Here we will
only give the details establishing \eqref{finite} since the argument for $I_{\rm sparse}$ runs in the same
way as in the real case.

For convenience we restate the refined version of Theorem \ref{H-main-osc} in the case $n=1$ as
detailed in Section \ref{Thm 5.1}; see \eqref{poly-several-int} and \eqref{phi-dependence}.

\begin{proposition}\label{vc-complex} For any polynomial $Q \in {\mathbb C}[X]$ of degree $d$,
we have
$$
I_{\phi}(Q)| \ \le \ C_{d,\phi} \, H_Q^{-2}
$$
where 
\begin{equation}\label{phi-dependence-again-again}
C_{d,\phi} \ \le \ C_{d,k} \, [1 + J^{-k} \|\phi\|_{C^k}]
\end{equation}
for any large $k$.
In particular, for $I({\underline{w}})$ appearing in \eqref{finite}, we have
\begin{equation}\label{osc-int}
|I({\underline{w}})| \ \le \ C_{d,\phi} \, H({\underline{w}})^{-2} .
\end{equation}
\end{proposition}

The sublevel set bound for $H({\underline{w}})$ is the following. 

\begin{proposition}\label{H-sublevel} For any $Q\gg 1$, we have
\begin{equation}\label{H-sublevel-bound}
\bigl| \{ {\underline{w}} \in {\mathbb C}^d : H({\underline{w}}) \le Q \} \bigr| \ \lesssim_d \ Q^{2[0.5(d^2 + d) + 1]}.
\end{equation}
\end{proposition}

The proof of Proposition \ref{H-sublevel} is a straightforward complex extension of the real version due to 
Arkhipov, Chuburakov and Karatsuba, see \cite{ACK}.

\section{Proof of Proposition \ref{ACK-intro} - the sufficiency}\label{sufficiency}

We decompose 
$$
{\mathcal I} \ = \ 
\int_{\{ H({\underline{w}}) \le 1\}} |I({\underline{w}})|^q \, d{\underline{w}} \ + \ \sum_{r\ge 0} 
\int_{\{2^r < H({\underline{w}}) \le 2^{r+1}\}} |I({\underline{w}})|^q \, d{\underline{w}} .
$$
Of course we have the trivial bound $|I({\underline{w}})| \lesssim 1$ and this shows that the first term is at
most $|\{ H({\underline{w}}) \le 1\}|$ which in turn is finite by \eqref{H-sublevel-bound} in Proposition \ref{H-sublevel}. 
Let us denote the $r$th term in the sum defining the second term by ${\mathcal I}_r$; that is,
$$
{\mathcal I}_r  \ = \ \int_{\{2^r < H({\underline{w}}) \le 2^{r+1}\}} |I({\underline{w}})|^q \, d{\underline{w}}  \ =: \ \int_{E_r} |I({\underline{w}})|^q \,
d{\underline{w}}.
$$

Proposition \ref{vc-complex} implies that for ${\underline{w}} \in E_r$, $|I({\underline{w}})| \lesssim_d 2^{-2r}$ 
and so ${\mathcal I}_r \lesssim_d 2^{-2r q} |E_r|$.
By Proposition \ref{H-sublevel}, we have
$|E_r| \lesssim_d 2^{2r[0.5(d^2 + d) + 1]}$ and so
$$
\sum_{r\ge 0} \int_{E_r} |I({\underline{w}})|^q \, d{\underline{w}} \ \lesssim_d \ \sum_{r\ge 0} 2^{-2r [ q - 0.5 (d^2 +d) -1]}
$$
which is finite when $q > q_d$. This completes the proof of the sufficiency part of Proposition \ref{ACK-intro}.

\section{Proof of Proposition \ref{H-sublevel}}\label{H-proof}

Let 
$$
S_Q \ := \ \bigl\{ {\underline{w}} \in {\mathbb C}^d :  H({\underline{w}}) \, \le \, Q \bigr\}
$$
so that \eqref{H-sublevel-bound} says $|S_Q| \lesssim_d Q^{2[0.5(d^2 + d)+1]}$.

For $(r,s) \in {\mathbb Z}^2$ satsifying $-Q/2 \le r, s \le Q/2$, let $z_{r,s} = r/Q + i r/Q$ and define
$$
S_{Q}^{r,s} \ := \ \bigl\{ {\underline{w}} \in {\mathbb C}^d : |P_{{\underline{w}}}^{(k)}(z_{r,s}) | \le k! c_k  \,  Q^k, \ 1\le k \le d \, \bigr\}
$$
for some appropriate large constants $c_k$.
Our basic claim is
\begin{equation}\label{basic-claim}
S_Q \ \subseteq \ \bigcup_{-Q/2 \le r,s \le Q/2} S_Q^{r,s} .
\end{equation}

We write
$$
|S_Q^{r,s}| \ = \ \int_{S_Q^{r,s}} \, d {\underline{w}}
$$
and we will compute the integral defining $|S_Q^{r,s}|$ by making a certain change of variables.
Note that
$$
P^{(k)}(z)/k! \ = \  w_k \ + \ g_{k+1} w_{k+1} z \ + \ \cdots \ + \ g_{d+1} w_d z^{d-k}
$$
for some combinatorial numbers $g_j$. We make the linear change of variables ${\underline{y}} = T {\underline{w}}$
where for each $1\le k \le d$,
$$
y_k \ := \ w_k \ + \ g_{k+1} w_{k+1} z_{r,s} \ + \ \cdots \ + \ g_{d+1} w_d z_{r,s}^{d-k}.
$$
Note that ${\rm det}T = 1$ and $|y_k| \le c_k Q^k$ for ${\underline{w}} \in S_Q^{r,s}$. Hence
$$
|S_Q^{r,s}| \ = \ \int_{|y_1| \le c_1 Q} \cdots \int_{|y_d| \le c_d Q^d} \, d{\underline{y}} \ = \ C_d \, Q^{2[0.5(d^2 + d)]}
$$
and so the claim above implies $|S_Q| \lesssim_d Q^2 Q^{2[0.5(d^2 + d)]}$, completing the proof of Proposition \ref{H-sublevel}.

To prove the claim, fix ${\underline{w}} \in S_Q$ so that $H({\underline{w}}) \le Q$. By the definition of $H$, we see that there is
a $z = u + iv \in {\mathbb D}$ such that $|P^{(k)}(z)| \le Q^k$ for all $1\le k \le d$. Let $r := \lfloor u Q\rfloor$,
$s := \lfloor v Q \rfloor$ and set $w = z_{r,s} - z$. Note that $w = (r/Q - u) + i (s/Q - v)$ and so $|w| \le \sqrt{2} Q^{-1}$.

For $1\le k \le d$, we Taylor expand
$$
P^{(k)}(z_{r,s}) \ = \ P^{(k)}(z) + P^{(k+1)}(z) w + \frac{1}{2!} P^{(k+2)}(z) w^2 \ + \ \cdots \ + \ \frac{1}{(d-k)!} P^{(n)}(z) w^{d-k}
$$
so that
$$
|P^{(k)}(z_{r,s})| \ \lesssim \  \sum_{\ell=0}^{d-k} Q^{k + \ell} |w|^{\ell} \ \le \ k! c_k \, Q^k
$$
for appropriate constants $c_k$.
This implies that ${\underline{w}} \in S_Q^{r,s}$, completing the proof of the claim and hence the proof of the proposition.

\section{Proof of Proposition \ref{ACK-intro} - the necessity}\label{necessity}

Here we follow \cite{ACK} in the real case. We bound
$$
{\mathcal I} \ \ge \ \int_R |I({\underline{x}})|^q \, d{\underline{x}} 
$$
where $R \subset {\mathbb C}^d$ is a region where we will be able to estimate $I({\underline{x}})$ from below.
The region $R$ will be a disjoint union of subregions $R_m, m\ge m_0$ for some large integer $m_0 = m_0(d)$. To define $R_m$,
we fix a lacunary sequence $Q_m = A^m$ where $A = A_d > 2$ and define the planar sets
$$
E_m \ := \ \bigl\{z = x + iy \in {\mathbb C} : Q_m^d  \le x \le (2 Q_m)^d, \ Q_m^d \le y \le (2Q_m)^d \, \big\}.
$$
These are disjoint since $A > 2$. Writing ${\underline{x}} = ({\underline{x}'}, x_d) \in {\mathbb C}^d$, the sets
$R_m$ will be of the form
$$
R_m \ := \ \{{\underline{x}} \in {\mathbb C}^d : x_d \in E_m, \ \, {\underline{x}'} \in R'_m(x_d) \}
$$
where for each $x_d \in E_m$, the slice
$R'_m(x_d) \subset {\mathbb C}^{d-1}$ will be specified momentarily. Note that the sets $\{R_m\}$ are pairwise disjoint
whatever our choice for $R'_m(x_d)$. 

Set $w_0 = 0.25 (1+i)$ and define, for each $m \gg 1$, 
$$
{\mathcal P}_m \ := \ \bigl\{ (r,s) \in {\mathbb N}^2 : 1 \le r, s \le Q_m/10 \bigr\}.
$$
For each $(r,s) \in {\mathcal P}_m$, let $z_{r,s} = w_0 + (r/Q_m + i s/Q_m) \in {\mathbb D}_{1/2}$.
Note that if $(r,s) \not= (r'.s')$, then $|z_{r,s} - z_{r',s'}| \ge Q_m^{-1}$.
For each $x_d \in E_m$, we define
$$
R'_m(x_d) \ := \ \bigcup_{(r,s) \in {\mathcal P}_m} {R'}_m^{r,s}(x_d)
$$
where ${R'}_m^{r,s}(x_d) = T_{x_d}^{r,s} {\mathcal B}$ and
$$
{\mathcal B} \ := \ 
\bigl\{ {\underline{y}'} = (y_1, \ldots, y_{d-1}) \in {\mathbb C}^{d-1} : |y_j| \le (c_1 Q_m)^j, \, 1\le j\le d-1  \, \bigr\},
$$
$c_1$ a small constant depending on $d$,
and $T_{x_d}^{r,s}$ an affine transformation $T_{x_d}^{r,s} {\underline{y}'} = {\underline{x}'}$ defined by the relationship
$$
x_d z^d + \cdots + x_1 z \ = \ x_d (z - z_{r,s})^d + y_{d-1} (z - z_{r,s})^{d-1} + \cdots + y_1 (z - z_{r,s});
$$
in other words, 
$$
x_{d-1} \ = \ y_{d-1} - d x_d z_{r,s}, \  x_{d-2} \ = \ y_{d-2} - (d-1) y_{d-1} z_{r,s} + {d\choose{2}} x_d \, z_{r,s}^2, \ etc...
$$
and generally, for $1\le k \le d-1$,
\begin{equation}\label{x-y}
x_{d-k} \  = \ y_{d-k} \ + \ \sum_{\ell = 1}^k (-1)^{\ell} {{d-k+\ell}\choose{\ell}} y_{d-k+\ell} \,  z_{r,s}^{\ell}.
\end{equation}
Here we write $x_d = y_d$ for convenience.
Hence the linear part of $T_{x_d}^{r,s}$ is upper triangular with $1$'s down the diagonal and so ${\rm det} T_{x_d}^{r,s} = 1$. 
In particular $T_{x_d}^{r,s}$ is a bijection.

For each ${\underline{x}'} \in {R'}_m^{r,s}(x_d)$, we see that $x_{d-1} \in {\mathbb D}_{(c_1 Q_m)^{d-1}}(-d x_d z_{r,s})$
and these discs are disjoint as we vary over $(r,s) \in {\mathcal P}_m$. Indeed if $(r,s), (r',s') \in {\mathcal P}_m$ are two distinct elements,
then $|z_{r,s} - z_{r',s'}| \ge Q_{m}^{-1}$ and so the centres of these discs
$$
|d x_d z_{r,s} - d x_d z_{r'.s'}| \ = \ d |x_d| |z_{r,s} - z_{r',s'}| \ \ge \ Q_m^{d} Q_m^{-1} \ = \ Q_m^{d-1}
$$
are separated by more than twice the radius $(c_1 Q_m)^{d-1}$ if $c_1$ is small. This shows that the
sets $\{{R'}_m^{r,s}(x_d)\}_{(r,s)\in {\mathcal P}_m}$ are pairwise disjoint for each $m$ and each $x_d \in E_m$.

Hence
$$
{\mathcal I} \ \ge \ \sum_{m\ge m_0} \int_{R_m} |I({\underline{x}})|^q \, d{\underline{x}} \ = \ 
\sum_{m\ge m_0} \sum_{(r,s)\in {\mathcal P}_m} \int_{x_d \in E_m}
 \int_{{\underline{x}'} \in {R'}_m^{r,s}(x_d)} |I({\underline{x}})|^q \, 
d{\underline{x}}.
$$

We perform the change of variables ${\underline{x}'} = T_{x_d}^{r,s} {\underline{y}'}$ in the inner integral to write
$$
\int_{{R'}_m^{r,s}(x_d)} |I({\underline{x}})|^q \, d{\underline{x}'} \ = \ 
\int_{|y_{d-1}|\le (c_1 Q_m)^{d-1}} \cdots \int_{|y_1|\le c_1 Q_m} |II_{r,s}({\underline{y}})|^q \, d{\underline{y}'}
$$
where
$$
II_{r,s}({\underline{y}}) \ = \ \int_{{\mathbb C}} e(y_d(z-z_{r,s})^d + \cdots + y_1(z-z_{r,s})) \, \phi(z) \, dz.
$$
Recall we are writing $y_d = x_d$ for convenience.

We make the change of variables $z \to z - z_{r,s}$ in $II_{r,s}$ and write
$$
II_{r,s}({\underline{y}}) \ = \ \int_{{\mathbb C}} e(y_d z^d + \cdots + y_1z) \, \phi(z + z_{r,s}) \, dz \ =: \ 
\int_{\mathbb C} e(P(z; {\underline{y}})) \, \phi(z+z_{r,s}) \, dz.
$$
Fix a nonnegative $\psi \in C^{\infty}_c({\mathbb D}_2)$ with $\psi \equiv 1$ on ${\mathbb D}$ and define
$\psi_m(z) := \psi((Q_m/a) z)$ where $a = a_d$ is a large constant. We split 
$II_{r,s}({\underline{y}}) = II_{r,s}^1({\underline{y}}) + II_{r,s}^2({\underline{y}})$ where
$$
II_{r,s}^1({\underline{y}}) \ = \ \int_{\mathbb C} e(P(z;{\underline{y}})) \, \phi(z + z_{r,s}) \psi_m(z) \, dz
$$
and
$$
II_{r,s}^2({\underline{y}}) \ = \ \int_{\mathbb C} e(P(z;{\underline{y}})) \, \phi(z + z_{r,s}) (1-\psi_m(z)) \, dz .
$$
Our goal is to prove the following:
\begin{equation}\label{y} 
{\rm for} \  {\underline{y}} \ {\rm satisfying} \ \ \ 
y_d \in E_m, \ \ |y_{d-1}| \le (c_1 Q_m)^{d-1}, \ \ldots, \  |y_1| \le c_1 Q_m,
\end{equation}
\begin{equation}\label{aim}
|II_{r,s}^1({\underline{y}})| \ge \ B  \, Q_m^{-2} \ \ \ {\rm and} \ \ \ |II_{r,s}^2({\underline{y}})| \ \le \ (B/2) \,  Q_m^{-2}
\end{equation}
for some $B = B_d > 0$.

The measure of those ${\underline{y}}$ satisfying \eqref{y} is $C Q_m^{2[d + (d-1) + \cdots + 1]}$
or $C Q_m^{2[ 0.5(d^2 + d)]}$. So if \eqref{aim} holds, then
$$
{\mathcal I} \ \gtrsim \ \sum_{m\ge m_0} \sum_{(r,s) \in {\mathcal P}_m} Q_m^{-2q}  Q_m^{2[0.5(d^2 + d)]} \ \gtrsim \
\sum_{m\ge m_0} Q_m^{-2 [q - 0.5(d^2 + d) - 1]}
$$
which shows that ${\mathcal I} = \infty$ if $q \le q_d$, establishing the necessity of Proposition \ref{ACK-intro}.

We first establish the bound for $II_{r,s}^2$ in \eqref{aim}. Set 
$$
\rho(z) \ = \ \rho_{m,r,s}(z) \ := \ \phi(z + z_{r,s}) (1 - \psi_m(z)).
$$
For $z \in {\rm supp}(\rho)$ and ${\underline{y}}$ satisfying \eqref{y}, we have 
$$
|P^{(d-1)}(z)| \ = \ |d! y_d z + (d-1)! y_{d-1}| \ \ge \ d! \bigl[ a |y_d| Q_m^{-1} -  (c_1 Q_m)^{d-1} \bigr] \ \ge a \, Q_m^{d-1}
$$
since $a$ is large and $c_1$ is small. Therefore
\begin{equation}\label{H-aim}
H_P \ = \ \inf_{z \in {\rm supp}(\phi)} \max(|P^{(d)}(z)|^{1/d}, \ldots, |P'(z)|) \ \ge \ a^{1/(d-1)} \, Q_m
\end{equation}
and and so we are in a position
to apply Proposition \ref{vc-complex} and conclude
$$
|II_{r,s}^2({\underline{y}})| \ \le \ C_{d,\rho} \, H_d^{-2} \ \le \ C_{d,\rho} \, a^{-2/(d-1)} \, Q_m^{-2} 
$$
for ${\underline{y}}$ satisfying \eqref{y}.
Since $\rho$ depends on $m$ and $r,s$, we need to understand how $C_{d,\rho}$ depends on $\rho$. 
We apply \eqref{phi-dependence-again-again} from Proposition \ref{vc-complex} to conclude
$$
C_{d, \rho} \ \le \ C_d \bigl[ 1 + J^{-k} \|\rho\|_{C^k}\bigr]
$$
for some large $k = k(d)$. Here $J \ge a^{1/(d-1)} Q_m$ and $\|\rho\|_{C^k} \lesssim_d [Q_m/a]^k$. Therefore
$C_{d,\rho} \le C_d \, a^{-k d/(d-1)}$ implying
\begin{equation}\label{II2-aim}
|II_{r,s}^2({\underline{y}})| \ \le \ C_d \, a^{-2/(d-1)} \, Q_m^{-2} \ \ {\rm for} \ {\underline{y}} \ {\rm satisfying} \ \eqref{y}.
\end{equation}
 We note here that we could have considered $|P^{(d)}(z)| \equiv d! |x_d| \ge d! Q_m^d$ and bounded $H_P$ in \eqref{H-aim}
below by $(d!)^{1/d} \, Q_m$.  But $(d!)^{1/d} \le d$ and we would not have gained the very 
large constant $a^{1/(d-1)}$ in \eqref{H-aim}
and hence the small constant in \eqref{II2-aim}.

We now turn to $II_{r,s}^1({\underline{y}})$ when ${\underline{y}}$ satisfies \eqref{y}. First we note
that $\phi(z+ z_{r,s}) \psi_m(z) = \psi_m(z)$ since when $\psi_m(z) \not= 0$, then $|z| \le 2 a Q_m^{-1}$ and 
so $z + z_{r,s} \in {\mathbb D}_{1/2}$, implying $\phi(z + z_{r,s}) = 1$.
We can therefore write
$$
II_{r,s}^1({\underline{y}}) \ = \ \int_{\mathbb C} e(y_d z^d) \, \psi_m(z) \, dz \ + \ E
$$
where 
$$
|E| =  \Big|\int_{\mathbb C} \bigl[e( P(z; {\underline{y}})) - e(y_d z^d) \bigr]  \psi_m(z) \, dz\Bigr| \ \le \
\int_{\{|z| \le 2 a Q_m^{-1}\}} |P(z; {\underline{y}}) - y_d z^d | \, dz  .
$$
For $|z| \le 2 a Q_m^{-1}$ we have
$$
|P(z;{\underline{y}}) - y_d z^d| \ \le \ \sum_{j=1}^{d-1} |y_j z^j| \ \le \
\sum_{j=1}^{d-1} (c_1 Q_m)^j (2 a Q_{m}^{-1})^j  \ \le \ 2 d \, c_1 a
$$
if we chose $a$ and $c_1$ so that $2 c_1 a \le 1$.
Therefore $|E| \lesssim_d (c_1 a) (2a Q_m^{-1})^2 \lesssim_d a^3 c_1 Q_m^{-2}$. 
This implies
\begin{equation}\label{II1-error}
|II_{r,s}^1({\underline{y}})| \ \ge \ \Bigl| \int_{\mathbb C} e(y_d z^d)  \, \psi_m(z) \, dz \Bigr| \ - \ C_d a^3 c_1 \, Q_m^{-2},
\end{equation}
leaving us to analyse the main term
$$
II_{r,s}^{\rm main}(y_d) \ := \ \int_{\mathbb C} e ( y_d z^d ) \, \psi_m(z) \, dz  \ = \ \int_{\mathbb C} e( y_d z^d) \, \psi((Q_m/a) z) \, dz.
$$

Finally we write
$$
II_{r,s}^{\rm main}(y_d) \ = \ \int_{\mathbb C} e (y_d z^d) \, dz \ + \ \int_{\mathbb C} e(y_d z^d) \bigl[ 1 - \psi((Q_m/a) z) \bigr] \, dz
\ =: \ A + B
$$
where these two improper integrals are interpreted as limits of truncated integrals. For example, for $B$
we mean $B = \lim_{R\to\infty} B_R$ where
$$
B_R \ := \ \int_{\mathbb C} e( y_d z^d) \, \bigl[ 1 - \psi((Q_m/a) z) \bigr] \, \Psi(R^{-1} z) \, dz
$$
and $\Psi \in C^{\infty}_c ({\mathbb D})$ with $\Psi(z) \equiv 1$ on ${\mathbb D}_{1/2}$. It is fairly straightforward
to see that these limits exist (in fact, the analysis below will show this as a consequence). To bound $|B|$ from above,
it suffices to bound $|B_R|$ from above, uniformly in $R$. 

We change variable $w = R^{-1} z$ to write
$$
B_R =  R^2 \, \int_{\mathbb C} e( y_d R^d w^d) \, \Psi(w) \bigl[ 1 - \psi((Q_m R/a) w) \bigr] \, dw  =: 
\int_{\mathbb C} e( y_d R^d w^d) \, \rho(w) \, dw.
$$
We will apply Proposition \ref{vc-complex} to $B_R$ with $P(w) = y_d R^d w^d$. Note that
$$
H_P \ \ge \ \inf_{w \in {\rm supp}(\rho)} |P'(w)| \ \ge \ d |y_d| R^d (a/(R Q_m))^{d-1} \ \ge \ d a^{d-1} R Q_m .
$$
Hence
$$
|B_R| \ \ \lesssim_d  \ \ a^{-2(d-1)} \bigl[ 1 + J^{-k} \|\rho\|_{C^k}\bigr] \ Q_m^{-2}
$$
where $\|\rho\|_{C^k} \lesssim (R Q_m/a)^k$ and 
$J = $
$$
\inf_{w \in {\rm supp}(\rho)} \max (|P^{(d)}(w)|^{1/d}, \ldots, |P''(w)|^{1/2} ) \ge  \sqrt{d(d-1)} (RQ_m)^{d/2} ((a/R Q_m))^{(d-2)/2},
$$
implying
$$
J \ \ge \ \sqrt{d(d-1)} a^{(d-2)/2} R Q_m.
$$
Therefore $|B_R| \lesssim_d  a^{-2(d-1)} \, Q_m^{-2}$ and so
\begin{equation}\label{B}
|B| \ \lesssim_d \ a^{-2(d-1)} \ Q_m^{-2}.
\end{equation}

It remains to treat the first integral
$$
A \ = \ \int_{\mathbb C} e(y_d z^d) \, dz
$$
of the main term $II_{r,s}^{\rm main}({\underline{y}})$ when $y_d \in E_m$. Fix $y_d = r e^{i\theta} \in E_m$ 
so that $r = |y_d|$.

We make the change of variables $w = u z$ where $u = r^{1/d} e^{i \theta/d}$ so that $w^d = y_d z^d$.
The jacobian of this change of variables is $r^{2/d} = |y_d|^{2/d}$ and so
$$
A \ = \  c_d \, \frac{1}{|y_d|^{2/d}} \ \ \ {\rm where} \ \ \ c_d \ := \ \int_{\mathbb C} e(w^d) \, dw
$$
is nonzero. This implies $|A| \ge |c_d| Q_m^{-2}$. This bound, together with the bound for $B$ in \eqref{B}, implies
\begin{equation}\label{main-II}
|II_{r,s}^{\rm main}({\underline{y}})| \ \ge \ (|c_d|/2) \, Q_m^{-2}
\end{equation}
if we choose the constant $a = a_d$ large enough. 

Putting \eqref{main-II} together with \eqref{II2-aim} and \eqref{II1-error} (choosing first the constant $a$ large and then
choosing the $c_1$ small so that $c_1 a^3 \ll 1$) establishes the desired bound \eqref{aim}, completing the necessity
part and hence the proof Proposition \ref{ACK-intro}.

\section{Appendix: proofs of Propositions \ref{vc-sublevel-first-derivative} and \ref{vc-complex-sublevel}}\label{appendix}

In this appendix we give the proofs of Proposition \ref{vc-sublevel-first-derivative} and Proposition \ref{vc-complex-sublevel}
which rely on Lemma \ref{hensel-L}.

First let us see how Lemma \ref{hensel-L} implies Proposition \ref{vc-sublevel-first-derivative}. 
In this case we have $|f'(z)|\ge 1$ on ${\mathbb D}$ and we will apply
Lemma \ref{hensel-L} with $L=1$ to
$$
\phi(z) \ := \ f(z) \ \ \ {\rm and} \ \ \  z_0 \in \{z\in {\mathbb D}: |f(z)| \le \epsilon\}.
$$
Note that
$\delta = |\phi(z_0) \phi'(z_0)^{-2}| \le \epsilon$ and since $\delta M \le \epsilon M < 1/64$,
we see that there exists a $z_{*} \in {\mathbb D}_{5/4}$ 
with $f(z_{*}) = 0$ and $|z_0 - z_{*}| \le 2 |\phi(z_0)\phi'(z_0)^{-1}| \le 2 \epsilon$. Hence
\begin{equation}\label{n=1}
\{z\in {\mathbb D}: |f(z)| \le \epsilon\} \ \subseteq \ \bigcup_{z_{*} \in {\mathcal Z}_0} {\mathbb D}_{2 \epsilon}(z_{*})
\end{equation}
where ${\mathcal Z}_0 := \{ z \in {\mathbb D}_{5/4} : f(z) = 0 \}$. This completes the proof
of Proposition \ref{vc-sublevel-first-derivative}.

We turn to the proof of Proposition \ref{vc-complex-sublevel} and we begin with the case $k=2$ so that $|f''(z)|\ge 1$ 
on ${\mathbb D}$. We
split the sublevel set $S := \{z\in {\mathbb D}: |f(z)| \le \epsilon\} = S_1 \cup S_2$ into two sets where
$$
S_1 \ := \ \{z \in S : |f'(z)| \le A \sqrt{\epsilon} \} \ \ {\rm and} \ \ 
S_2 \ := \ \{z\in S: |f'(z)| > A \sqrt{\epsilon}\}
$$
where $A = c_0 \sqrt{M}$ for some large but absolute $c_0\gg 1$. For $S_1$, we apply 
Lemma \ref{hensel-L} with $L=1$ to 
$\phi(z) := f'(z)$ and $z_0 \in S_1$. Note that
$\delta = |\phi(z_0) \phi'(z_0)^{-2}| \le A \sqrt{\epsilon}$ and $\delta M \le c_0 M^{3/2}\sqrt{\epsilon}  < 1/64$
since $\epsilon \ll M^{-3}$. Hence
we see that there exists a $z_{*}\in {\mathbb D}_{5/4}$ 
with $f'(z_{*}) = 0$ and $|z_0 - z_{*}| \le 2 |\phi(z_0)\phi'(z_0)^{-1}| \le 2 A \sqrt{\epsilon}$ and so
\begin{equation}\label{n=2,S1}
S_1  \ \subseteq \ \bigcup_{z_{*} \in {\mathcal Z}_1} {\mathbb D}_{2 A \sqrt{\epsilon}}(z_{*})
\end{equation}
where ${\mathcal Z}_1 := \{ z \in {\mathbb D}_{5/4} : f'(z) = 0 \}$. 

For $S_2$, we apply
Lemma \ref{hensel-L} with $L=1$ to $\phi(z) := f(z)$ and $z_0 \in S_2$.
Note that
$\delta = |\phi(z_0) \phi'(z_0)^{-2}| \le A^{-2}$ and $\delta M \le (c_0)^{-1}  < 1/64$
since $c_0 \gg 1$ is large. Hence
we see that there exists a $z_{*} \in {\mathbb D}_{5/4}$ 
with $f(z_{*}) = 0$ and $|z_0 - z_{*}| \le 2 |\phi(z_0)\phi'(z_0)^{-1}| \le 2 A^{-1} \sqrt{\epsilon}$ and so
\begin{equation}\label{n=2,S2}
S_2  \ \subseteq \ \bigcup_{z_{*} \in {\mathcal Z}_0} {\mathbb D}_{\sqrt{\epsilon}}(z_{*}),
\end{equation}
completing the proof of Proposition \ref{vc-complex-sublevel} when $k=2$.

For $k\ge 3$, we
split the sublevel set $S := \{z\in {\mathbb D}: |f(z)| \le \epsilon\} = S_1 \cup \cdots \cup S_k$ into $k$ sets where
$$
S_1 \ := \ \{z \in S : |f^{(k-1)}(z)| \le A \epsilon^{1/k} \}, \ \ {\rm and}
$$
$$
S_2 \ := \ \{z\in S: |f^{(k-1)}(z)| > A \epsilon^{1/k}, \ |f^{(k-2)}(z)|\le d_2 
\epsilon^{1/k}|f^{(k-1)}(z)| \}.
$$
The remaining sets $S_j,  \,3\le j \le n,$ are defined with respect to a series of inequalities 
$$
{\mathcal I}_j(z)  :  \ \ \ \ |f^{(k-j+1)}(z)| \ > \ 
d_{j-1} \epsilon^{1/k} |f^{(k-j+2)}(z)| \ > \ \cdots \ > 
$$
$$
(d_{j-1}\cdots d_2) \epsilon^{(j-2)/k}
|f^{(k-1)}(z)| >  A (d_{j-1} \cdots d_2) \epsilon^{(j-1)/k} 
$$
being satisfied. For $3\le j \le k-1$, we define
$$
S_j  \ := \ \bigl\{z\in S: {\mathcal I}_j(z) \ {\rm is \ satisfied}, \ |f^{(k-j)}(z)| \le d_j \epsilon^{1/k} |f^{(k-j+1)}(z)| \bigr\}.
$$
Here $d_2 = M^{-1/k}$ and $d_j = \eta d_{j-1}, 3\le j \le k-1$ for an appropriate small $\eta = \eta_k > 0$. Also $A = c_0 M^{(k-1)/k}$
where $c_0 \gg 1$ will be chosen large enough, depending on $\eta$ (and $k$). Finally we have
$$
S_k \ = \ \bigl\{ z \in S : {\mathcal I}_k \ {\rm is \ satisfied} \bigr\}.
$$
Of course we see that the definition of $S_2$ is incorporated in the definitions of $S_j$ above when $j=2$ if we 
interpret the product $d_{j-1} \cdots d_2$ as the empty product when $j=2$ and so equal to 1. 

We will apply the Lemma \ref{hensel-L} with various choices of $L$ to various derivatives $\phi(z) = f^{(j)}(z)$. 
For $z \in {\mathbb D}_{7/4}$, Cauchy's integral
formula gives us
$$
\phi(z) \ = \ f^{j)}(z) \ = \frac{j!}{2\pi i} \int_{C_{1/4}(z)} \frac{f(w)}{(w - z)^{j+1}} \, dw
$$
where $C_r(z) = \{w : |w-z| = r\}$ denotes the circle of radius $r$ centred at $z$. Hence
$M_j := M_{\phi} = \sup_{z\in {\mathbb D}_{7/4}}  |f^{(j)}(z)| \le j! 4^j M$.

For $z_0 \in S_1$, we apply Lemma \ref{hensel-L} with $L=1$ to $\phi(z) = f^{(k-1)}(z)$ and $M_{k-1} \le 4^{k-1} (k-1)! M$.
Here 
$$
\delta \ = \ |\phi(z_0) \phi'(z_0)^{-2}| \ = \ |f^{(k-1)}(z_0) f^{(k)}(z_0)^{-2}| \ \le \ A \epsilon^{1/k}
$$ 
and so 
$$
\delta M_{k-1} \ \lesssim_k \ A M \epsilon^{1/k} \ \ll_k \ M^{(k-1)/k} M \epsilon^{1/k} \ \ll_k \ 1
$$
since $\epsilon \ll_k M^{-(2k -1)}$. Since $2 |\phi(z_0) \phi'(z_0)^{-1}| \le 2 A \epsilon^{1/k}$, we find
a zero $z_{*}$ of $\phi = f^{(k-1)}$ in ${\mathbb D}_{5/4}$ such that $|z_{*} - z_0| \le  2A \epsilon^{1/k}$ and hence
\begin{equation}\label{S1}
S_1 \ \subseteq \ \bigcup_{z_{*}\in {\mathcal Z}_{k-1}} {\mathbb D}_{2 A \epsilon^{1/k}} (z_{*})
\end{equation}
where ${\mathcal Z}_{k-1} := \{ z\in {\mathbb D}_{5/4} : f^{(k-1)}(z) = 0 \}$. In general we define
$$
{\mathcal Z}_{j} \ := \ \bigl\{ z\in {\mathbb D}_{5/4} : f^{(j)}(z) = 0 \bigr\}.
$$

Now for $z_0 \in S_2$, we apply Lemma \ref{hensel-L} with $L=1$ to $\phi(z) = f^{(k-2)}(z)$ and $M_{k-2} \lesssim_k M$. Here
$$
\delta \ = \ |\phi(z_0) \phi'(z_0)^{-2}| \ = \ |f^{(k-2)}(z_0) f^{(k-1)}(z_0)^{-2}| \ \le \ d_2 A^{-1} \ = \ c_0^{-1} M^{-1}
$$ 
and so 
$$
\delta M_{k-2} \ \lesssim_k \ c_0^{-1}  M^{-1} M  \ \ll_k \ 1
$$
since $c_0 \gg 1$.  Since $2 |\phi(z_0) \phi'(z_0)^{-1}| \le 2 d_2 \epsilon^{1/k}$, we find
a zero $z_{*}$ of $\phi = f^{(k-2)}$ in ${\mathbb D}_{5/4}$ such that $|z_{*} - z_0| \le  2 d_2 \epsilon^{1/k}$ and hence
\begin{equation}\label{S2}
S_2 \ \subseteq \ \bigcup_{z_{*}\in {\mathcal Z}_{k-2}} {\mathbb D}_{2 d_2 \epsilon^{1/k}} (z_{*}).
\end{equation}

Next consider $z_0 \in S_j$ for $3\le j \le k$. Here we will apply Lemma \ref{hensel-L} to $\phi(z) = f^{(k-j)}(z)$
with $L= k-1$
and $M_{k-j} \lesssim_k M$. Hence
$$
\delta \ = \ |\phi(z_0) \phi'(z_0)^{-1}\phi^{(L)}(z_0)^{-1}| \ = \ |f^{(k-j)}(z_0) f^{(k-j+1)}(z_0)^{-1}f^{(k-1)}(z_0)^{-1}| 
$$
$$
\ \le \ d_j \epsilon^{1/k}  |f^{(k-1)}(z_0)^{-1}| \ \le \ d_j \epsilon^{1/k} A^{-1} \epsilon^{-1/k} \ = \ \eta^{j-2} M^{-1/k} c_0^{-1}
M^{-(k-1)/k}
$$ 
and so 
$$
\delta M_{k-j} \ \lesssim_k \ \eta^{j-2} c_0^{-1}  M^{-1} M  \ \ll_k \ 1.
$$
Also 
$$
\delta_1 = |\phi(z_0) \phi''(z_0) \phi'(z_0)^{-2}| = |f^{(k-j)}(z_0) f^{(k-j+2)}(z_0) f^{(k-j+1)}(z_0)^{-2}| \le d_j/d_{j-1} = \eta
$$
and so $\delta M_{k-j}, \, \delta_1 \ll_k 1$. These bounds need to modified in the case $j=k$; here 
$$
\delta \ = \ |f(z_0) f'(z_0)^{-1} f^{(k-1)}(z_0)^{-1}| \ \le \ \epsilon \, \frac{1}{A D_k \epsilon^{(k-1)/k}}  
\frac{\epsilon^{(k-2)/k}}{A \epsilon^{(k-1)/k}} = \frac{1}{A^2 D_k}
$$
where $D_k = d_{k-1} \cdots d_2$. Hence (for some $N = N_k \in {\mathbb N}$ which can be computed)
$$
\delta M \ \le \ c_0^{-1} \eta^{-N_k} M^{-2(k-1)/k} M^{(k-2)/k} M \ = \ c_0^{-1} \eta^{-N_k} \ \ll_k \ 1
$$
since we choose $c_0$ large depending on $\eta$. Also
$$
\delta_1 = |f(z_0) f''(z_0) f'(z_0)^{-2}| \le \epsilon  \, \frac{1}{d_{k-1} \epsilon^{1/k}} \frac{1}{ A D_k \epsilon^{(k-1)/k}} \ = \
c_0^{-1} \eta^{-N_k} \ \ll_k \ 1.
$$
For $2\le r \le j-2$, we have
$$
\delta_r \ = \ |f^{(k-j)}(z_0) f^{k-j+1}(z_0)^{-1} f^{(k-j+r+1)}(z_0) f^{(k-j+r)}(z_0)^{-1}| 
$$
$$
\le \ d_j \, \epsilon^{1/k} \
\frac{\epsilon^{(r-1)/k} (d_{j-1} \cdots d_{j-r+1})}{\epsilon^{r/k} (d_{j-1}\cdots d_{j-r})} \ = \ 
d_j d_{j-r}^{-1} \ = \ \eta^r  \ \lesssim_k \ 1.
$$
Finally we see that for $j\le k-1$,
$$
|\phi(z_0) \phi'(z_0)^{-1}| \ = \ |f^{(k-j)}(z_0) f^{(k-j+1)}(z_0)^{-1}| \ \le \ d_j \epsilon^{1/k} \ \le \ 1/8
$$
and so Lemma \ref{hensel-L} implies there is a zero $z_{*}$ of $f^{(k-j)}$ in ${\mathbb D}_{5/4}$ such
that $|z _{*}- z_0| \le 2 |\phi(z_0)\phi'(z_0)^{-1}| \le 2 d_j \epsilon^{1/k}$. Hence for $3\le j < k$,
\begin{equation}\label{Sk}
S_j \ \subseteq \ \bigcup_{z_{*}\in {\mathcal Z}_{k-j}} {\mathbb D}_{2 d_j \epsilon^{1/k}} (z_{*}).
\end{equation}
For $j=k$,
$$
|\phi(z_0) \phi'(z_0)^{-1}| \ = \ |f(z_0) f'(z_0)^{-1}| \ \le \ \epsilon \, \frac{1}{A D_k \epsilon^{(k-1)/k}}  \ = \
c_0^{-1} \eta^{-N_k} M^{-1/k} \, \epsilon^{1/k}
$$
and so Lemma \ref{hensel-L} implies there is a zero $z_{*}$ of $f$ in ${\mathbb D}_{5/4}$ such
that $|z_{*} - z_0| \le 2 |f(z_0)f'(z_0)^{-1}| \lesssim_k  \epsilon^{1/k}$. Hence 
\begin{equation}\label{Sn}
S_k \ \subseteq \ \bigcup_{z_{*}\in {\mathcal Z}_{0}} {\mathbb D}_{b_k \epsilon^{1/k}} (z_{*}).
\end{equation}

Hence \eqref{n=1}, \eqref{n=2,S1}, \eqref{n=2,S2},\eqref{S1}, \eqref{S2}, \eqref{Sk} and \eqref{Sn} imply the desired bound
\eqref{sublevel-vc}, completing the proof of Proposition \ref{vc-complex-sublevel}.
%$$
%|\{z \in {\mathbb D}: |f(z)| \le \epsilon\}| \ = \ |S| \ \lesssim_n \ M^{2(n-1)/n} \log(M) \, \epsilon^{2/n}.
%$$

\end{document}